\documentclass[12pt]{article}
\usepackage{amsmath,amssymb,amsfonts,amsthm}
\usepackage{yhmath}
\usepackage{eucal,oldgerm}
\usepackage[dvips]{graphics}
\usepackage{pstricks,pst-plot,pst-node,pst-coil}
\newpsobject{showgrid}{psgrid}{subgriddiv=1,griddots=5,gridlabels=0pt}
\usepackage[arrow,matrix,tips,2cell]{xy}

\newcommand{\rang}{\right\rangle}

\newcommand{\zz}{{\mathfrak{z}}}

\newcommand{\com}{{\mathbb C}}

\newcommand{\bC}{\mathsf{C}}

\newcommand{\bw}{\mathsf{w}}

\newcommand{\PP}{{\mathbf{P}}}
\newcommand{\bF}{\mathcal{F}}
\newcommand{\bD}{\mathcal{D}}
\newcommand\FF{\mathbb F}
\newcommand\ZZ{\mathsf Z}

\newcommand{\oh}{{\mathcal{O}}}
\newcommand{\T}{{\mathbf{T}}}

\newcommand{\LL}{{\mathbb{L}}}
\newcommand{\CC}{{\widehat{C}}}
\newcommand{\bbullet}{{\widehat{\bullet}}}

\newcommand{\C}{\mathbb{C}}
\newcommand{\Q}{\mathbb{Q}}
\newcommand{\Z}{\mathbb{Z}}

\newcommand{\cO}{\mathcal{O}}

\newcommand{\Pp}{{\mathbf{P}^1}}

\newcommand{\rarr}{\rightarrow}

\newcommand{\bA}{\mathcal{A}}
\newcommand{\bZ}{\mathsf{Z}}
\newcommand{\oM}{\overline{M}}

\newcommand{\uncrazeright}{\, {\unrhd}^{\hspace{-2pt}*}\, }
\newcommand{\crazeright}{\, {\rhd}^{\hspace{-2pt}*}\, }

\DeclareMathOperator{\Hilb}{Hilb}

\newcommand{\p}{{\mathsf{p}}}

\newtheorem{Theorem}{Theorem}
\newtheorem{Lemma}{Lemma}
\newtheorem{Corollary}{Corollary}

\newtheorem{Proposition}[Lemma]{Proposition}
\newtheorem{Conjecture}{Conjecture}

\begin{document}
\title{Gromov-Witten/Pairs descendent correspondence
for toric
3-folds}
\author{R. Pandharipande and A. Pixton}
\date{December 2012}
\maketitle

\begin{abstract}
We construct a fully equivariant correspondence
between Gromov-Witten and stable pairs
descendent theories for toric 3-folds $X$.
Our method uses geometric constraints on
descendents, $\bA_n$ surfaces,
and the topological vertex. The rationality of
the stable pairs descendent theory plays a crucial
role in the definition of the correspondence.
We prove our correspondence has a non-equivariant limit.

As a result of the construction, we prove 
an explicit non-equivariant stationary
descendent correspondence for $X$ (conjectured 
by MNOP).
Using descendent methods, we establish
the relative GW/Pairs correspondence
for $X/D$ in several basic new 
log Calabi-Yau geometries.
Among the consequences is
a rationality constraint for  non-equivariant 
descendent Gromov-Witten series
for $\mathbf{P}^3$.
\end{abstract}

\maketitle

\setcounter{tocdepth}{2} 
\tableofcontents


\setcounter{section}{-1}
\section{Introduction}

\subsection{Descendents in Gromov-Witten theory} \label{gwdess}
Let $X$ be a nonsingular projective 3-fold.
Gromov-Witten theory is defined via integration over the moduli
space of stable maps.
Let
 $\overline{M}_{g,r}(X,\beta)$ denote the moduli space of
$r$-pointed stable maps from connected genus $g$ curves to $X$ representing the
class $\beta\in H_2(X, \Z)$. Let 
$$\text{ev}_i: \overline{M}_{g,r}(X,\beta) \rarr X,$$
$$ \LL_i \rarr \overline{M}_{g,r}(X,\beta)$$
denote the evaluation maps and the cotangent line bundles associated to
the marked points.
Let $\gamma_1, \ldots, \gamma_r\in H^*(X,{\mathbb{Q}})$, and
let $$\psi_i = c_1(\LL_i) \in H^2(\overline{M}_{g,n}(X,\beta),\mathbb{Q}).$$
The {\em descendent} fields, denoted by $\tau_k(\gamma)$, correspond 
to the classes $\psi_i^k \text{ev}_i^*(\gamma)$ on the moduli space
of maps. 
Let
$$\Big\langle \tau_{k_1}(\gamma_{1}) \cdots
\tau_{k_r}(\gamma_{r})\Big\rangle_{g,\beta} = \int_{[\overline{M}_{g,r}(X,\beta)]^{vir}} 
\prod_{i=1}^r \psi_i^{k_i} \text{ev}_i^*(\gamma_{_i})$$
denote the descendent
Gromov-Witten invariants. Foundational aspects of the theory
are treated, for example, in \cite{Beh, BehFan, LiTian}.

Let $C$ be a possibly disconnected curve with at worst nodal singularities.
The genus of $C$ is defined by $1-\chi(\oh_C)$. 
Let $\overline{M}'_{g,r}(X,\beta)$ denote the moduli space of maps
with possibly {disconnected} domain
curves $C$ of genus $g$ with {\em no} collapsed connected components.
The latter condition requires 
 each connected component of $C$ to represent
a nonzero class in $H_2(X,{\mathbb Z})$. In particular, 
$C$ must represent a {nonzero} class $\beta$.

We define the descendent invariants in the disconnected 
case by
$$\Big\langle \tau_{k_1}(\gamma_{1}) \cdots
\tau_{k_r}(\gamma_{r})\Big\rangle'_{g,\beta} = \int_{[\overline{M}'_{g,r}(X,\beta)]^{vir}} 
\prod_{i=1}^r \psi_i^{k_i} \text{ev}_i^*(\gamma_{i}).$$
The associated partition function is defined by{\footnote{Our 
notation follows \cite{MNOP2} and emphasizes the
role of the moduli space $\overline{M}'_{g,r}(X,\beta)$. 
The degree 0 collapsed contributions
will not appear anywhere in our paper.}} 
\begin{equation}
\label{abc}
\bZ'_{\mathsf{GW}}\Big(X;u\ \Big|\ \prod_{i=1}^r \tau_{k_i}(\gamma_{i})\Big)_\beta = 
\sum_{g\in{\mathbb Z}} \Big \langle \prod_{i=1}^r
\tau_{k_i}(\gamma_{i}) \Big \rangle'_{g,\beta} \ u^{2g-2}.
\end{equation}
Since the domain components must map nontrivially, an elementary
argument shows the genus $g$ in the  sum \eqref{abc} is bounded from below. 
The descendent insertions in \eqref{abc} should
match  the  (genus independent) virtual dimension,
\begin{equation}\label{k345}
\text{dim} \ [\overline{M}'_{g,r}(X,\beta)]^{vir} = \int_\beta c_1(T_X) + r.
\end{equation}

If $X$ is a nonsingular toric 3-fold, then the descendent
invariants can be lifted to equivariant cohomology.
Let 
$$\T=(\com^*)^3$$ be the 3-dimensional algebraic torus acting on $X$.
Let $s_1,s_2,s_3$ be the equivariant first Chern classes
of the standard representations of the three factors of $\T$. The
equivariant cohomology of the point is well-known to be
$$H^*_{\T}(\bullet) = \mathbb{Q}[s_1,s_2,s_3]\ .$$ 
For equivariant classes $\gamma_{i}\in H^*_{\T}(X,\mathbb{Q})$, the
descendent invariants
$$\Big\langle \tau_{k_1}(\gamma_{1}) \cdots
\tau_{k_r}(\gamma_{r})\Big\rangle'_{g,\beta} = \int_{[\overline{M}'_{g,r}(X,\beta)]^{vir}} 
\prod_{i=1}^r \psi_i^{k_i} \text{ev}_i^*(\gamma_{i})\ \in H^*_{\T}(\bullet)$$
are well-defined. In the equivariant setting, the descendent
insertions may exceed the virtual dimension \eqref{k345}.
The equivariant partition function
$$\bZ'_{\mathsf{GW}}\Big(X;u\ \Big|\ \prod_{i=1}^r \tau_{k_i}(\gamma_{i})\Big)^\T_\beta \in \mathbb{Q}[s_1,s_2,s_3]((u))$$ 
is a Laurent series in $u$ with coefficients in $H^*_\T(\bullet)$.

If $X$ is a nonsingular quasi-projective toric 3-fold, the 
equivariant Gromov-Witten invariants of $X$ are still
well-defined by localization residues \cite{BP}. In
the quasi-projective case,
$$\bZ'_{\mathsf{GW}}\Big(X;u\ \Big|\ \prod_{i=1}^r \tau_{k_i}(\gamma_{i})\Big)^\T_\beta \in \mathbb{Q}(s_1,s_2,s_3)((u))\ . $$ 
For the study of the Gromov-Witten theory of toric 3-folds,
the open geometries play an imporant role.

\subsection{Descendents in the theory of stable pairs}\label{dess}
Let $X$ be a nonsingular projective 3-fold, and let
$\beta \in H_2(X,\mathbb{Z})$ be a nonzero class. We consider next the
moduli space of stable pairs
$$[\cO_X \stackrel{s}{\rightarrow} F] \in P_n(X,\beta)$$
where $F$ is a pure sheaf supported on a Cohen-Macaulay subcurve of $X$, 
$s$ is a morphism with 0-dimensional cokernel, and
$$\chi(F)=n, \  \  \ [F]=\beta.$$
The space $P_n(X,\beta)$
carries a virtual fundamental class obtained from the 
deformation theory of complexes in
the derived category \cite{pt}.

Since $P_n(X,\beta)$ is a fine moduli space, there exists a universal sheaf
$$\FF \rightarrow X\times P_{n}(X,\beta),$$
see Section 2.3 of \cite{pt}.
For a stable pair $[\cO_X\to F]\in P_{n}(X,\beta)$, the restriction of
$\FF$
to the fiber
 $$X \times [\cO_X \to F] \subset 
X\times P_{n}(X,\beta)
$$
is canonically isomorphic to $F$.
Let
$$\pi_X\colon X\times P_{n}(X,\beta)\to X,$$
$$\pi_P\colon X\times P_{n}(X,\beta)
\to P_{n}(X,\beta)$$
 be the projections onto the first and second factors.
Since $X$ is nonsingular
and
$\FF$ is $\pi_P$-flat, $\FF$ has a finite resolution 
by locally free sheaves.
Hence, the Chern character of the universal sheaf $\FF$ on 
$X \times P_n(X,\beta)$ is well-defined.
By definition, the operation
$$
\pi_{P*}\big(\pi_X^*(\gamma)\cdot \text{ch}_{2+i}(\FF)
\cap \pi_P^*(\ \cdot\ )\big)\colon 
H_*(P_{n}(X,\beta))\to H_*(P_{n}(X,\beta))
$$
is the action of the descendent $\tau_i(\gamma)$, where
$\gamma \in H^*(X,\Z)$.

For nonzero $\beta\in H_2(X,\Z)$ and arbitrary $\gamma_i\in H^*(X,\Q)$,
define the stable pairs invariant with descendent insertions by
\begin{eqnarray*}
\Big\langle \tau_{k_1}(\gamma_1)\ldots \tau_{k_r}(\gamma_r)
\Big\rangle_{\!n,\beta}&  = &
\int_{[P_{n}(X,\beta)]^{vir}}
\prod_{i=1}^r \tau_{k_i}(\gamma_i) \\
& = & 
\int_{P_n(X,\beta)} \prod_{i=1}^r \tau_{k_i}(\gamma_{i})
\Big( [P_{n}(X,\beta)]^{vir}\Big).
\end{eqnarray*}
The partition function is 
$$
\ZZ_{\mathsf{P}}\Big(X;q\ \Big|   \prod_{i=1}^r \tau_{k_i}(\gamma_{i})
\Big)_\beta
=\sum_{n} 
\Big\langle \prod_{i=1}^r \tau_{k_i}(\gamma_{i}) 
\Big\rangle_{\!n,\beta} q^n.
$$

Since $P_n(X,\beta)$ is empty for sufficiently negative
$n$, the partition function 
is a Laurent series in $q$. The following conjecture was made in 
\cite{pt2}.

\begin{Conjecture}
\label{111} 
The partition function
$\ZZ_{\mathsf{P}}\big(X;q\ |   \prod_{i=1}^r \tau_{k_i}(\gamma_{i})
\big)_\beta$ is the 
Laurent expansion of a rational function in $q$.
\end{Conjecture}

Let $X$ be a nonsingular quasi-projective toric 3-fold. 
The stable pairs descendent
invariants can be lifted to equivariant cohomology (and
defined by residues in the open case). For
equivariant classes $\gamma_i \in H^*_{\T}(X,\mathbb{Q})$, we see
$$\bZ_{\mathsf{P}}\Big(X;q\ \Big|\ \prod_{i=1}^r \tau_{k_i}(\gamma_{i})\Big)^\T_\beta \in \mathbb{Q}(s_1,s_2,s_3)((q))$$ 
is a Laurent series in $q$ with coefficients in $H^*_\T(\bullet)$.
A central result of \cite{part1,PP2} is the following rationality
property.

\vspace{10pt}
\noindent{\bf Rationality.} {\em
Let 
$X$ be a nonsingular quasi-projective toric 3-fold. The partition function
$$\ZZ_{P}\Big(X;q\ \Big|   \prod_{i=1}^r \tau_{k_i}(\gamma_i)
\Big)^\T_\beta$$ is the 
Laurent expansion in $q$ of a rational function in the field 
$\mathbb{Q}(q,s_1,s_2,s_3)$.}

\vspace{10pt}

The above rationality result implies Conjecture 1 when $X$ is
a nonsingular projective toric 3-fold. 
The corresponding statement for the equivariant
Gromov-Witten descendent partition function is expected
(from calculational evidence) to be false.

\subsection{Correspondence} \label{corr45}
Let $X$ be a nonsingular quasi-projective toric 3-fold, and let 
$p_1,\ldots,p_m\in X$ be the distinct $\T$-fixed points.
Let $\mathsf{p}_j\in H^*_\T(X,\mathbb{Q})$
be the class of the $\T$-fixed point $p_j$.
Let $\alpha$ be a partition,
$$\alpha = (\alpha_1 \geq \alpha_2 \geq \cdots \geq \alpha_{\ell}>0), $$
of size{\footnote{The unique partition of size 0 is the empty partition
of length $\ell=0$.
In the empty case, $\tau_{\emptyset}(\mathsf{p}_j)=1$.} $|\alpha|$ and length $\ell$. 
Define the descendent insertion
\begin{equation}\label{ggp22}
\tau_\alpha(\mathsf{p}_j) =
\tau_{\alpha_1-1}(\mathsf{p}_j) \tau_{\alpha_2-1}(\mathsf{p}_j) \cdots
\tau_{\alpha_{\ell}-1}(\mathsf{p}_j) \ .
\end{equation}

Since the classes of the
$\T$-fixed points span a basis of  localized 
equivariant cohomology
$$H^*_\T(X,\mathbb{Q})\otimes
\mathbb{Q}(s_1,s_2,s_3),$$ we can consider equivariant 
descendents of $X$ in the following form
$$ 
\ZZ_{\mathsf{P}}\Big(X;q\ \Big|   \prod_{j=1}^m \tau_{\alpha^{(j)}}(\mathsf{p}_j)
\Big)^\T_\beta\ , \ \ \ 
\ZZ'_{\mathsf{GW}}\Big(X;u\ \Big|   \prod_{j=1}^m \tau_{\alpha^{(j)}}(\mathsf{p}_j)
\Big)^\T_\beta
$$
for partitions $\alpha^{(1)}, \ldots,\alpha^{(m)}$ associated to
the $\T$-fixed points.

A central result of the paper is the construction of a
universal correspondence matrix $\mathsf{K}$ indexed by partitions
$\alpha$ and $\widehat{\alpha}$ of positive size with{\footnote{As usual,
$i^2=-1$.}}
$$\mathsf{K}_{\alpha,\widehat{\alpha}}\in \mathbb{Q}[i,\bw_1,\bw_2,\bw_3]((u))\ $$
and ${\mathsf{K}}_{\alpha,\widehat{\alpha}}=0$ unless $|\alpha|\geq |\widehat{\alpha}|$.
The coefficients $\mathsf{K}_{\alpha,\widehat{\alpha}}$ are {\em symmetric}
in the variables $\bw_i$.
The matrix $\mathsf{K}$ is used to define a correspondence
rule
$$\tau_\alpha(\mathsf{p}_j) \ \ \ \mapsto \ \ \  
\widehat{\tau}_\alpha(\mathsf{p}_j) =
\sum_{|\alpha|\geq |\widehat{\alpha}|} 
{\mathsf{K}}_{\alpha,\widehat{\alpha}}(w^j_1,
w_2^j, w_3^j)\, 
\tau_{\widehat{\alpha}}(\mathsf{p}_j)\ $$
where $w_1^j, w_2^j, w_3^j$ are the tangent weights of $X$ at $p_j$.
The symmetry of $\mathsf{K}$ in the variables $\bw_i$
is required for the correspondence rule to be well-defined.
If $\alpha=\emptyset$, we formally set 
$$\widehat{\tau}_\emptyset(\mathsf{p}_j)=\widehat{1} =1\ .$$
To state the correspondence property of $\mathsf{K}$, the basic degree
$$d_\beta = \int_\beta c_1(X) \ \in \mathbb{Z}\ $$
associated to the class $\beta\in H_2(X,\mathbb{Z})$ will be 
required.

\begin{Theorem} There exists a universal correspondence matrix $\mathsf{K}$
(symmetric in the variables $\bw_i$)
satisfying \label{aaa}
$$(-q)^{-d_\beta/2}\ZZ_{\mathsf{P}}\Big(X;q\ \Big|   \prod_{j=1}^m \tau_{\alpha^{(j)}}(\mathsf{p}_j)
\Big)^\T_\beta =
(-iu)^{d_\beta}\ZZ'_{\mathsf{GW}}\Big(X;u\ \Big|   \prod_{j=1}^m \widehat{\tau}_{\alpha^{(j)}}(\mathsf{p}_j)
\Big)^\T_\beta 
$$
under the variable change $-q=e^{iu}$ 
for all nonsingular quasi-projective toric 3-folds $X$.
\end{Theorem}

The variable change in the descendent correspondence is well-defined
by the rationality result for the stable pairs partition function.
However, much of  
the $u$ dependence of $\mathsf{K}$ remains 
mysterious.{\footnote{Conjectural formulas
for a partial descendent correspondence between Gromov-Witten theory
and the Donaldson-Thomas theory of ideal sheaves are proposed in \cite{oop}.
The investigation of the relationship between descendents for
stable pairs and ideal sheaves is an interesting direction for further
study.
Though not fully equivariant, the formulas of \cite{oop} should
partially constrain $\mathsf{K}$.
}}
A central point of the paper is to show the consequences which
can be derived from various accessible properties of the $u$ dependence.

We will construct the matrix $\mathsf{K}$ from the  study
of 1-leg equivariant descendent invariants. 
A geometric argument
using capped descendent vertices following \cite{PP2}
is used to prove the 2-leg  and then the 
complete 3-leg result of Theorem \ref{aaa}.
The argument uses the full force of the equivariant Gromov-Witten/Pairs
correspondence for primary fields in \cite{moop,mpt}.

Along with the construction of $\mathsf{K}$, we prove
several basic properties.
A uniqueness statement for $\mathsf{K}$  in the
context of capped vertices appears in Theorem \ref{ccc} of Section \ref{cdvert}. 
The leading terms of $\mathsf{K}$ are determined
by the following result.

\begin{Theorem}\label{T789}
For partitions $\alpha$ of positive size, 
$\mathsf{K}_{\alpha,\alpha} =  (iu)^{\ell(\alpha)-|\alpha|}$
and 
$$\mathsf{K}_{\alpha,\widehat{\alpha}\neq \alpha} = 0 \ \ \ \text{if} \ \ \
|\alpha|\leq
 |\widehat{\alpha}| + |\ell(\alpha)-\ell(\widehat{\alpha})|\ . $$
\end{Theorem}
In other words, 
we can write the correspondence as
$$\widehat{\tau}_{\alpha}(\mathsf{p}) = (iu)^{\ell(\alpha)-|\alpha|} 
\tau_{\alpha}(\mathsf{p})
+ \ldots$$
where the dots stand for terms $\tau_{\widehat{\alpha}}$
with partitions
$\widehat{\alpha}$ of positive size satisfying 
$$|\alpha|> |\widehat{\alpha}| + |\ell(\alpha)-\ell(\widehat{\alpha})|
\ .$$

Theorem \ref{T789}, proven in Section \ref{ooolll},
plays an important role
in the applications.
We prove the
$u$ coefficients of $\mathsf{K}_{\alpha,\widehat{\alpha}}$
are symmetric polynomials in the variables $\bw_i$ in Section \ref{fullt}.

\begin{Theorem}\label{ll33}
The $u$ coefficients of  $\mathsf{K}_{\alpha,\widehat{\alpha}}\in
\mathbb{Q}[i,\bw_1,\bw_2,\bw_3]((u))$
are symmetric and homogeneous in the variables $\bw_i$
of degree $|\alpha|+\ell(\alpha) - |\widehat{\alpha}| 
- \ell(\widehat{\alpha})$. 
\end{Theorem}

\subsection{Consequences} \label{conse}
We derive several  
implications of our descendent
correspondence which require only basic properties of 
$\mathsf{K}$.

A first consequence
is the following result for the
non-equivariant partition functions with primary
fields $\tau_0(\gamma)$ and stationary descendents $\tau_k(\mathsf{p})$.

\begin{Theorem} \label{mmpp22}
Let $X$ be a nonsingular projective toric 3-fold.
After the variable change $-q=e^{iu}$, we have
\begin{multline*}
(-q)^{-d_\beta/2}\ \bZ_{\mathsf{P}}\left(X;q \
\Bigg| \ \prod_{i=1}^r {\tau}_0(\gamma_{i})
 \prod_{j=1}^s {\tau}_{k_j}(\mathsf{p}) \right)_{\beta}=\\
(-iu)^{d_\beta} (iu)^{-\sum k_j}\ 
 \bZ'_{\mathsf{GW}}\left(X;u \ \Bigg| \ \prod_{i=1}^r \tau_0(\gamma_{i}) 
\prod_{j=1}^s {\tau}_{k_j}(\mathsf{p})  \right)_{\beta} \ 
\end{multline*}
where $\gamma_i\in H^*(X,\mathbb{Q})$ are classes of positive degree.
\end{Theorem}

Theorem \ref{mmpp22} was conjectured for arbitrary nonsingular
projective 3-folds in \cite{MNOP2} for the Donaldson-Thomas
theory of ideal sheaves. Our proof, via Theorem
\ref{aaa}, uses only the leading terms of the $u$ dependence of
correspondence matrix $\mathsf{K}$.
The non-equivariant limit plays an important role in the
simple form of the descendent correspondence in Theorem \ref{mmpp22}.
If fully $\T$-equivariant partition functions are considered,
then complete knowledge of the matrix $\mathsf{K}$ is required.

By Theorem \ref{mmpp22} and the rationality result for stable pairs 
descendents, we conclude
$$e^{-\frac{iud_\beta}{2}}\cdot (-iu)^{d_\beta} (iu)^{-\sum k_j}\ 
 \bZ'_{\mathsf{GW}}\left(X;u \ \Bigg| \ \prod_{i=1}^r \tau_0(\gamma_{i}) 
\prod_{j=1}^s {\tau}_{k_j}(\mathsf{p})  \right)_{\beta}$$
is a rational function of $e^{-iu}$. We know no other approaches
to such rationality results for descendents
 in Gromov-Witten theory.

In a different direction, 
we can also prove the Gromov-Witten/Pairs correspondences
for primary fields in several new relative cases.
The first is a non-equivariant log Calabi-Yau geometry
with the relative
divisor given by a $K3$ surface.

\begin{Theorem} \label{mmpp44}
Let $X$ be a nonsingular projective Fano toric 3-fold, and let
$K3\subset X$ be a nonsingular anti-canonical $K3$ surface.
After the variable change $-q=e^{iu}$, 
we have
\begin{multline*}
(-q)^{-d_\beta/2}\,
\bZ_{\mathsf{P}}
\Big( X/K3;q\ \Big|\
\tau_0(\gamma_1)\cdots\tau_0(\gamma_r)\, \Big| \mu \Big)_\beta
=\\
(-iu)^{d_\beta+\ell(\mu)-|\mu|}\,  \bZ'_{\mathsf{GW}}
\Big( X/K3;u\ \Big|\
\tau_0(\gamma_1)\cdots\tau_0(\gamma_r) \, \Big| \mu  \Big )_\beta
\,,
\end{multline*}
where $\gamma_i\in H^*(X,\mathbb{Q})$ are arbitrary classes.
\end{Theorem}

Relative Gromov-Witten and stable pairs theory are reviewed
in Section \ref{relth}. Our standard conventions for the boundary
conditions $\mu$ along the  relative divisors are
explained there.
The rationality of the stable pairs series of Theorem \ref{mmpp44}
has been establised earlier in Section 4 of \cite{PP2}.
Theorem \ref{mmpp44} can be used to prove a rationality constraint for
the 
Gromov-Witten descendent series of $\mathbf{P}^3$.

Let $\mathbb{Q}(-q,i)[u,\frac{1}{u}]$ be the ring of Laurent polynomials in $u$
with coefficients given by rational functions in $-q$ over
$\mathbb{Q}[i]$. For example
$$   \frac{q-\frac{1}{q}}{2i}{u^{-2}} + \frac{q+\frac{1}{q}}{2} 
u^4 \ \in\  
\mathbb{Q}(-q,i)[u,\frac{1}{u}]\ .$$

\setcounter{Corollary}{2}
\begin{Corollary} \label{yaya34}
For the non-equivariant descendent series, we have 
$$
 \bZ'_{\mathsf{GW}}\left(\mathbf{P}^3;u \ \Bigg| \ 
\prod_{j=1}^s {\tau}_{k_j}(\gamma_j)  \right)_{\beta} \ \in\  
\mathbb{Q}(-q=e^{iu},i)[u,\frac{1}{u}]\ ,
$$
where $\gamma_j\in H^*(\mathbf{P}^3,\mathbb{Q})$ are classes of positive degree.
\end{Corollary}

 Let  $\bA_n$ be the minimal  
toric resolution of the standard
$A_n$-singularity obtained by a quotient of a
cyclic $\mathbb{Z}_{n+1}$-action on $\com^2$, see Section
\ref{angeo} for a review.
Consider the $(\com^*)^2$-equivariant geometry relative
geometry
$$\bA_n \times \Pp/ D =
\bA_n \times \Pp\ / \ (\bA_n)_{x_1} \cup (\bA_n)_{x_2} \cup (\bA_n)_{x_3} 
$$
relative to the  fibers over the
distinct point $x_1,x_2,x_3 \in \Pp$.
Here, $(\com^*)^2$ acts only on the toric surface $\bA_n$.
We prove the following result.

\begin{Theorem}  \label{yaya38}
After the variable change $-q=e^{iu}$, 
we have
\begin{multline*}
(-q)^{-d_\beta/2}\,
\bZ_{\mathsf{P}}
\Big( \bA_n\times\Pp/D\   ;q\ \Big| \mu^{(1)}, \mu^{(2)}, \mu^{(3)} 
\Big)^{(\com^*)^2}_\beta
=\\
(-iu)^{d_\beta+\sum_{i=1}^k \ell(\mu^{(i)})-|\mu^{(i)}|}\,  \bZ'_{\mathsf{GW}}
\Big(\bA_n\times\Pp/D\ ;u\  \Big| \mu^{(1)}, \mu^{(2)}, \mu^{(3)}
   \Big )^{(\com^*)^2}_\beta
\,,
\end{multline*}
where the $\mu^{(i)}$ are arbitrary $(\com^*)^2$-equivariant 
relative conditions along 
the fibers $(\bA_n)_{x_i}$.
\end{Theorem}

Theorem \ref{yaya38}
resolves questions left open in \cite{mo1,mo2}. In fact,
Theorem \ref{yaya38} would follow from the
results of \cite{mo1,mo2} if certain conjectured
invertibilities were
established, see Section 8.3 of  \cite{mo2}. Our proof of the 3-point 
Gromov-Witten/Pairs correspondence
for $\bA_n$-local curves completely bypasses such
 invertibility issues.

The results stated above are the first applications of the
equivariant descendent correspondence. 
The main application will be to establish the
Gromov-Witten/Pairs correspondence for several basic 
families of compact Calabi-Yau 3-folds. The strategy is
to follow the methods of \cite{mptop} which determine the
Gromov-Witten theory of the quintic 3-fold
$$X_5 \subset \mathbf{P}^4\ $$
and to take parallel geometric steps for 
the stable pairs theory.
A non-equivariant Gromov-Witten descendent
correspondence is necessary for the argument.

A basic result of the present paper is
a non-equivariant formulation of the 
Gromov-Witten/pairs descendent correspondence.
The application to compact Calabi-Yau 3-folds
will be taken up in \cite{PPcy3}.

\subsection{Non-equivariant limit}
Let $X$ be a nonsingular quasi-projective toric 3-fold
with $\T$-fixed points $p_1,\ldots,p_m$ and
inclusions 
$$\iota_j: p_j \hookrightarrow X,\ 
\ \ \  \iota: X^\T\hookrightarrow X\ .$$
The pull-back of the
top Chern class of the tangent bundle,
$$\iota_j^*(c_3(T_X))= {e(\text{Tan}_j)} \in H^*_\T(p_j,\mathbb{Q})\ ,$$
is the Euler class of the tangent representation at $p_j$.

Theorem \ref{aaa} establishes a descendent
correspondence for $\T$-equivariant Gromov-Witten
and stable pairs theories.
Does Theorem \ref{aaa} define a correspondence
for non-equivariant theories?
Certainly every non-equivariant descendent  
$\tau_{k}(\gamma)$ 
can be lifted to
a combination of $\T$-equivariant descendents
of the form $\tau_k(\mathsf{p}_j)$ by
localization
\begin{equation}\label{ff34}
\widetilde{\gamma}= \iota_* 
\sum_{j=1}^m \frac{\iota_j^*(\widetilde{\gamma})}
{\iota_j^*(c_3(T_X))}
\ \p_j
\end{equation}
where $\widetilde{\gamma}$ is any $\T$-equivariant 
lift of $\gamma$.
Theorem \ref{aaa}
can then be applied.  
However, the coefficients
$$\frac{\iota_j^*(\widetilde{\gamma})}
{e(c_3(T_X))}\in \mathbb{Q}(s_1,s_2,s_3)$$
which appear on the right of \eqref{ff34} are 
rational functions of the $s_i$ (almost
always with poles).
After the application of Theorem \ref{aaa},  
poles in $s_i$ will occur on the Gromov-Witten
side of the correspondence.
Whether the resulting combination of
Gromov-Witten invariants 
can be rewritten in
non-equvariant terms is not immediately clear.
The outcome depends upon properties of the
correspondence matrix $\mathsf{K}$.

We prove a Gromov-Witten/Pairs descendent
correspondence for $X$ which admits a 
non-equivariant limit. In order to state the answer, we
will require the following notation.
Let $\widehat{\alpha}$ be a partition
of length $\widehat{\ell}$ as in Section \ref{corr45}.
Let $\Delta$ be the cohomology class of the
small diagonal in the product $X^{\widehat{\ell}}$.
For a cohomology class $\gamma$ of $X$, let
$$\gamma\cdot \Delta=
\sum_{{j_1, \ldots, j_{\hat{\ell}}}} \theta_{j_1}^\gamma \otimes
\ldots\otimes \theta_{j_{\hat{\ell}}}^\gamma$$
be the K\"unneth decomposition of $\gamma\cdot
\Delta$ in the cohomology of  $X^{\widehat{\ell}}$.
We define the descendent insertion $\tau_{\widehat{\alpha}}(\gamma)$ by
\begin{equation}\label{j77833}
\tau_{\widehat{\alpha}}(\gamma)= 
\sum_{j_1,\ldots,j_{\hat{\ell}}}
\tau_{\widehat{\alpha}_1-1}(\theta_{j_1}^\gamma)
\cdots\tau_{\widehat{\alpha}_{\hat{\ell}}-1}(\theta_{j_{\hat{\ell}}}^\gamma)\ .
\end{equation}
For example,
if 
$\gamma$ is the class of a point, then 
\[
\tau_{\widehat{\alpha}}(\mathsf{p})=
\tau_{\widehat{\alpha}_1-1}(\mathsf{p})\cdots\tau_{\widehat{\alpha}_{\hat{\ell}}-1}(\mathsf{p})
\]
in accordance with convention \eqref{ggp22}.
Definition \eqref{j77833} is valid for both the standard and
the $\T$-equivariant cohomology of $X$.

We  
construct a second
correspondence matrix $\widetilde{\mathsf{K}}$ 
indexed by partitions
$\alpha$ and $\widehat{\alpha}$ of positive size with
$$\widetilde{\mathsf{K}}_{\alpha,\widehat{\alpha}}\in 
\mathbb{Q}[i,c_1,c_2,c_3]((u))\ $$
and $\widetilde{\mathsf{K}}_{\alpha,\widehat{\alpha}}=0$ 
unless $|\alpha|\geq |\widehat{\alpha}|$.
Via the substitution
$$c_i=c_i(T_X),$$
the elements of $\widetilde{\mathsf{K}}$
act on the cohomology (both standard
and $\T$-equivariant) of $X$ with $\mathbb{Q}[i]$-coefficients.
Of course, we take the canonical lift of $\T$ to the
tangent bundle $T_X$ in the equivariant case.

The matrix $\widetilde{\mathsf{K}}$ is 
used to define a new correspondence
rule
\begin{equation}\label{pddff}
{\tau_{\alpha_1-1}(\gamma_1)\cdots
\tau_{\alpha_{\ell}-1}(\gamma_{\ell})}\ \  \mapsto\ \ 
\overline{\tau_{\alpha_1-1}(\gamma_1)\cdots
\tau_{\alpha_{\ell}-1}(\gamma_{\ell})}\ .
\end{equation}
The formula for the right side
of \eqref{pddff} requires a sum over all set
partitions $P$ of $\{ 1,\ldots, \ell \}$.
 For such a  set partition
$P$, each element $S\in P$
is a subset of $\{1,\ldots, \ell\}$.
Let $\alpha_S$ be the associated subpartition of
$\alpha$, and let
$$\gamma_S = \prod_{i\in S}\gamma_i.$$
We define the right side of \eqref{pddff} by
\begin{equation}\label{mqq23}
\overline{\tau_{\alpha_1-1}(\gamma_1)\cdots
\tau_{\alpha_{\ell}-1}(\gamma_{\ell})}
=
\sum_{P \text{ set partition of }\{1,\ldots,\ell\}}\ \prod_{S\in P}\ \sum_{\widehat{\alpha}}\tau_{\widehat{\alpha}}(\widetilde{\mathsf{K}}_{\alpha_S,\widehat{\alpha}}\cdot\gamma_S) \ .
\end{equation}

\begin{Theorem} There exists a universal correspondence matrix 
$\widetilde{\mathsf{K}}$
satisfying \label{zzz}
\begin{multline*}
(-q)^{-d_\beta/2}\ZZ_{\mathsf{P}}\Big(X;q\ \Big|  
{\tau_{\alpha_1-1}(\gamma_1)\cdots
\tau_{\alpha_{\ell}-1}(\gamma_{\ell})}
\Big)^\T_\beta \\ =
(-iu)^{d_\beta}\ZZ'_{\mathsf{GW}}\Big(X;u\ \Big|   \
\overline{\tau_{\alpha_1-1}(\gamma_1)\cdots
\tau_{\alpha_{\ell}-1}(\gamma_{\ell})}\
\Big)^\T_\beta 
\end{multline*}
under the variable change $-q=e^{iu}$ 
for all nonsingular quasi-projective toric 3-folds $X$.
\end{Theorem}

We prove Theorem \ref{zzz} by 
constructing $\widetilde{\mathsf{K}}$ canonically
from $\mathsf{K}$. 
Divisibility properties of the coefficients of $\mathsf{K}$, required
for the construction of $\widetilde{\mathsf{K}}$, are proven geometrically.
From our construction of $\widetilde{\mathsf{K}}$, we will see
Theorem \ref{zzz} specializes to Theorem \ref{aaa}.
The main advantage of Theorem \ref{zzz} over Theorem \ref{aaa}
is the obvious existence of a non-equivariant limit.
In fact, since \eqref{mqq23} makes sense for
the standard cohomology of {\em any} nonsingular projective
3-fold, we conjecture the following.

\begin{Conjecture}
\label{ttt222} 
For any nonsingular projective $3$-fold $X$, 
we have 
\begin{multline*}
(-q)^{-d_\beta/2}\ZZ_{\mathsf{P}}\Big(X;q\ \Big|  
{\tau_{\alpha_1-1}(\gamma_1)\cdots
\tau_{\alpha_{\ell}-1}(\gamma_{\ell})}
\Big)_\beta \\ =
(-iu)^{d_\beta}\ZZ'_{\mathsf{GW}}\Big(X;u\ \Big|   
\ \overline{\tau_{\alpha_1-1}(\gamma_1)\cdots
\tau_{\alpha_{\ell}-1}(\gamma_{\ell})}\ 
\Big)_\beta 
\end{multline*}
under the variable change $-q=e^{iu}$.
\end{Conjecture}

By Conjecture \ref{111}, the stable pairs descendent series
on the left is expected to be a rational function in $q$, so the change
of variables is well-defined.
Conjecture \ref{ttt222} is a consequence of
Theorem \ref{zzz} in case $X$ is toric by taking
the non-equivariant limit. In the non-toric
case, Conjecture \ref{ttt222} predicts the
correspondence is the same.

Formula \eqref{mqq23} assumes all the cohomology classes
$\gamma_j$ are even.
In the presence of odd cohomology, a natural sign
must be included in \eqref{mqq23}. We may write 
set partitions $P$ of $\{1,\ldots, \ell\}$ indexing 
the sum on the right side of \eqref{mqq23} 
as
$$S_1\cup \ldots\cup S_{|P|} = \{1,\ldots, \ell\}.$$
The parts $S_i$ of $P$ are unordered, but we choose an ordering
for each $P$.
We then 
obtain a permutation of $\{1, \ldots, \ell\}$
by placing the elements in the ordered parts $S_i$ (and
respecting the original order in each part).
The permutation determines a sign $\sigma(P)$
 by the anti-commutation of the associated
odd classes. We then write:
\begin{equation*}
\overline{\tau_{\alpha_1-1}(\gamma_1)\cdots
\tau_{\alpha_{\ell}-1}(\gamma_{\ell})}
=
\sum_{P \text{ set partition of }\{1,\ldots,\ell\}}\ (-1)^{\sigma(P)}
\prod_{S_i\in P}\ \sum_{\widehat{\alpha}}\tau_{\widehat{\alpha}}(\widetilde{\mathsf{K}}_{\alpha_{S_i},\widehat{\alpha}}
\cdot\gamma_{S_i}) \ .
\end{equation*}

\subsection{Plan of the paper}
We start in Section \ref{cdvert} by reviewing relative
Gromov-Witten and stable pairs theories. The capped
descendent vertex of \cite{PP2}, defined in Section \ref{defcap},
will play a central role in the construction of 
the correspondence matrix $\mathsf{K}$.
Theorem \ref{aaa} is implied by a relative 
descendent correspondence stated in Theorem \ref{ccc} of
Section \ref{defcap}.

To construct $\mathsf{K}$, we proceed leg by leg for
capped descendent vertices.
The study of the 1-leg case
in Section \ref{ooolll} uniquely determines
$\mathsf{K}$ by the invertibility of the associated descendent/relative
matrices.
We define $\mathsf{K}$ via the 1-leg geometry.
After the proof of the 1-leg descendent correspondence
in Section \ref{cmptt}, the intial terms of $\mathsf{K}$
are calculated in Section \ref{bacpro} to prove Theorem \ref{T789}. 
The symmetry between the variables $s_i$ is broken
in the 1-leg geometry, so the symmetry of $\mathsf{K}$
is not immediate.

We review the technique of capped localization in Section 
\ref{cl}. The capped descendent correspondence in 
the 2-leg case is established in Section \ref{fullt}
using the geometry of $\bA_1$ surfaces (a strategy already
employed in \cite{moop,PP2}).
Crucial here is a new invertibility proven in
Section \ref{maxer}. 
A consequence of the 2-leg correspondence
is the symmetry of $\mathsf{K}$ in the variables $\bw_i$
proven in Section \ref{4444}. The 3-leg case
is obtained by a parallel argument in Section \ref{fulltt},
completing the proof of Theorem \ref{ccc} and thus of Theorem \ref{aaa}.

The first applications of the descendent correspondence
are taken up in Section \ref{ffirr}. 
The easiest is Theorem \ref{yaya38} proven in Section \ref{yayaya}.
After a further
study of the initial terms of $\mathsf{K}$, we prove
Theorem \ref{mmpp22} in Section \ref{mmppr}.

Section \ref{neql} concerns the non-equivariant formulation of
the descendent correspondence. After delicate
divisibility properties
for the coefficients of $\mathsf{K}$ are established,
the formula for $\widetilde{\mathsf{K}}$
in terms of $\mathsf{K}$ is given in Section \ref{div555}.
Theorem \ref{zzz} is then a consequence of Theorem \ref{aaa}.
The final applications of the paper, in Section \ref{8888}
to log Calabi-Yau geometries, require the non-equivariant
correspondence. Theorem \ref{mmpp44} is proven in
Section \ref{mmpp44pr}.

\subsection{Acknowledgments}
Discussions with J. Bryan, D. Maulik, A. Oblomkov, A. Okounkov, and
R. Thomas 
about  Gromov-Witten theory, stable pairs,
and descendent invariants
played an important role. 
The study of descendents
for 3-fold sheaf theories in \cite{MNOP2,pt2} 
motivated several 
aspects of the paper. 
 
R.P. was partially supported by NSF grants DMS-0500187
and DMS-1001154.
A.P. was supported by a NDSEG graduate fellowship.
The paper was started in the summer of 2011
while visiting the 
Instituto Superior T\'ecnico in Lisbon where
R.P. was supported by a Marie Curie fellowship and
a grant from the Gulbenkian foundation.
Much of the work was done during visits of A.P.
to ETH Z\"urich during the 2011-12 year.

\section{Capped descendent vertex} \label{cdvert}

\subsection{Relative theories}\label{relth}
Let $D\subset X$ be a nonsingular divisor. Relative Gromov-Witten and 
relative stable pairs
theories enumerate curves with specified tangency to
the divisor $D$. See \cite{MNOP2,part1} for a technical discussion
of relative theories.

In Gromov-Witten theory, relative conditions  are
represented by a partition $\mu$ of the integer
$
\int_\beta [D],
$
each part $\mu_i$ of which is marked
by a cohomology class $\gamma_i\in H^*(D,\mathbb{Z})$. 
The numbers $\mu_i$ record
the multiplicities of intersection with $D$
while the cohomology labels $\gamma_i$ record where the tangency occurs.
More precisely, let $\oM_{g,r}'(X/D,\beta)_\mu$ be the
moduli space of stable relative maps with tangency conditions
$\mu$ along $D$. To impose the full boundary condition,
we 
pull-back the product $\Pi_i\gamma_i$
via the evaluation maps
$$
\oM_{g,r}'(X/D,\beta)_\mu \to D
$$
at the points of tangency.
By convention, an absent cohomology label stands for
$1\in H^*(D,\mathbb{Z})$. Also, the tangency points are
considered to be unordered.

In the stable pairs theory, the relative moduli space admits a natural
morphism to the Hilbert scheme of
$d$ points in $D$,
$$P_n(X/D,\beta) \to \Hilb(D,\int_\beta [D])\ .$$ 
Cohomology classes on $\Hilb(D,\int_\beta [D])$ may 
thus
be pulled-back to the relative moduli space. We will work in
the \emph{Nakajima basis} of $H^*(\Hilb(D,\int_\beta [D]),\mathbb{Q})$ indexed
by a partition $\mu$ of $\int_\beta [D]$
labeled by cohomology classes of $D$. For example, the
class
$$
\left.\big|\mu\rang \in H^*(\Hilb(D,\int_\beta [D]),\mathbb{Q})\,,
$$
with all cohomology labels equal to the identity,
 is $\prod \mu_i^{-1}$ times
the Poincar\'e dual of the closure of the subvariety formed by unions of
schemes of length
$$
\mu_1,\dots, \mu_{\ell(\mu)}
$$
supported at $\ell(\mu)$ distinct points of $D$.

The conjectural relative GW/Pairs correspondence for primary
fields \cite{MNOP2}
equates the partition functions of the theories.

\begin{Conjecture}\label{htht}
We have
\begin{multline*}
(-q)^{-d_\beta/2}\,
\bZ_{\mathsf{P}}
\Big( X/D;q\ \Big|\
\tau_0(\gamma_1)\cdots\tau_0(\gamma_r)\, \Big| \mu \Big)_\beta
=\\
(-iu)^{d_\beta+\ell(\mu)-|\mu|}\,  \bZ'_{\mathsf{GW}}
\Big( X/D;u\ \Big|\
\tau_0(\gamma_1)\cdots\tau_0(\gamma_r) \, \Big| \mu  \Big )_\beta
\,,
\end{multline*}
after the change of variables $e^{iu}=-q$. 
\end{Conjecture}

As before,  $\bZ_{\mathsf{P}}
\left(\left. X/D;q\ \right|
\tau_0(\gamma_1)\cdots\tau_0(\gamma_r)\, \big| \mu \right)_\beta$
is
conjectured to be a rational function of $q$.

\subsection{Degeneration formulas}\label{sdeg}

Relative theories satisfy degeneration formulas. Let
$$
\mathfrak{X}\to B
$$
be a nonsingular $4$-fold fibered over an irreducible and
nonsingular base curve $B$. Let $X$ be a nonsingular
fiber, and let
$$
X_1 \cup_{D} X_2
$$
be a reducible special fiber consisting of two nonsingular
$3$-folds intersecting transversally along a nonsingular
surface $D$.

If all insertions $\gamma_1,\dots,\gamma_r$ lie in the
image of
$$
H^*(X_1 \cup_{D} X_2,\mathbb{Z}) \to H^*(X,\mathbb{Z})\,,
$$
the degeneration formula in Gromov-Witten theory takes the form
\cite{IP, LR,L}
\begin{multline}\label{degGW}
 \bZ'_{\mathsf{GW}}
\left(\left. X\right|
\tau_{k_1}(\gamma_1)\cdots\tau_{k_r}(\gamma_r)\, \right)_\beta =\\
\sum \bZ'_{\mathsf{GW}}
\left(\left. X_1\right|
\,\dots\, \big| \mu  \right)_{\beta_1}  \,
\zz(\mu) \, u^{2\ell(\mu)} \,
\bZ'_{\mathsf{GW}}
\left(\left. X_2\right|
\,\dots\, \big| \mu^\vee  \right)_{\beta_2} \,,
\end{multline}
where the summation is over all curve splittings
$\beta=\beta_1+\beta_2$, all splittings of the
insertions $\tau_{k_i}(\gamma_i)$, and all relative conditions $\mu$.

In \eqref{degGW}, the cohomological labels of
$\mu^\vee$ are Poincar\'e duals of the labels of $\mu$.
The gluing factor $\zz(\mu)$ is the order of the
centralizer in the symmetric group $S(|\mu|)$ of
an element with cycle type $\mu$.

The degeneration formula in the stable pairs theory takes a very similar
form, 
\begin{multline*}
 \bZ_{\mathsf{P}}
\left(\left. X\right|
\tau_{k_1}(\gamma_1)\cdots \tau_{k_r}(\gamma_r)\, \right)_\beta =\\
\sum \bZ_{\mathsf{P}}
\left(\left. X_1\right|
\,\dots\, \big| \mu  \right)_{\beta_1}  \, 
(-1)^{|\mu|-\ell(\mu)} \,
\zz(\mu) \, q^{-|\mu|} \,
\bZ_{\mathsf{P}}
\left(\left. X_2\right|
\,\dots\, \big| \mu^\vee  \right)_{\beta_2} \,,
\end{multline*}
see \cite{MNOP2,part1}. The sum over the relative conditions
$\mu$ is interpreted as the coproduct of $1$,
$$
\Delta 1 = \sum_{\mu}
(-1)^{|\mu|-\ell(\mu)} \, \zz(\mu) \, 
\left.\big|\mu\rang \otimes \left.\big|\mu^\vee\rang\ ,
$$
in the
tensor square of $H^*(\Hilb(D,\int_\beta [D]),\mathbb{Z})$.
Conjecture \eqref{htht} is easily seen to be compatible with
degeneration.

\subsection{Definition of the capped vertex}\label{defcap}
Bare capped vertices were first considered in \cite{moop} in the
context of Donaldson-Thomas theory.
The capped descendent vertex was introduced in \cite{PP2} to 
prove the rationality of the stable pairs theory of
toric 3-folds. Capped vertices will play a basic
role in the construction and proof of the descendent
correspondence.
We review the definitions here.

Let $\T$ be a 3-dimensional algebraic torus. As before,  let
$s_1,s_2,s_3 \in H^*_{\T}(\bullet)$ be the first  Chern classes
of the standard representations of the three factors of $\T$.
Let $\T$ act diagonally on $\Pp\times \Pp \times \Pp$,
$$(\xi_1,\xi_2,\xi_3) \cdot ([x_1,y_1],[x_2,y_2],[x_3,y_3]) =
([x_1,\xi_1 y_1],[x_2,\xi_2 y_2],[x_3,\xi_3 y_3])\ .$$
Let $0,\infty \in \Pp$ be the points $[1,0]$ and $[0,1]$ respectively.
The tangent weights{\footnote{Our sign conventions here
follow \cite{PPstat} and disagree with \cite{part1}.}} of $\T$ at the point
$$\mathsf{p}=(0,0,0) \in \Pp \times \Pp \times \Pp$$
are $s_1$, $s_2$, and $s_3$.

Let $U\subset \Pp\times \Pp \times \Pp$ be the $\T$-invariant 3-fold obtained
by removing the three $\T$-invariant lines 
$$
L_1,L_2,L_3 \subset \Pp \times \Pp \times \Pp
$$
passing through the point $(\infty,\infty,\infty)$,
\begin{equation}\label{tthh2}
U = \Pp \times \Pp \times \Pp \ \setminus \ \cup_{i=1}^3 L_i.
\end{equation}
Let $D_i \subset U$
be the divisor with $i^{th}$ coordinate $\infty$.
For $i\neq j$, the divisors ${D}_i$ and $D_j$
are disjoint in $U$.

The capped descendent vertex is a 
partition function of $U$ with integrand
$$\tau_{\alpha}(\mathsf{p})=\tau_{\alpha_1-1}(\mathsf{p}) \ldots 
\tau_{\alpha_{\ell}-1}(\mathsf{p})$$ and 
free relative conditions
imposed at the divisors ${D}_i$. 
While the relative geometry $U/\cup_i D_i$
is not compact, the moduli spaces
 $P_n(U/\cup_i D_i,\beta)$
have compact $\T$-fixed loci.
The stable pairs invariants of $U/\cup_i D_i$ are
well-defined by $\T$-equivariant residues.
In the localization
formula for the reduced theories of 
 $U/\cup_i D_i$,
 nonzero degrees can occur {\em only} on the edges meeting
the origin $\mathsf{p}\in U$.

We
denote the capped descendent vertex by
\begin{equation} \label{tgg6}
\bC( \tau_\alpha(\mathsf{p})\  | \ \lambda^{(1)},\lambda^{(2)},\lambda^{(3)} )
  =  \mathsf{Z}
(U/\cup_i D_i, \prod_{i=1}^\ell \tau_{\alpha_i-1}(\mathsf{p})\ | \ 
\lambda^{(1)},\lambda^{(2)},\lambda^{(3)}
)^\T_\beta 
\end{equation}
where the partition
 $\alpha$ specifies the descendent integrand and the partitions
$
\lambda^{(1)},\lambda^{(2)},\lambda^{(3)}
$ denote relative conditions imposed at $D_1,D_2,D_3$.
While
 $\alpha$ must be non-empty as before, the partitions
$\lambda^{(i)}$ are permitted to be empty.
However,
we require the condition
\begin{equation}\label{tre3}
|\lambda^{(1)}|+|\lambda^{(2)}|+|\lambda^{(3)}| > 0 
\end{equation}
to hold. The number of legs of the vertex
is the number of non-empty partitions among
$\lambda^{(1)}$, $\lambda^{(2)}$, and $\lambda^{(3)}$.

The curve class $\beta$ in \eqref{tgg6} is determined 
by the relative conditions: $\beta$ is the sum of 
the three axes passing through $\mathsf{p}\in U$
with coefficients $|\lambda^{(1)}|$, $|\lambda^{(2)}|$, and 
$|\lambda^{(3)}|$ respectively. 
The superscript $\T$ after the bracket denotes
$\T$-equivariant integration on
$P_n(U/\cup_i D_i,\beta)$.
The condition \eqref{tre3} implies $\beta$ is nonzero.

The above definition is valid for both Gromov-Witten and
stable pairs theories. The relative conditions are
interpretated as tangency in Gromov-Witten theory
and as element of the Nakajima basis in the theory of
stable pairs. We denote the vertices in the two theories by
$$\bC_{\mathsf{P}}( \tau_\alpha(\mathsf{p})\  |\ \lambda^{(1)},\lambda^{(2)},\lambda^{(3)} ), \ \ \ 
\bC_{\mathsf{GW}}( \tau_\alpha(\mathsf{p})\  |\ 
\lambda^{(1)},\lambda^{(2)},\lambda^{(3)})\ .$$
We will prove Theorem \ref{aaa}
by a refined correspondence result for the capped descendent vertex.
\begin{Theorem}\label{ccc}
There exists a unique correspondence matrix $\mathsf{K}$
satisfying 
\begin{multline*}
(-q)^{-\sum_i|\lambda^{(i)}|}\bC_{\mathsf{P}}(\tau_{\alpha}(\mathsf{p})\ |\ \lambda^{(1)},
\lambda^{(2)},\lambda^{(3)})= \\
(-iu)^{\sum_i|\lambda^{(i)}|+\ell(\lambda^{(i)})}\bC_{\mathsf{GW}}(
 \widehat{\tau}_{\alpha}(\mathsf{p})
\ | \ \lambda^{(1)},
\lambda^{(2)},\lambda^{(3)})
\end{multline*}
under the variable change $-q=e^{iu}$. 
\end{Theorem}

In case no descendents are present, the basic equality of
the equivariant capped vertices
\begin{equation*}
(-q)^{-\sum_i|\lambda^{(i)}|}\bC_{\mathsf{P}}(1\ |\ \lambda^{(1)},
\lambda^{(2)},\lambda^{(3)})= 
(-iu)^{\sum_i|\lambda^{(i)}|+\ell(\lambda^{(i)})}\bC_{\mathsf{GW}}(1
 \
|\ \lambda^{(1)},
\lambda^{(2)},\lambda^{(3)})
\end{equation*}
is the main result of \cite{moop}.
The
number of legs of the descendent vertex refers to the
number of non-empty partitions among $\lambda^{(1)}$ , $\lambda^{(2)}$, and
$\lambda^{(3)}$.

By capped localization discussed in Section \ref{cl}, we will
easily derive Theorem \ref{aaa} from Theorem \ref{ccc}.
The proof of Theorem \ref{ccc} will be given leg by leg
starting with the 1-leg case.

\section{Descendent correspondence: 1-leg}
\label{ooolll}
\subsection{Construction of $\mathsf{K}$} \label{conk}
We construct the matrix $\mathsf{K}$ in several steps. The first is
very simple.
Let $d>0$ be an integer, and let $\mathcal{P}_d$ be the set
of  partitions of positive size at most $d$.
Let 
\begin{equation*}
\bC^d_{\mathsf{P}}(\alpha,\lambda) =\bC_{\mathsf{P}}(\tau_\alpha(\mathsf{p}) \ | \
 \emptyset,\emptyset,\lambda), \ \ \ \ \alpha,\lambda\in \mathcal{P}_d
\end{equation*}
be a matrix with both rows and
columns indexed by $\mathcal{P}_d$.
The coefficients of $\bC^d_{\mathsf{P}}$ are rational functions in $q$.

\begin{Lemma} \label{lsp} For all $d>0$, the matrix $\bC^d_{\mathsf{P}}$ is
invertible.
\end{Lemma}

Similarly, we define the corresponding matrix using the 1-leg
Gromov-Witten descendent vertices,
\begin{equation*}
\bC^d_{\mathsf{GW}}(\alpha,\lambda) =\bC_{\mathsf{GW}}(\tau_\alpha(\mathsf{p}) \ | \
\emptyset,\emptyset,\lambda), \ \ \ \ \alpha,\lambda\in \mathcal{P}.
\end{equation*}
The coefficients of $\bC^d_{\mathsf{P}}$ are Laurent series in $u$.

\begin{Lemma} \label{lgw} For all $d>0$, the matrix $\bC^d_{\mathsf{GW}}$ is
invertible.
\end{Lemma}

By the invertibility of Lemmas \ref{lsp} and \ref{lgw}, there exists
a unique correspondence matrix ${\mathsf{K}}^d$ indexed by $\mathcal{P}_d$ 
with coefficients in $\mathbb{Q}(i,s_1,s_2,s_3)((u))$ which
satisfies the condition
\begin{equation*}
(-q)^{-|\lambda|}\bC_{\mathsf{P}}(\tau_{\alpha}(\mathsf{p})\ |\ \emptyset,
\emptyset,\lambda)= 
(-iu)^{|\lambda|+\ell(\lambda)}\bC_{\mathsf{GW}}\Big( 
\sum_{\widehat{\alpha}\in {\mathcal{P}_d}} {\mathsf{K}}^d_{\alpha,\widehat{\alpha}}
\,  {\tau}_{\widehat{\alpha}}(\mathsf{p})
\ \Big| \ \emptyset,
\emptyset,\lambda\Big)
\end{equation*}
under the variable change $-q=e^{iu}$ 
for all $\alpha,\lambda \in \mathcal{P}_d$.

\vspace{10pt}
\noindent{\bf{Definition.}} {\em
 The correspondence matrix ${\mathsf{K}}$ is defined
by the following
rule:
$${\mathsf{K}}_{\alpha,\widehat{\alpha}}=
{\mathsf{K}}^{|\alpha|}_{\alpha,\widehat{\alpha}}$$
if $|\alpha|\geq |\widehat{\alpha}|$ and 
${\mathsf{K}}_{\alpha,\widehat{\alpha}}=0$
otherwise.}
\vspace{10pt}

After proving Lemmas \ref{lsp} and \ref{lgw} in Section \ref{lll} below,
we will use a geometric argument in Section \ref{cmptt} to prove the following
compatibility statement.

\begin{Proposition} \label{prprr}
For all $d\geq |\alpha|$,
${\mathsf{K}}^d_{\alpha,\widehat{\alpha}}=
{\mathsf{K}}^{|\alpha|}_{\alpha,\widehat{\alpha}}$
for all $\widehat{\alpha}$.
\end{Proposition}

Lemmas \ref{lsp}-\ref{lgw} and Proposition \ref{prprr}
together yield
a unique correspondence matrix
$${\mathsf{K}}_{\alpha,\widehat{\alpha}}\in \mathbb{Q}(i,s_1,s_2,s_3)((u))$$
satisfying Theorem \ref{ccc}
in the 1-leg case.
The proof of Theorem \ref{ccc} in the 2- and 3-leg
cases will be presented in Sections \ref{fullt} and \ref{fulltt}. 
Theorem \ref{aaa}
will be derived as a consequence.

By our construction, the $u$ coefficients of
${\mathsf{K}}_{\alpha,{\widehat{\alpha}}}$ are easily seen to
be homogeneous rational functions in the variables $s_i$ of
degree $|\alpha|+\ell(\alpha)-|\widehat{\alpha}|-\ell(\widehat{\alpha})$.
The claim follows from the homogeneity of the 
coefficients of the matrices
$$\bC^d_{\mathsf{P}}(\alpha,\lambda), \ \ \ 
\bC^d_{\mathsf{GW}}(\alpha,\lambda)$$
obtained from geometric dimensional analysis.

Theorem \ref{ll33} asserts further  the symmetry
and polynomiality of the coefficients ${\mathsf{K}}_{\alpha,{\widehat{\alpha}}}$.
Unfortunately, the construction  of $\mathsf{K}$ from the 1-leg
geometry breaks the symmetry between the variables
$s_i$. The symmetry of $\mathsf{K}$ will be established
in Section \ref{4444} as a step in the proof of Theorem \ref{ccc}.
The restriction of the coefficients of $\mathsf{K}$
to the subring
$$\mathbb{Q}[i,s_1,s_2,s_3]((u)) \subset \mathbb{Q}(i,s_1,s_2,s_3)((u))$$
will also be proven in Section \ref{4444},
completing the proof of Theorem \ref{ll33}.

\subsection{Proof of Lemmas \ref{lsp} and \ref{lgw}}\label{lll}
On the set $\mathcal{P}_d$, we define two partial orderings:
\begin{enumerate}
\item[$\bullet$]
$\alpha \unrhd \widetilde{\alpha} \ \ \ \ 
\Longleftrightarrow \ \ \ \ |\alpha|-\ell(\alpha) \geq |\widetilde{\alpha}|
-\ell(\widetilde{\alpha}) \ ,$
\item[$\bullet$]
$\alpha \succcurlyeq \widetilde{\alpha} \ \ \ \ \Longleftrightarrow \ \ \ \
\ell_+(\alpha) \geq
\ell_+(\widetilde{\alpha}) \ ,$
\end{enumerate}
where $\ell_+(\alpha)$ is the number of parts of $\alpha$
which are strictly greater than 1. The conditions
$\alpha \rhd \widetilde{\alpha}$ and $\alpha \succ \widetilde{\alpha}$
are defined via the corresponding strict inequalities.

\begin{Lemma}  \label{efef}
If $\alpha \lhd \lambda$, then
$\bC^d_{\mathsf{P}}(\alpha,\lambda)=0$ and
 $\bC^d_{\mathsf{GW}}(\alpha,\lambda)=0$.
\end{Lemma}

\begin{proof}
The result is a consequence of a dimension count.
The 1-leg vertex can be studied via the {\em cap} geometry,
$$N = \cO_{\Pp} \oplus \cO_{\Pp} \rightarrow \Pp \ ,$$ 
relative to the fiber
$$N_\infty \subset N$$
over $\infty \in \Pp$.
The total space $N$ naturally carries an action of a 
3-dimensional torus $\T$ where the first two factors
scale the components of the rank 2 trivial bundle
 and last
factor acts on $\Pp$ with fixed points $0, \infty\in \Pp$. 
Let the scaling weights be $s_1$ and $s_2$, and let the
tangent weight along $\Pp$ at the fixed point $p_0$
over $0\in \Pp$ be $s_3$.
The $\T$-action
preserves the relative divisor $N_\infty$. The 
equivariant relative geometry $N/N_\infty$ is
equivalent to
 \eqref{tthh2} in the 1-leg case.

Let $N_0$ be the fiber over $0\in \Pp$ of the cap,
and let $\mathsf{N}_0$ be the associated class in equivariant
cohomology. We have the relation
$$\mathsf{N}_0 = \frac{\mathsf{p}_0}{s_1s_2} \ $$
in the equivariant cohomology of the cap.
For the relative conditions, 
we can weight each part of $\lambda$ by the 
fixed point $p_\infty$ over $\infty\in \Pp$.
Let  $\lambda[\mathsf{p}_\infty]$ denote the
resulting weighted partition.
We have the relation
$$\mathsf{1} = \frac{\mathsf{p}_\infty}{s_1s_2} $$
in the equivariant cohomology of the divisor over $\infty\in \Pp$.
Hence, we can write
$$\bC^d_{\mathsf{P}}(\alpha,\lambda) =
(s_1s_2)^{\ell(\alpha)-\ell(\lambda)}\
\bZ_{\mathsf{P}}\Big(\mathsf{Cap};q\ \Big|
\ \prod_{i=1}^{\ell(\alpha)} \tau_{\alpha_i-1}(\mathsf{N}_0)\ \Big| \ 
\lambda[\mathsf{p}_\infty]
\Big)^\T_{|\lambda|} \ ,  $$
$$\bC^d_{\mathsf{GW}}(\alpha,\lambda) =
(s_1s_2)^{\ell(\alpha)-\ell(\lambda)}\
\bZ'_{\mathsf{GW}}\Big(\mathsf{Cap} ;u\ \Big|
\ \prod_{i=1}^{\ell(\alpha)} \tau_{\alpha_i-1}(\mathsf{N}_0)\ \Big| \ 
\lambda[\mathsf{p}_\infty]
\Big)^\T_{|\lambda|} \ .  $$

The moduli spaces of relative stable pairs and relative
stable maps to the cap with boundary condition
$\lambda[\mathsf{p}_\infty]$ are both compact of virtual dimension
$|\lambda|-\ell(\lambda)$. The integrand
$$\prod_{i=1}^{\ell(\alpha)} \tau_{\alpha_i-1}(\mathsf{N}_0)$$
imposes $|\alpha|-\ell(\alpha)$ condition in both theories.
If the dimension of the integrand is strictly
less than the virtual dimension, the equivariant integral
vanishes for compact moduli spaces.
\end{proof}

\begin{Lemma} \label{rwrw}
If $|\alpha|-\ell(\alpha)= |\lambda|-\ell(\lambda)$
and $\alpha\prec \lambda$, then
$$\bC^d_{\mathsf{P}}(\alpha,\lambda)=0, \ \ \ \
 \bC^d_{\mathsf{GW}}(\alpha,\lambda)=0\ .$$
\end{Lemma}

\begin{proof}
If $\lambda=(\lambda_1)$ has a single part (which must exceed
1 by the second hypothesis), then $\alpha$ must have all
parts equal to 1. The first hypothesis
\begin{equation}\label{tjjt}
|\alpha|-\ell(\alpha)= |\lambda|-\ell(\lambda)
\end{equation}
then can not hold. Hence, the Lemma is true
if the length of $\lambda$ is 1.
We will assume the length of $\lambda$ is at least 2
and proceed by induction.

Consider the equivariant geometry of $\PP^2 \times \PP^1$
relative to the fiber 
$$\PP^2_\infty = \PP^2 \times \{ \infty \}\subset
\PP^2 \times \PP^1\ .$$
Let 
$\mathsf{L}\in H_2(\PP^2\times \PP^1, \mathbb{Z})$
be the class of the section $\PP^1$ 
contracted over $\PP^2$.
The first two factors of $\T$ acts on $\PP^2$
with fixed points $\xi_0,\xi_1,\xi_2\in \PP^2$.
The tangent weights can be chosen as follows:
$$s_1,s_2 \ \text{ for } \xi_0,\ \ 
-s_1,s_2-s_1 \ \text{ for } \xi_1,\ \ 
s_1-s_2,-s_2 \ \text{ for } \xi_2\ .$$
The last factor of $\T$ acts on $\Pp$ as before with weight $s_3$ at
$0\in \Pp$. 
Let $\widetilde{\lambda}$
be the cohomology weighted partition obtained from $\lambda$ 
with all parts  
  weighted by $[\xi_0]\in H^*_\T(\PP^2,\mathbb{Q})$.
The integral
$$\int_{[P_n(\PP^2 \times \PP^1/\PP^2_\infty
,|\lambda|\mathsf{L})_{\widetilde{\lambda}}]^{vir}}
\prod_{i=1}^{\ell(\alpha)} \tau_{\alpha_i-1}(\PP^2_0) \ \ \  \in 
{\mathbb{Q}} $$
has  integrand dimension equal to  the virtual dimension
and hence is independent of equivariant lift. 

Let $\lambda_1> 1$ be the largest part of
 $\lambda$. Let
$$\lambda' = \lambda \setminus \{ \lambda_1\}\ $$
which is not empty since $\ell(\lambda)$ at least 2.
Let $\widetilde{\widetilde{\lambda}}$
be the cohomology weighted partition with the parts
$\lambda_i$  weighted by $[\xi_0]\in H^*_\T(\PP^2,\mathbb{Q})$
except $\lambda_1$ which is weighted
by  $[\xi_1]\in H^*_\T(\PP^2,\mathbb{Q})$. By the independence
of the choice of equivariant lift,
\begin{multline}\label{prrtt}
\int_{[P_n(\PP^2 \times \PP^1/
\PP^2_\infty
,|\lambda|\mathsf{L})_{\widetilde{\lambda}}]^{vir}}
\prod_{i=1}^{\ell(\alpha)} \tau_{\alpha_i-1}(\PP^2_0)  
= \\
\int_{[P_n(\PP^2 \times \PP^1/
\PP^2_\infty
,|\lambda|\mathsf{L})_{\widetilde{\widetilde{\lambda}}}]^{vir}}
\prod_{i=1}^{\ell(\alpha)} \tau_{\alpha_i-1}(\PP^2_0)\ .\ \ \ \ \  
\end{multline}
The left side of \eqref{prrtt} is exactly the $q^n$ coeffient of
$$\bZ_{\mathsf{P}}\Big(\mathsf{Cap};q\ \Big|
\ \prod_{i=1}^{\ell(\alpha)} \tau_{\alpha_i-1}(\mathsf{N}_0)\ \Big| \ 
\lambda[\mathsf{p}_\infty]
\Big)^\T_{|\lambda|}\ .$$
Similarly the right side of \eqref{prrtt} is the $q^n$ coefficient of
\begin{multline*}
\sum_{\alpha'\cup \alpha''=\alpha}
\frac{|\text{Aut}(\lambda')|}{|\text{Aut}(\lambda)|}
 \
\bZ_{\mathsf{P}}\Big(\mathsf{Cap};q\ \Big|
\ \prod_{i=1}^{\ell(\alpha')} \tau_{\alpha'_i-1}(\mathsf{N}_0)\ \Big| \ 
\lambda'[\mathsf{p}_\infty]
\Big)^\T_{|\lambda'|}\cdot\\
\left[\bZ_{\mathsf{P}}\Big(\mathsf{Cap};q\ \Big|
\ \prod_{i=1}^{\ell(\alpha'')} \tau_{\alpha''_i-1}(\mathsf{N}_0)\ \Big| \ 
\lambda_{1}[\mathsf{p}_\infty]
\Big)^\T_{\lambda_1} \right]_{s_1=-s_1, s_2=s_2-s_1}\ ,
\end{multline*}
where the sum is over disjoint splittings of $\alpha$ and

The hypothesis \eqref{tjjt} implies
$$|\alpha'|-\ell(\alpha') + |\alpha''|-\ell(\alpha'')
= |\lambda'|-\ell(\lambda')+ |\lambda_1| -1\ .$$
For nonvanishing terms of the sum, by Lemma \ref{efef}, we must
have
$$|\alpha'|-\ell(\alpha') =|\lambda'|-\ell(\lambda'), \ \ \ 
 |\alpha''|-\ell(\alpha'')
= |\lambda_1| -1\ .$$
At least one of the conditions $\alpha'\prec \lambda'$ or
$\alpha''\prec (\lambda_1)$
must hold. Since $\alpha''\prec (\lambda_1)$
is impossible,
the condition
$\alpha'\prec \lambda'$ must hold.
The induction statement is established.

The argument for Gromov-Witten theory is literally identical.
The formal properties used above hold also for Gromov-Witten
theory.
\end{proof}

We define an equivalence relation $\sim$ on $\mathcal{P}_d$
by the following rule: $\alpha\sim \widetilde{\alpha}$ if 
$\alpha$ and $\widetilde{\alpha}$ differ only by parts of size 1.
For example,
$$(4,4,3,1) \sim (4,4,3,1,1,1)\ .$$
The proof of Lemma \ref{rwrw} in fact yields a refined result.

\begin{Lemma} \label{rwrwrw}
If $|\alpha|-\ell(\alpha)= |\lambda|-\ell(\lambda)$,
$\ \ell_+(\alpha)=\ell_+(\lambda)$, 
and $\alpha\nsim \lambda$
then
$$\bC^d_{\mathsf{P}}(\alpha,\lambda)=0, \ \ \ \
 \bC^d_{\mathsf{GW}}(\alpha,\lambda)=0\ .$$
\end{Lemma}

By Lemma \ref{efef}, the matrices $\bC^d_{\mathsf{P}}$ and
$\bC^d_{\mathsf{GW}}$ are block lower-triangular with respect
to the partial ordering $\unrhd$. In order to establish
invertibility, we need only study the blocks where 
\begin{equation}\label{pe44}
|\alpha|-\ell(\alpha) = |\lambda|-\ell(\lambda)
\end{equation}
 is fixed.
By Lemma \ref{rwrw}, the above blocks themselves are
block lower-triangular with respect to the partial
ordering $\succcurlyeq$. So, we need only study blocks where
both \eqref{pe44} and
$$\ell_+(\alpha)= \ell_+(\lambda)$$
are fixed.
By Lemma \ref{rwrwrw}, we finally restrict ourselves
to the square blocks where the equivalence class 
under $\sim$ is fixed.

Let $\gamma\in\mathcal{P}_d$ be a partition with no parts equal to 
1. The evaluation 
\begin{equation}\label{ffgg}
\bZ_{\mathsf{P}}\Big(\mathsf{Cap};q\ \Big|
\ \prod_{i=1}^{\ell(\gamma)} \tau_{\gamma_i-1}(\mathsf{N}_0)\ \Big| \ 
\gamma[\mathsf{p}_\infty]\Big)^\T_{|\gamma|} = 
\frac{q^{|\gamma|}}{|\text{Aut}(\gamma)|}
 \prod_{i=1}^{\ell(\gamma)} \frac{1}{\gamma_i!}\ ,
\end{equation}
has been computed in \cite{PPstat} and does not vanish.
The cardinality of the equivalence class
under $\sim$ determined by $\gamma$ is  $d-|\gamma|$.
By the divisor equation,
the block of $\bC^d_{\mathsf{P}}$ corresponding to the
equivalence class of $\gamma$ is, up to harmless $(s_1s_2)^{\ell(\alpha)-\ell(\lambda)}$ factors, the matrix
$$\left( \begin{array}{ccccc}
 \frac{1}{0!}&    \frac{1}{1!} &  \frac{1}{2!} &  \ldots &   \frac{1}{(d-|\gamma|)!} \\
    \frac{|\gamma|}{0!} &   \frac{|\gamma|+1}{1!} &  \frac{|\gamma|+2}{2!} &  \ldots &   \frac{d}{(d-|\gamma|)!} \\
    \frac{|\gamma|^2}{0!} &   \frac{(|\gamma|+1)^2}{1!} &  \frac{(|\gamma|+2)^2}{2!} &  \ldots &   \frac{d^2}{(d-|\gamma|)!} \\
      \vdots  & \vdots &  \vdots &   \vdots &  \vdots \\
      \frac{|\gamma|^{d-|\gamma|}}{0!} &   \frac{(|\gamma|+1)^{d-|\gamma|}}{1!} &  
\frac{(|\gamma|+2)^{d-|\gamma|}}{2!} &  \ldots &   \frac{d^{d-|\gamma|}}{(d-|\gamma|)!}
\end{array}\right)$$
with every element multiplied by \eqref{ffgg}.
Invertiblity is immediate from  the Vandermonde determinant.

The argument for the Gromov-Witten matrix $\bC^d_{\mathsf{GW}}$ is identical. 
The replacement
for \eqref{ffgg} is the evaluation
\begin{equation}\label{ffggg}
\bZ'_{\mathsf{GW}}\Big(\mathsf{Cap};u\ \Big|
\ \prod_{i=1}^{\ell(\gamma)} \tau_{\gamma_i-1}(\mathsf{N}_0)\ \Big| \ 
\gamma[\mathsf{p}_\infty]\Big)^\T_{|\gamma|} = 
\frac{u^{-2\ell(\gamma)}}{|\text{Aut}(\gamma)|}
 \prod_{i=1}^{\ell(\gamma)} \frac{1}{\gamma_i!}
\ 
\end{equation}
obtained{\footnote{Beware of the typographical
error of a factor of $d$ in Lemma 7 of \cite{unknot}.}} from Lemma 7 of \cite{unknot}.
The proofs of Lemmas \ref{lsp} and \ref{lgw} are complete.
\qed

\vspace{10pt}
We define $\mathsf{K}^{|\alpha|}_{\alpha,\widehat{\alpha}}$ using the 
invertibility of $\bC^{|\alpha|}_{\mathsf{P}}$ and
$\bC^{|\alpha|}_{\mathsf{P}}$. A direct consequence of Lemma 
\ref{efef} is the following vanishing.

\begin{Lemma}  \label{efefef}
If $\alpha \lhd \widehat{\alpha}$, then
$\mathsf{K}^{|\alpha|}_{\alpha,\widehat{\alpha}}=0$.
\end{Lemma}

We will later require an invertibility result which is
clear from our matrix analysis here.

\begin{Lemma} \label{7777}
The submatrix of $\mathsf{C}^d_{\mathsf{P}}(\alpha,\lambda)$
determined by the conditions
$$d=|\alpha|=|\lambda|$$
is invertible (even after the restriction $s_3=0$).
\end{Lemma}

\begin{proof} By Lemmas \ref{efef}-\ref{rwrwrw}, the submatrix is
lower-triangular with respect to the partial ordering.
Moreover the diagonal elements are nonzero with no $s_3$
dependence.
The evaluation $s_3=0$ is well-defined since the
dependence of the descendents of the cap have no
poles along $s_3$, see Lemma 1 of \cite{part1}. 
\end{proof}

\subsection{Proof of Proposition \ref{prprr}}\label{cmptt}
The Proposition will follow once we establish the
identity
\begin{multline}
(-q)^{-|\lambda|}\bC_{\mathsf{P}}(\tau_{\alpha}(\mathsf{p})\ |\ \emptyset,
\emptyset,\lambda)=\\ 
(-iu)^{|\lambda|+\ell(\lambda)}\bC_{\mathsf{GW}}\Big( 
\sum_{\widehat{\alpha}\in {\mathcal{P}_{|\alpha|}}} {\mathsf{K}}^{|\alpha|}_{\alpha,\widehat{\alpha}}
\,  {\tau}_{\widehat{\alpha}}(\mathsf{p})
\ \Big| \ \emptyset,
\emptyset,\lambda\Big) \label{gtt12}
\end{multline}
for all nonempty partitions $\lambda$.

Let $d=|\lambda|$.
If $d\leq|\alpha|$, the identity holds by the definition of
$\mathsf{K}^{|\alpha|}$.
To prove the identity for $d>|\alpha|$, 
we employ a geometric relation using
the relative space $$\PP^2 \times \PP^1 / \PP^2_\infty$$
introduced in the proof of Lemma \ref{rwrw}.
Once identity \eqref{gtt12} is proven, the Proposition follows
from the invertibility of $\mathsf{K}^d$.

Let $\widehat{\lambda}$ be the 
the cohomology weighted partition with the largest part
$\lambda_1$  weighted by $[\xi_0]\in H^*_\T(\PP^2,\mathbb{Q})$
and all  the remaining parts $\lambda_2,\ldots, 
\lambda_{\ell(\lambda)}$ weighted
by  $1\in H^*_\T(\PP^2,\mathbb{Q})$.
We will consider the
partition functions
\begin{equation*}
\bZ_{\mathsf{P}}(\alpha, \widehat{\lambda})=\bZ_{\mathsf{P}}\Big( \PP^2 \times \PP^1 / \PP^2_\infty      ;q\ \Big|
\ \Delta \cdot\prod_{i=1}^{\ell(\alpha)} \tau_{\alpha_i-1}(\mathsf{N}_0)\ \Big| \ 
\widehat{\lambda}\Big)^\T_{d\mathsf{L}},
\end{equation*}
\begin{equation*}
\bZ_{\mathsf{GW}}(\alpha, \widehat{\lambda})=
\bZ_{\mathsf{GW}}\Big( \PP^2 \times \PP^1 / \PP^2_\infty      ;u\ \Big|
\ \Delta \cdot\prod_{i=1}^{\ell(\alpha)} \tau_{\alpha_i-1}(\mathsf{N}_0)\ \Big| \ 
\widehat{\lambda}\Big)^\T_{d\mathsf{L}},
\end{equation*}
where $\Delta$ is the small diagonal condition obtained from $\PP^2$.

To explain the small diagonal class $\Delta$ in the case of
stable pairs, a recasting of the descendents is required.
Let $X/D$ be a 3-fold relative geometry, and 
let $\beta\in H_2(X,\mathbb{Z})$.
Let 
$$\mathcal{X} \rightarrow P_n(X/D,\beta)$$
be the universal space. We consider the $r^{th}$ fiber product
$$\pi_{P}:\mathcal{X}^r \rightarrow P_n(X/D,\beta)$$
of $\mathcal{X}$ over $P_n(X/D,\beta)$ with
projections
$$\pi_i: {\mathcal{X}}^r \rightarrow \mathcal{X} $$
for $1\leq i \leq r$ onto the $i^{th}$ factor.
After composing with the canonical contraction $\mathcal{X}\rightarrow X$,
we obtain
$$\pi_{i,X}: \mathcal{X}^r \rightarrow X\ .$$
The Chern character of the universal sheaf $\FF\rightarrow {\mathcal{X}}$ on 
the univeral space $\mathcal{X}$ is well-defined.
The operation
\begin{equation}\label{nrrt}
\pi_{P*}\left(\prod_{i=1}^r
\pi_{i,X}^*(\gamma_i)\cdot \pi_i^*(\text{ch}_{2+k_i}(\FF))
 \ \cap \pi_P^*(\ \cdot\ )\right)
\end{equation}
on $H_*(P_n(X/D,\beta)$
is defined to be the action of the 
descendent $\prod_{i=1}^r\tau_{k_i}(\gamma_i)$, where
$\gamma_i \in H^*(X,\Z)$. By the push-pull formula, definition
\eqref{nrrt} agrees with the descendents constructed in Section \ref{dess}.

The advantage here of definition \eqref{nrrt} is the existence of a morphism
$$\pi_{X^r}:\mathcal{X}^r \rightarrow X^r, \ \ \ 
\pi_{X^r}=(\pi_{1,X},\ldots,\pi_{r,X})
\ .$$
Any class in $\delta\in H^*(X^r,\mathbb{Q})$ can be included in the
descendent as
\begin{equation*}
\pi_{P*}\left(
\pi_{X^r}^*(\delta)
\prod_{i=1}^r
\pi_{i,X}^*(\gamma_i)\cdot \pi_i^*(\text{ch}_{2+k_i}(\FF))
\ \cap \pi_P^*(\ \cdot\ )\right)\ .
\end{equation*}
Of course, $\pi_{X^r}^*(\delta)$ can be incorporated in the $\gamma_i$
by the K\"unneth decompostion. However, in the equivariant
case, the K\"unneth decompostion requires inversion of the equivariant
parameters and interferes with dimension arguments.

In the case relevant to the proof of Proposition \ref{prprr}, 
$$X/D=\PP^2 \times \PP^1 / \PP^2_\infty\ ,$$
and $\Delta$ is the class on the $\ell(\alpha)$-fiber product
of the universal space $\mathcal{X}$
obtained by pulling back the
class of the small diagonal of $(\PP^2)^{\ell(\alpha)}$,
$$ \mathcal{X}^{\ell(\alpha)} \rightarrow  
(\PP^2\times \PP^1)^{\ell(\alpha)}\rightarrow (\PP^2)^{\ell(\alpha)}\ .$$
%
Since the moduli space of maps has marked points, the
parallel construction for Gromov-Witten theory is immediate.

\begin{Lemma} If $d>|\alpha|$, then
$$\bZ_{\mathsf{P}}(\alpha, \widehat{\lambda})=\bZ_{\mathsf{GW}}(\alpha, \widehat{\lambda})=0\ .$$ \label{vann2}
\end{Lemma}
\begin{proof}
The virtual dimension of the stable pairs and stable maps
moduli spaces with the relative condition $\widehat{\lambda}$
imposed is
$$ d+\ell(\lambda)-2 \ .$$
The dimension of the integrand in both cases is
$$|\alpha|-\ell(\alpha)+ 2(\ell(\alpha)-1),$$
with the last term accounting for the small diagonal $\Delta$.
If 
\begin{equation}\label{pp556}
|\alpha|+\ell(\alpha)-2 < d+ \ell(\lambda)- 2
\end{equation}
then the Lemma is obtained from dimension constraints
for the compact geometry.

We may also express the theories of $\PP^2 \times \PP^1/\PP^2_\infty$
by localization in terms of the 1-leg descendent vertex.
A simple analysis using Lemma \ref{efef} shows
the vanishing of the Lemma holds if
\begin{equation}\label{qq556}
|\alpha|-\ell(\alpha) < d - \ell(\lambda)\ .
\end{equation}

Finally, we observe if neither condition \eqref{pp556}
nor condition \eqref{qq556} are satisfied, then
$$ 2|\alpha| -2 \geq 2d -2$$
which violates the hypothesis $d>|\alpha|$.
\end{proof}

\vspace{10pt}
The constraint on 1-leg descendent vertices obtained by
localization of the vanishing of Lemma \ref{vann2}
expresses
$$\bC_{\mathsf{P}}(\tau_{\alpha}(\mathsf{p})\ |\ \lambda,
\emptyset,\emptyset) \ \ \  \text{and}\ \ \
\bC_{\mathsf{GW}}(\tau_{\alpha}(\mathsf{p})\ |\ \lambda,
\emptyset,\emptyset)
$$
 in terms of 
$$\bC_{\mathsf{P}}(\tau_{\alpha}(\mathsf{p})\ |\ \lambda',
\emptyset,\emptyset)
\ \ \  \text{and}\ \ \
\bC_{\mathsf{GW}}(\tau_{\alpha}(\mathsf{p})\ |\ \lambda',
\emptyset,\emptyset)
$$
where $\lambda'\subset \lambda$ is a strict subset.
Moreover, the reduction equation respects
identity \eqref{gtt12} since all
the  $\tau_{\widehat{\alpha}}$ which appear
in $$\widehat{\tau}_\alpha= 
\sum_{\widehat{\alpha}\in {\mathcal{P}_{|\alpha|}}}
{\mathsf K}^{|\alpha|}_{\alpha,\widehat{\alpha}}
\, \tau_{\widehat{\alpha}}$$ 
also satisfy $d>|\widehat{\alpha}|$.
By induction, 
we have established identity \eqref{gtt12} and completed
the proof of Proposition \ref{prprr}. 
\qed

\vspace{10pt}
The result of Proposition \ref{prprr} is simultaneously
the construction of  the matrix $\mathsf{K}$ and the
proof of Theorem \ref{ccc} in the 1-leg case.

\subsection{Basic properties of $\mathsf{K}$} \label{bacpro}
By construction, we have the vanishing
$$|\alpha|< |\widehat{\alpha}| \ \ \Longrightarrow \ \  \mathsf{K}_{\alpha,\widehat{\alpha}}=0 \ .$$
We have already established the vanishing
\begin{equation}\label{fragg3}
\alpha \lhd \widehat{\alpha} \ \ \Longrightarrow \ \ 
\mathsf{K}_{\alpha,\widehat{\alpha}}=0 
\end{equation}
in Lemma \ref{efefef}.
A more subtle result is the following.
 
\begin{Proposition}\label{t789}
We have
$$\widehat{\tau}_{\alpha}(\mathsf{p}) = (iu)^{\ell(\alpha)-|\alpha|} 
\tau_{\alpha}(\mathsf{p})
+ \ldots$$
where the dots stand for terms $\tau_{\widehat{\alpha}}$
with
$\alpha \rhd \widehat{\alpha}$.
\end{Proposition}

\begin{proof} 
By the vanishing \eqref{fragg3},
we need only consider $\widehat{\alpha}$
for which $$|\alpha|-\ell(\alpha) = |\widehat{\alpha}|- 
\ell(\widehat{\alpha}).$$
By Proposition \ref{prprr}, we have
\begin{multline*}
(-q)^{-|\widehat{\alpha}|}\,
\bZ_{\mathsf{P}}
\Big( \mathsf{Cap}   ;q\ \Big|\
\tau_{\alpha}(\mathsf{p}_0)
\, \Big| \widehat{\alpha} \ \Big)^{\T}_{|\widehat{\alpha}|}
=\\
(-iu)^{|\widehat{\alpha}|+\ell(\widehat{\alpha})}\,  \bZ'_{\mathsf{GW}}
\Big( \mathsf{Cap} ;u\ \Big|\
\sum_{\mu} \mathsf{K}_{\alpha,\mu}
\tau_{\mu}(\mathsf{p}_0)
\, \Big| \widehat{\alpha}  \Big )^\T_{|\widehat{\alpha}|}
\, .
\end{multline*}
Using the proven invertibilities,
we need only match
\begin{equation}\label{mat1111}
(-q)^{-|\widehat{\alpha}|}\bZ_{\mathsf{P}}\Big(\mathsf{Cap};q\ \Big|
\ \tau_{\alpha}(\mathsf{p}_0)\ \Big| \ 
\widehat{\alpha} \Big)^\T_{|\widehat{\alpha}|}  
\end{equation}
with the series
\begin{equation} \label{mat2222}
(-iu)^{|\widehat{\alpha}|+\ell(\widehat{\alpha})} 
\bZ'_{\mathsf{GW}}
\Big(\mathsf{Cap};u\ \Big|
\ 
(iu)^{\ell(\alpha)-|{\alpha}|}
\tau_{\alpha}(\mathsf{p}_0)\ \Big| \ 
\widehat{\alpha} \Big)^\T_{|\widehat{\alpha}|}\ .
\end{equation}
The required matching is established in the next Lemma.
\end{proof}

After trading $\mathsf{p}_0$ insertions on the cap for
$\mathsf{N}_0$ via
$$\mathsf{p}_0= s_1s_2 \mathsf{N}_0$$
and using standard dimension arguments, we reduce
the necessary matching to the following result.

\begin{Lemma}\label{match1}
Let $\alpha$ be a partition of positive size, and let
$d-1= |\alpha|-\ell(\alpha)$. Then, 
\begin{equation*}
\bZ_{\mathsf{P}}\Big(\mathsf{Cap};q\ \Big|
\ \prod_{i=1}^{\ell(\alpha)} \tau_{\alpha_i-1}(\mathsf{N}_0)\ \Big| \ 
d[\mathsf{p}_\infty]\Big)^\T_{d} = q^d\  \frac{d^{\ell(\alpha)-2}}{\prod_{i=1}^{\ell(\alpha)} (\alpha_i-1)!} \ ,
\end{equation*}
\begin{equation*}
\bZ'_{\mathsf{GW}}\Big(\mathsf{Cap};u\ \Big|
\ \prod_{i=1}^{\ell(\alpha)} \tau_{\alpha_i-1}(\mathsf{N}_0)\ \Big| \ 
d[\mathsf{p}_\infty]\Big)^\T_{d}
= u^{-2} \ \frac{d^{\ell(\alpha)-2}}{\prod_{i=1}^{\ell(\alpha)} (\alpha_i-1)!}\ .
\end{equation*}
\end{Lemma}
\begin{proof}  The Gromov-Witten calculation is well-known. The result
follows directly from Lemma 7 of \cite{unknot} after translation of
notation (and accounting of a factor $d$ typographical error 
in the formula in \cite{unknot}).

The stable pairs evaluation can be computed by localization using the same methods as in Lemma~4 of \cite{PPstat}. By dimension counting, the result
 has no dependence on $s_1$ and $s_2$. Therefore,
 we can work mod $s_1+s_2$. We obtain the sum
\begin{equation*}
\frac{(-1)^{d+\ell(\alpha)-1)} q^d}{d\cdot d! \cdot \prod_{i = 1}^{\ell(\alpha)}(\alpha_i+1)!}\sum_{a+b=d-1}(-1)^a\binom{d-1}{a}P_\alpha(a),
\end{equation*}
where
\begin{equation*}
P_\alpha(a) = \prod_{i = 1}^{\ell(\alpha)}(-(-b-1)^{\alpha_i+1}+(-b)^{\alpha_i+1}+a^{\alpha_i+1}-(a+1)^{\alpha_i+1})
\end{equation*}
is a polynomial in $a$ with leading term{\footnote{The lower terms
do not contribute since $\sum_{a+b=d-1}(-1)^a\binom{d-1}{a} a^k = 0$ for $k<d-1$.}}
\begin{equation*}
\prod_{i = 1}^{\ell(\alpha)}(-d(\alpha_i+1)\alpha_i)a^{d-1}.
\end{equation*}
The stable pairs evaluation is then
\begin{equation*}
\frac{(-1)^{d+\ell(\alpha)-1} q^d}{d\cdot d! \cdot \prod_{i = 1}^{\ell(\alpha)}(\alpha_i+1)!}(-1)^{d-1}(d-1)!\prod_{i = 1}^{\ell(\alpha)}(-d(\alpha_i+1)\alpha_i)
= \frac{d^{\ell(\alpha)-2}q^d}{\prod_{i = 1}^{\ell(\alpha)}(\alpha_i-1)!}.
\end{equation*}
\end{proof}

We define yet another
partial ordering $\uncrazeright$ on partitions by 
$$\alpha \uncrazeright \widetilde{\alpha} \ \ \ \ 
\Longleftrightarrow \ \ \ \ |\alpha|+\ell(\alpha) \geq |\widetilde{\alpha}|
+\ell(\widetilde{\alpha}) \ \ .$$
The conditions
$\alpha \crazeright \widetilde{\alpha}$ 
is defined via the corresponding strict inequalities.
In Section \ref{descsp}, we will prove the following result
parallel to Proposition \ref{t789}. 

\begin{Proposition}\label{t789789}
We have
$$\widehat{\tau}_{\alpha}(\mathsf{p}) = (iu)^{\ell(\alpha)-|\alpha|} 
\tau_{\alpha}(\mathsf{p})
+ \ldots$$
where the dots stand for terms $\tau_{\widehat{\alpha}}$
with
$\alpha \crazeright \widehat{\alpha}$.
\end{Proposition}

Propositions \ref{t789} and \ref{t789789} together immediately
imply Theorem \ref{T789} constraining the intial terms of
$\mathsf{K}$.

As a Corollary of Proposition \ref{t789}, we see the simple form
$$\widehat{\tau}_{1^\ell}(\mathsf{p}) = \tau_{1^\ell}(\mathsf{p})$$
holds for $\alpha=(1^\ell)$ since no partition satisfies
$(1^\ell)\rhd \widehat{\alpha}$.

Let $(1)+\alpha$
be the partition obtained by adding a part equal to 1
to $\alpha$.
The part $1$ corresponds to a $\tau_0(\mathsf{p})$
factor in $\tau_{(1)+\alpha}(\mathsf{p})$.
We can write
$$\tau_{0}(\mathsf{p}_0) = s_1s_2\tau_{0}(\mathsf{N}_0)$$
in the cap geometry. Using the $\T$-equivariant divisor equations{\footnote{
In the Gromov-Witten case, if $k_j-1<0$, the summand is omitted
in the divisor equation.}} for the cap,
\begin{eqnarray*}
\Big\langle \tau_0(\mathsf{N}_0) \prod_{i=1}^r \tau_{k_i}(\mathsf{p}_0) 
\  \Big | \ \mu \Big\rangle_{n,d}^{\mathsf{P}} & = &
d\ \Big\langle \tau_0(\mathsf{N}_0) \prod_{i=1}^r \tau_{k_i}(\mathsf{p}_0) 
\  \Big | \ \mu 
\Big\rangle_{n,d}^{\mathsf{P}} \\
\Big\langle \tau_0(\mathsf{N}_0) \prod_{i=1}^r \tau_{k_i}(\mathsf{p}_0) 
\  \Big | \ \mu 
\Big\rangle_{g,d}^{\mathsf{P}} & = &
d\ \Big\langle \tau_0(\mathsf{N}_0) \prod_{i=1}^r \tau_{k_i}(\mathsf{p}_0) 
\  \Big | \ \mu 
\Big\rangle_{g,d}^{\mathsf{GW}}\\
& & \ \ \ \
+\sum_{j=1}^r s_3 
 \Big\langle 
\tau_{k_j-1}(\mathsf{p}_0)
\prod_{i\neq j} \tau_{k_i }(\mathsf{p}_0) 
\  \Big | \ \mu 
\Big\rangle_{g,d}^{\mathsf{GW}}\ ,
\end{eqnarray*}
we can easily understand 
how the matrix $\mathsf{K}$ treats the part 1.
To state the answer, let
$$\Phi\left(\tau_{k_1}(\mathsf{p}) \cdots \tau_{k_r}(\mathsf{p})\right) =
\sum_{j=1}^r 
\tau_{k_j-1}(\mathsf{p})
\prod_{i\neq j} \tau_{k_i }(\mathsf{p}) $$
and extend $\Phi$ linearly to linear combinations of
monomials in $\tau_k(\mathsf{p})$.

\begin{Proposition}
For partitions $\alpha$,
$$\widehat{\tau}_{ (1)+\alpha}(\mathsf{p}) =
\tau_0(\mathsf{p}) \cdot \widehat{\tau}_{\alpha}
- s_1s_2s_3 \cdot \Phi\left( \widehat{\tau}_{\alpha}(\mathsf{p})\right)\ .$$
\end{Proposition}

\begin{proof}
The divisor equations show the proposed formula 
for $\widehat{\tau}_{ 1+\alpha}(\mathsf{p})$ respects
the 1-leg
correspondence of Theorem \ref{ccc} which uniquely defines $\mathsf{K}$.
\end{proof}

\subsection{Example}

We calculate the first coefficients of $\mathsf{K}$.
The terms
$$\mathsf{K}_{(1),(1)} = 1, \ \ \mathsf{K}_{(1),\ \widehat{\alpha}\neq (1)}=0\ .$$
have already been established.
More interesting are the coefficients $\mathsf{K}_{(2),\widehat{\alpha}}$.
By Lemma \ref{t789}, we have
$$\mathsf{K}_{(2),(2)}= \frac{1}{iu}\ $$
and the only other non-vanishing coefficients are possibly
$\mathsf{K}_{(2),(1^2)}$ and $\mathsf{K}_{(2),(1)}$.
However, $\mathsf{K}_{(2),(1^2)}$ vanishes by Proposition \ref{t789789}.
A degree 1 calculation below yields
$$\mathsf{K}_{(2),(1)} = \frac{s_1+s_2+s_3}{iu}\ ,$$
so we see
\begin{equation}\label{fgg3}
\widehat{\tau}_{(2)}(\mathsf{p}) = \frac{1}{iu} \tau_{(2)}(\mathsf{p})
+ \frac{s_1+s_2+s_3}{iu} \tau_{(1)}(\mathsf{p})\ .
\end{equation}
Note, $\mathsf{K}_{(2),(1)}$ is symmetric in the $s_i$. 

To prove \eqref{fgg3}, we need only check a single correspondence
(as only the single coefficient $\mathsf{K}_{(2),(1)}$ is unknown).
In \cite{PPstat}, we have already calculated{\footnote{The tangent
weight conventions of \cite{PPstat} differ from the
conventions here by a sign.}}
$$
\bZ_{\mathsf{P}}\Big(\mathsf{Cap};q\ \Big|\
\tau_{(2)}(\mathsf{p}_0)\ \Big| \ 
(1)[\mathsf{p}_\infty]\Big)^\T_{1} = -q\left(\frac{s_1+s_2}{2}\right)\frac{1-q}{1+q}.$$
There is not much difficulty in calculating the
corresponding Gromov-Witten series
\begin{multline*}
\bZ'_{\mathsf{GW}}\Big(\mathsf{Cap};q\ \Big|\
\tau_{(2)}(\mathsf{p}_0)\ \Big| \ 
(1)[\mathsf{p}_\infty]\Big)^\T_{1} =\\ -s_3 u^{-2}+
\frac{1}{u} \frac{d}{du} \left( s_3 \left(\frac{u/2}{\sin(u/2)}\right)^
{-\frac{(s_1+s_2)}{s_3}}\right) \cdot
\left(\frac{u/2}{\sin(u/2)}\right)^
{\frac{(s_1+s_2)}{s_3}}\ 
\end{multline*}
by the Hodge integral methods of \cite{FP,GraberP}.
The descendent is inserted via the dilaton equation which
appears as differentiation of the vertex term. The factor
furthest to the right is the rubber contribution.
The series
\begin{equation*}
\bZ'_{\mathsf{GW}}\Big(\mathsf{Cap};q\ \Big|\
\tau_{(1)}(\mathsf{p}_0)\ \Big| \ 
(1)[\mathsf{p}_\infty]\Big)^\T_{1} =\\ u^{-2}
\end{equation*}
is simple.
After including the $(-q)^{-1}$ and $(-iu)^2$ 
scalings of the 1-leg descendent correspondence
\eqref{gtt12}, we check
\begin{multline*}
\left(\frac{s_1+s_2}{2}\right) \frac{1-q}{1+q} = \\
-\frac{1}{iu} \left( (s_1+s_2)+ 
u \frac{d}{du} \left( s_3 \left(\frac{u/2}{\sin(u/2)}\right)^
{-\frac{(s_1+s_2)}{s_3}}\right) \cdot
\left(\frac{u/2}{\sin(u/2)}\right)^
{\frac{(s_1+s_2)}{s_3}}\right)
\end{multline*}
after $-q=e^{iu}$.

In the above example, the
 stable pairs descendent had been exactly
calculated and the dilaton equation at the vertex
could be
used to handle the Gromov-Witten side.
 While the
stable pairs descendents series are difficult
to calculate, at least methods exist \cite{part1,PPstat}.
At the moment, there is no reasonable way
to calculate the Gromov-Witten descendent series except, 
of course, order by order in $u$.

\section{Capped localization}
\label{cl}

\subsection{Toric geometry}
Let $X$ be a nonsingular toric $3$-fold.
Virtual localization with respect to the
action of the full 3-dimensional torus $\T$ reduces all stable pairs
and Gromov-Witten invariants
of $X$
to local contributions of the vertices and edges of the 
associated toric polytope.   
We will use the regrouped localization procedure 
introduced in \cite{moop} 
with capped vertex and edge contributions.  
The capped vertex and edge terms are equivalent 
building blocks for global toric calculations, 
but are much better behaved.

Let $\Delta$ denote the polytope associated to
$X$. The vertices of $\Delta$ are in bijection
with $\T$-fixed points $X^\T$.
The edges $e$ correspond
to $\T$-invariant curves $$C_e\subset X.$$ The three edges
incident to any vertex carry canonical $\T$-weights ---
the tangent weights of the torus action.

We will consider both compact and noncompact toric
varieties $X$. In the latter case, edges 
may be compact or noncompact.
Every compact edge is incident to two vertices.

\subsection{Capping}
Capped localization expresses the 
$\T$-equivariant stable pairs descendents 
of $X$ as a sum of capped descendent vertex and capped edge
data.

A half-edge $h=(e,v)$
 is a compact edge $e$ together with the
choice of an incident vertex $v$.
A partition assignment
$$h \mapsto \lambda(h)$$
to half-edges is {\em balanced} if the equality
$$|\lambda(e,v)| = |\lambda(e,v')|$$
always holds 
for the two halfs of  $e$.
For a balanced assignment, let $$|e|=|\lambda(e,v)|=|\lambda(e,v')|$$
denote the {\em edge degree}.

The outermost sum in the capped localization formula
runs over all balanced 
assignments of partitions $\lambda(h)$ to
the half-edges $h$ of $\Delta$ satisfying
\begin{equation}\label{ddfff}
\beta = \sum_e |e|\cdot \left[C_e\right] \in H_2(X,\mathbb{Z})\,.
\end{equation}
Such a partition assignment will be called a \emph{capped
marking} of
$\Delta$.
The weight of each capped marking in the localization sum for the
stable pairs descendent partition function equals the product
of three factors:
\begin{enumerate}
\item[(i)] capped descendent vertex contributions,
\item[(ii)] capped edge contributions,
\item[(iii)] gluing terms.
\end{enumerate}

Each vertex determines up to three half-edges specifying the
partitions for the capped vertex.
Each compact edge determines two half-edges specifying
the partitions of the capped edge.
The capped edge contributions (ii) 
and gluing terms (iii) here are {\em exactly} the same as
for the the capped localization formula in \cite{moop}.
Precise formulas are written in Section \ref{exxx}.

The capped localization formula
is easily derived from the standard localization formula (with
roots in \cite{GraberP,MNOP1}).
Indeed, the capped objects are obtained from the
uncapped objects by  rubber integral{\footnote{
Rubber integrals $\langle \lambda \ |\ \frac{1}{1-\psi_\infty} \ |  
\ \mu \rangle ^\sim$ 
arise in the localization formulas for relative
geometries.
See
\cite{part1} for a discussion.}} factors.
The rubber integrals cancel in pairs in capped localization
to yield standard localization.

\label{bgty}

\subsection{Formulas} \label{exxx}

The $\T$-equivariant cohomology of $X$ is generated (after localization)
by the classes of the $\T$-fixed points $X^\T \subset X$.
Let $\alpha$ be a partition with parts $\alpha_1, \ldots, \alpha_\ell$,
and let 
$$\sigma: \{1,\ldots, \ell\} \rightarrow X^\T\ .$$
Let $\mathsf{p}_{\sigma(i)} \in H^*_\T(X,\mathbb{Q})$ denote the
class of the $\T$-fixed point $\sigma(i)$.
We 
consider the 
capped localization formula 
for the $\T$-equivariant stable pairs and Gromov-Witten
descendent partition functions 
\begin{equation}\label{fq2}
\bZ_{\mathsf{P}}\Big(X;q\ \Big| \ \prod_{i=1}^r \tau_{k_i}(\mathsf{p}_{\sigma(i)})\Big)_\beta^\T\ \ , \ \ \ \ 
\bZ'_{\mathsf{GW}}\Big(X;u\ \Big|\ \prod_{i=1}^r \tau_{k_i}(\mathsf{p}_{\sigma(i)})\Big)_\beta^\T\ .
\end{equation}
We will indicate the slight differences bewtween the
formula for stable pairs and stable maps below.

Let $\mathcal{V}$ be the set of vertices of $\Delta$ which
we identify 
with $X^\T$. 
%
For each vertex $v\in \mathcal{V}$,
let $h^{v}_1, h^v_2, h^v_3$ be the associated half-edges{\footnote{
For simplicity, we assume $X$ is projective so each vertex is
incident to 3 compact edges.}}  
with tangent weights $s^v_1,s^v_2,s^v_3$ respectively.
Let $\Gamma_\beta$ be the set of capped markings
 satisfying the degree condition \eqref{ddfff}.
Each $\Gamma \in \Gamma_\beta$  associates  
a partition $\lambda(h)$ to every half-edge $h$. Let
$$|h| = |\lambda(h)|$$
 denote the half-edge degree.

For each $v\in \mathcal{V}$, the assignments $\sigma$ and
$\Gamma$ determine an evaluation of the capped vertex,
$$\mathsf{C}(v,\sigma,\Gamma) = \mathsf{C}\Big(\prod_{i\in \sigma^{-1}(v)}
\tau_{k_i}(\mathsf{p}_v)\ \Big|\
\lambda(h^v_1),
\lambda(h^v_2), \lambda(h^v_3)\Big)\Big|_{s_1= s^v_1, s_2=s^v_2, s_3=s^v_3}.$$
Let
$h^e_1$ and $h^e_2$ be the half-edges
associated to the edge $e$.
The assignment 
$\Gamma$ also determines an evaluation of the capped edge,
$$\mathsf{E}(e,\Gamma) = \mathsf{E}(\lambda(h^e_1),
 \lambda(h^e_2)).$$
The capped edge geometry is discussed 
for in \cite{moop}. 
A gluing factor is specified by $\Gamma$ at each
half-edge $h^v_i\in \mathcal{H}$. For stable pairs
$$\mathsf{G}_{\mathsf{P}}(h^v_i,\Gamma) = 
(-1)^{|h^v_i|-\ell(\lambda(h^v_i))}
\mathfrak{z}(\lambda(h^v_i)) 
\left(\frac{\prod_{j=1}^3 s^v_j}{s^v_i}\right)^{\ell(\lambda(h^v_i))} 
q^{-|h^v_i|}\ $$
where 
 $\zz(\lambda)$ is the order of the
centralizer in the symmetric group of
an element with cycle type $\lambda$.
For Gromov-Witten theory,
$$\mathsf{G}_{GW}(h^v_i,\Gamma) = \mathfrak{z}(\lambda(h^v_i)) 
\left(\frac{\prod_{j=1}^3 s^v_j}{s^v_i}\right)^{\ell(\lambda(h^v_i))} 
u^{2\ell(\lambda(h^v_i))}.$$

The capped localization formula for stable pairs can be written
exactly in the form presented in Section \ref{bgty},
\begin{equation*}
\bZ_{\mathsf{P}}\Big(X,\prod_{i=1}^r \tau_{k_i}(\mathsf{p}_{\sigma(i)})
\Big)^\T_\beta=
\sum_{\Gamma\in \Gamma_\beta}\
\prod_{v\in \mathcal{V}}\ \prod_{e\in \mathcal{E}}\ \prod_{h\in \mathcal{H}}
\mathsf{C}_{\mathsf{P}}(v,\sigma,\Gamma) \
\mathsf{E}_{\mathsf{P}}(e,\Gamma)\
\mathsf{G}_{\mathsf{P}}(h, \Gamma)
\end{equation*}
where the product is over the sets of vertices $\mathcal{V}$, edges
 $\mathcal{E}$, and half-edges $\mathcal{H}$ 
of the polytope $\Delta$. 
Similarly,
\begin{equation*}
\bZ'_{\mathsf{GW}}\Big(X,\prod_{i=1}^r \tau_{k_i}(\mathsf{p}_{\sigma(i)})
\Big)^\T_\beta=
\sum_{\Gamma\in \Gamma_\beta}\
\prod_{v\in \mathcal{V}}\ \prod_{e\in \mathcal{E}}\ \prod_{h\in \mathcal{H}}
\mathsf{C}_{\mathsf{GW}}(v,\sigma,\Gamma) \
\mathsf{E}_{\mathsf{GW}}(e,\Gamma)\
\mathsf{G}_{\mathsf{GW}}(h, \Gamma) \ .
\end{equation*}

An immediate consequence of the
 above capped localization formulas for stable pairs
and Gromov-Witten theories is the implication of Theorem
\ref{aaa} by Theorem \ref{ccc} and the symmetry of $\mathsf{K}$
in the variables $s_i$. Theorem \ref{ccc} is 
applied to the capped descendent vertices on the
right side of the formula. The capped edge correspondences
have already been proven in \cite{moop} (with the
stable pairs case discussed in Section 5 of \cite{mpt}).
Tracing the factors of $q$ and $u$ here is an easy exercise.

\section{Descendent correspondence: 2-leg} \label{fullt}

\subsection{Overview}
Our goal here is to prove Theorem \ref{ccc} in the 2-leg case.
Consider the capped 2-leg descendent vertices
$$\bC_{\mathsf{P}}(\tau_\alpha(\mathsf{p})\ |\ \emptyset,\mu,\lambda), \ \ \
\bC_{\mathsf{GW}}(\tau_\alpha(\mathsf{p})\ |\ \emptyset,\mu,\lambda)\ .$$
Our proof of Theorem \ref{ccc}
 will be by induction on the complexity of the legs.
The descendent insertion $\tau_\alpha(\mathsf{p})$ will be
fixed for the argument.
If $\mu=\emptyset$, we are in the 1-leg case where  
Theorem \ref{ccc} has already been established in Section \ref{ooolll}. 
The 1-leg case will be 
the base of the induction.

Define a partial ordering on pairs of partitions $(\mu,\lambda)$
satisfying the condition 
$(\mu,\lambda)\neq (\emptyset,\emptyset)$
 by the following rules. We say
\begin{equation*}
(\mu,\lambda) \ \triangleright \ (\mu',\lambda')
\end{equation*}
if we have $|\mu| > |\mu'|$.
The proof of 
Theorem \ref{ccc} in the 2-leg case 
is by induction with respect to the partial ordering $\triangleright$.

Along the way, we will also establish basic
properties of the correspondence
matrix $\mathsf{K}$ constructed in
Section \ref{conk}. The following result,
completing the proof of Theorem \ref{ll33}, will be proven in
Section \ref{4444}:

\vspace{10pt}
\noindent {\em The coefficients of 
$\mathsf{K}$ lie in the subring
\begin{equation}\label{cl223}
\Lambda_3((u))\subset \mathbb{Q}(i,s_1,s_2,s_3)((u))
\end{equation}
where $\Lambda_3$ is the ring of symmetric 
{\em polynomials} in $s_1,s_2,s_3$ over the field $\mathbb{Q}[i]$.}

\subsection{$\bA_1$ geometry} \label{angeo}
Let $\zeta$ be a primitive $(n+1)^{th}$ root
of unity, for $ n \geq 0$.
Let the generator of the
cyclic group $\mathbb{Z}_{n+1}$ act on $\com^2$ by
$$
(z_1,z_2)\mapsto  (\zeta\, z_1, \zeta^{-1}z_2)\,.
$$
Let  $\bA_n$ be  the minimal resolution of the quotient
$$
\bA_n \rightarrow \com^2/\mathbb{Z}_{n+1}.
$$
The diagonal $(\C^*)^2$-action on $\com^2$ commutes
with the action of $\mathbb{Z}_n$. As a result, the
surfaces $\bA_n$ are toric.

The surface $\bA_1$ is isomorphic to the total space of
$$\cO(-2) \rarr \Pp\ $$
and admits a toric compactification 
$$\bA_1 \subset \mathbf{P}(\cO+\cO(-2)) = \bF_2$$
by the Hirzebruch surface. 

Let $C\subset \bA_1$ be the 0-section of $\cO(-2)$, and
let $\star,\bullet\in C$ be the $(\com^*)^2$-fixed points.
Let $$\overline{\star} ,\overline{\bullet}\in\bF_2\setminus \bA_1$$
be the $(\com^*)^2$-fixed points lying above $\star,\bullet$
repectively.
We fix our $(\com^*)^2$-action by specifying tangent weights
at the four  $(\com^*)^2$-points:
\begin{eqnarray}
 T_{\star}(\bF_2):\ &  s_1-s_2, \ & \ \ 2s_2   \label{rrtt}\\
 T_{\bullet}(\bF_2):\ &  s_2-s_1, \ &\ \  2s_1 \nonumber\\
 T_{\overline\star}(\bF_2):\ &  s_1-s_2, \ & -2s_2 \nonumber\\
 T_{\overline\bullet}(\bF_2):\ &  s_2-s_1, \ & -2s_1\ . \nonumber
\end{eqnarray}
No tangent weight here is divisible by $s_1+s_2$.

Consider the nonsingular projective toric variety  $\bF_2 \times \Pp$.
The 3-torus 
$$\T=(\com^*)^3$$
 acts on $\bF_2$ as above via the first two factors and
acts on $\Pp$ via the third factor with tangent weights $s_3$ and $-s_3$
at the points $0,\infty\in \Pp$ respectively.
The two $\T$-invariant divisors of $\bF_2 \times \Pp$
$$\bD_0 = \bF_2 \times \{0\}, \ \ \bD_\infty = \bF_2 \times \{\infty\}$$
will play a basic role.
The 3-fold $\bF_2\times \Pp$ has eight $\T$-fixed points which we denote by
$$\star_0,\overline{\star}_0,\bullet_0,\overline{\bullet}_0,
\star_\infty,\overline{\star}_\infty,\bullet_\infty,\overline{\bullet}_\infty \in 
\bF_2\times\Pp $$
where the subscript indicates the coordinate in $\Pp$.

Let $L_0 \subset \bF_2\times \Pp$ be the $\T$-invariant line
connecting $\star_0$ and $\overline{\star}_0$. 
Similarly, let $L_\infty \subset \bF_2\times \Pp$ be the $\T$-invariant line
connecting $\star_\infty$ and $\overline{\star}_\infty$. 
The lines $L_0$ and $L_\infty$
are $\Pp$-fibers of the Hirzebruch surfaces $\bD_0$ and $\bD_\infty$.
We have
$$H_2(\bF_2\times \Pp,\mathbb{Z}) = \mathbb{Z}[C] \oplus
\mathbb{Z} [L_0] \oplus
\mathbb{Z}  [P]$$
where $P$ is the fiber of the projection to $\bF_2$.

\subsection{Integration}
We will find relations which express
$\bC(\tau_{\alpha}(\mathsf{p})\ |\ \emptyset, \mu,\lambda)$ in terms
of inductively treated vertices for stable pairs and Gromov-Witten
theory. The inductive equations will respect the
correspondence claimed in Theorem \ref{ccc}.

Let $\mu'$ be a partition. The relations will
be obtained from
vanishing  invariants of the relative
geometry $\bF_2\times \Pp/\bD_\infty$
in curve class
$$\beta = |\mu|\cdot [C] + (|\lambda|+|\mu'|)\cdot [P] \ . $$
The virtual dimensions of the associated moduli spaces are
\begin{eqnarray*} 
\text{dim}^{vir}\ P_n(\bF_2\times \Pp,\beta) & = &
2|\lambda|+2|\mu'| \ ,\\
\text{dim}^{vir}\ \overline{M}_g'(\bF_2\times \Pp,\beta) & = &
2|\lambda|+2|\mu'| \ .
\end{eqnarray*}

Relative conditions in $\text{Hilb}(\bD_\infty, |\lambda|+|\mu'|)$
are best expressed in terms of the Nakajima basis given by a
$\T$-equivariant cohomology weighted partition of $|\lambda|+|\mu'|$.
We impose
the relative condition determined by the partition
$$\lambda \cup \mu'=\lambda_1+\ldots+\lambda_{\ell(\lambda)}
+\mu_1'+\ldots+\mu'_{\ell(\mu')}$$ 
weighted by
$[\star_\infty]\in H^*_\T(\bD_\infty,\mathbb{Q})$ for the parts of $\lambda$
and  
$[\bullet_\infty]\in H^*_\T(\bD_\infty,\mathbb{Q})$ for the parts
of $\mu'$.
We denote the relative condition by $\mathsf{r}(\lambda,\mu')$.
The same data expresses  relative conditions in Gromov-Witten theory.
After imposing $\mathsf{r}(\lambda,\mu')$,
the virtual dimension drops to
\begin{eqnarray*}
\text{dim}^{vir}\ P_n(\bF_2\times \Pp/\bD_\infty,{\mathsf{r}(\lambda,\mu')}  )_{\beta}
& = & |\lambda|-\ell(\lambda)+ |\mu'|-\ell(\mu')\ , \\
\text{dim}^{vir}\ \overline{M}_g'(\bF_2\times \Pp/\bD_\infty,
{\mathsf{r}(\lambda,\mu')})_\beta
& =& |\lambda|-\ell(\lambda)+ |\mu'|-\ell(\mu')\ .
\end{eqnarray*}

To define an equivariant integral, we specify the descendent
insertion by 
$$\tau_{\alpha}(\mathsf{[\star_0]}) =
\tau_{\alpha_1-1}(\mathsf{[\star_0]})\cdots
\tau_{\alpha_{\ell(\alpha)}-1}(\mathsf{[\star_0]})\ .$$
The descendent insertion imposes $|\alpha|+\ell(\alpha)$ conditions.
Therefore, the integrals
\begin{equation}\label{hhhooo}
\int_{ [P_n(\bF_2\times \Pp/\bD_\infty,{\mathsf{r}(\lambda,\mu')})_{\beta}]^{vir}}
\tau_{\alpha}(\mathsf{[\star_0]}), \  
\int_{ [\overline{M}'_{g,\ell(\alpha)}(\bF_2\times \Pp/\bD_\infty,
{\mathsf{r}(\lambda,\mu')})_{\beta}]^{vir}}
\tau_{\alpha}(\mathsf{[\star_0]})
\end{equation}
viewed as $\T$-equivariant push-forwards to a point,
both have dimension 
$$|\lambda|-\ell(\lambda)+ |\mu'|-\ell(\mu')-|\alpha|-\ell(\alpha)\ .$$
We conclude the following result.

\begin{Proposition} 
If $|\mu'|-\ell(\mu')>|\alpha|+\ell(\alpha)$, then
the $\T$-equivariant integrals 
\eqref{hhhooo} vanish for all Euler characteristics $n$ and 
genera $g$. \label{ggh2}
\end{Proposition}

\subsection{Relations}\label{2legrel}
We consider first the stable pairs case. 
Define the $\T$-equivariant series 
\begin{equation*}
\bZ_{\mathsf{P}}\Big(\alpha,\lambda,\mu'\Big)_\beta
=
\sum_{n}  q^n
\int_{ [P_n(\bF_2\times \Pp/\bD_\infty,{\mathsf{r}(\lambda,\mu')})_{\beta}]^{vir}}
\tau_{\alpha}([\star_0]) 
\end{equation*}
obtained from the stable pairs integrals \eqref{hhhooo}.
By Proposition \ref{ggh2}, the series 
$\bZ_{\mathsf{P}}\Big(\alpha,\lambda,\mu'\Big)_\beta$
vanishes identically if $|\mu'|-\ell(\mu')>|\alpha|+\ell(\alpha)$.
We will calculate the left side of
\begin{equation}\label{ttx}
\bZ_{\mathsf{P}}\Big(\alpha,\lambda,\mu'\Big)_\beta=0
\end{equation}
by capped localization to obtain a relation constraining
the stable pairs capped descendent vertices.

The stable pairs theory of the relative geometry $\bF_2\times \Pp/\bD_\infty$ admits a
capped localization formula. 
Over $0\in \Pp$, capped descendent vertices occur as in
the cappled localization formula of Section \ref{exxx}.
Over $\infty \in \Pp$,
capped rubber terms for $\T$-equivariant localization in the relative
geometry arise.
Capped rubber
is  discussed in Section 3.4 of \cite{moop}.
Since all our descendent
insertions lie over $0\in \Pp$, 
our capped rubber has the same definition as
the capped rubber of \cite{moop}.

By the curve choice $\beta$ and the relative constraints
${\mathsf{r}(\lambda,\mu')}$, the only capped rubber
contributions of $\bF_2 \times \Pp/\bD_\infty$ which arise
over $\infty\in \Pp$ in the 
$\T$-equivariant localization formula for
$\bZ_{\mathsf{P}}\Big(\alpha,\lambda,\mu'\Big)_\beta$
lie in 
$$\bA_1 \times \Pp\subset \bF_2 \times \Pp\ .$$
The capped rubber
contributions of $\bA_1 \times \Pp/\bD_\infty$
are proven to satisfy the GW/DT correspondence in 
Lemma 6 of \cite{moop} relying on the results of \cite{mo1,mo2}.
See Section 5 of \cite{mpt} for GW/Pairs correspondence
for the $\bA_1$ capped rubber.

We now analyze the capped localization 
of $\bZ_{\mathsf{P}}\Big(\alpha,\lambda,\mu'\Big)_\beta$
over $0\in \Pp$. 
A term in the capped localization formula is said to be
{\em principal} if not all the capped descendent
vertices which arise are lower than $(\mu,\lambda)$ in  the
partial ordering $\triangleright$.
Our first task now is to
identify  the principal terms. 

First consider the descendent insertions.
The descendents 
$$\tau_{\alpha_1-1}([\star_0])\cdots\tau_{\alpha_{\ell(\alpha)}-1}([\star_0])$$
all lie on $\star_0$.
Hence, the only capped vertex with non-trivial descendents
is $\star_0$. 
The tangent weights at $\star_0$ are 
\begin{equation}\label{l445}
s_1-s_2,\ 2s_2,\ s_3
\end{equation}
where the first two lie along $\bD_0$. For the
capped vertices occuring at $\star_0$, the weights
\eqref{l445} are substituted
$$\bC_{\mathsf{P}}( \tau_{\alpha}([\star_0]) \ |\ \emptyset,\lambda^{(2)}, \lambda^{(3)})
=\bC_{\mathsf{P}}( \tau_{\alpha}(\mathsf{p}) \ |\ \emptyset,\lambda^{(2)}, \lambda^{(3)})
\Big|_{s_1=s_1-s_2, \ s_2=2s_2, \ s_3=s_3}$$
into the standard capped vertex defined in Section \ref{defcap}.
An equivariant vertex with no descendents occurs at
$\bullet_0$. 
By the choice of $\beta$ and ${\mathsf{r}(\lambda,\mu')}$,
no 
vertices can only occur at $\overline{\star}_0$ and $\overline{\bullet}_0$.

Next consider the edge degree $d$ of $C$ over $0\in \Pp$ in
the capped localization formula.
If $d<|\mu|$, then the capped descendent vertex at $\star_0$
is 
lower than $(\mu,\lambda)$
in the partial ordering $\triangleright$. 
We restrict ourselves to the principal terms where $d=|\mu|$.

Since all of $|\mu|\cdot [C]$ occurs over $0\in \Pp$, the
rubber over $\infty\in\Pp$ is all 1-leg.
The relative conditions are determined by 
$\lambda$ with weights $[\star_\infty]$ and
$\mu'$ with weights $[\bullet_\infty]$.
In the principal terms of the
 capped localization of
 \eqref{ttx}, precisely the following set of
capped 2-leg descendent vertices occur at $\star_0$:
\begin{equation}\label{bxs}
\Big\{ \ \bC_{\mathsf{P}}(\tau_{\alpha}([\star_0])\ |\ 
\emptyset, \widehat{\mu},\lambda)\ \Big| \ 
|\widehat{\mu}|=|\mu| \ \Big\}.
\end{equation}
The principal terms arise as displayed in Figure \ref{A1cap}.
In addition to the vertex
$\bC(\alpha|{\lambda},\widehat{\mu},\emptyset)$ at $\star_0$,
there is a capped edge with partitions
$$|\widehat{\mu}|=|\widehat{\mu}'|$$
along the curve $C$ over $0\in \Pp$.
Finally, there is a
capped 2-leg vertex with no descendents at $\bullet_0$ with outgoing
partitions $\widehat{\mu}'$ and $\mu'.$

\vspace{10pt}
\psset{unit=0.3 cm}
\begin{figure}[!htbp]
  \centering
   \begin{pspicture}(0,0)(19,20)
   \psline(4,4)(16,4)
   \psline(4,16)(16,16)
   \psline(0,0)(4,4)(4,16)(0,20)
   \pszigzag[coilarm=0.1,coilwidth=0.5,linearc=0.1](12,20)(16,16)
   \pszigzag[coilarm=0.1,coilwidth=0.5,linearc=0.1](16,4)(16,16)
   \pszigzag[coilarm=0.1,coilwidth=0.5,linearc=0.1](12,0)(16,4)
   \psline[doubleline=true](3,7)(4,8)(5,8)
   \psline[doubleline=true](3,13)(4,12)(5,12)
   \rput[l](2.3,16){$\star_0$}
   \rput[l](2.3,4){$\bullet_0$}
   \rput[l](17,18){$\lambda$}
   \rput[l](16.2,16){$\star_\infty$}
   \rput[l](16.2,4){$\bullet_\infty$}
   \rput[l](17,2){$\mu'$}
   \rput[c](6,8){$\widehat{\mu}'$}
   \rput[c](6,12){$\widehat{\mu}$}
   \end{pspicture}
 \caption{Principal terms}
  \label{A1cap}
\end{figure}
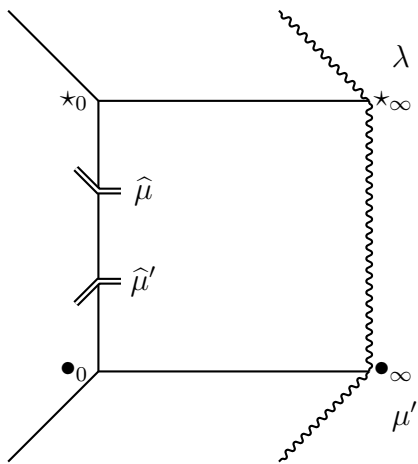

The system of equations \eqref{ttx} as the partition $\mu'$ varies has
 unknowns \eqref{bxs} parameterized
by partitions of $|\mu|$. However, the number of equations
is infinite. The induction step is established if
the set of equations as $\mu'$ varies subject to the condition
$$|\mu'|-\ell(\mu')>|\alpha|+\ell(\alpha)$$
has maximal 
rank
(equal to the number 
of partitions of size $|\mu|$) with respect to the unknowns \eqref{bxs}.

\subsection{Maximal rank} \label{maxer}
The capped edge matrix along $C$ has maximal rank
\cite{moop}. The main difficulty is to prove 
the matrix of capped 2-leg vertices
$$\bC_{\mathsf{P}}( 1\ |\emptyset, \widehat{\mu}',\mu')$$
has maximal rank when $\widehat{\mu}'$ varies
among partitions of size $|\mu|$ and  $\mu'$ varies among
the infinite set of partitions satisfying
$$|\mu'|-\ell(\mu')>|\alpha|+\ell(\alpha)\ .$$

\begin{Proposition} 
\label{icecream}
For any postive integers $d$ and $N$, the matrix
$$\bC_{\mathsf{P}}( 1\ |\emptyset, \delta,\nu)$$
determined as $\delta$ varies among partition of $d$ and
$\nu$ varies among partitions 
satisfying
$$|\nu|-\ell(\nu)\geq N-1 $$
is of maximal rank.
\end{Proposition}

\begin{proof}
We may prove the maximal rank condition after the 
\emph{topological
vertex} specialization
$$
s_1 + s_2 + s_3 =0\, .
$$
The
capped vertex is related to the standard uncapped
vertex by invertible capped rubbers.
The uncapped vertex
may be evaluated directly, see
\cite{MNOP1,ORV}.
Up to further invertible factors, the matrix to consider
becomes
$$
\sum_{\eta} s_{\delta^{t}/\eta}(q^\rho)\, s_{\nu/\eta}(q^{\rho})
$$
where $s_{\delta^t/\eta}$ and $s_{\nu/\eta}$
 are skew Schur functions evaluated
at
$$
q^\rho = (q^{-1/2}, q^{-3/2}, \dots) \,.
$$

As $\delta$ and $\eta$ vary over all partitions{\footnote{$\eta$
is permitted here to be empty.}} of 
size at $d$ and at most $d$ respectively, the matrix of skew Schur functions
$s_{\delta^t/\eta}$ is of maximal rank.  
We are thus reduced to proving the matrix
\begin{equation}\label{p349}
\Big[s_{\nu/\eta}(q^\rho)\Big]\,,
\quad |\nu|-\ell(\nu)\geq N-1\,,\  |\eta| \leq d\,,
\end{equation}
has maximal rank.

We select a square minor from \eqref{p349} by the following
construction.
For every partition 
$$\eta=(e_1\geq e_2\geq \ldots \geq e_{\ell(\eta)})$$ of size at most $d$,
define
$$\eta^+=(e_1+N \geq e_2\geq \ldots \geq e_{\ell(\eta)})$$
In case $\eta=\emptyset$, then $\eta^+=(N)$ has length 1.
As $\eta$ varies among partitions of size at most $d$,
$\eta^+$ varies among partitions which satisfy
$$|\eta^+|-\ell(\eta^+)\geq N-1$$
with equality achieved for $\eta=\emptyset$.
We will prove the matrix
\begin{equation}\label{pp349}
\Big[s_{\nu^+/\eta}(q^\rho)\Big]\,,
\quad |\nu|,  |\eta| \leq d\,,
\end{equation}
is invertible.

We define a partial ordering $\geq$ on partitions $\eta$ of size at most $d$
by the following rule. Let $\eta_-$ be the partition obtained
by removing the largest part of $\eta$, 
$$\eta_-=(e_2 \geq \ldots \geq e_{\ell(\eta)})\ .$$
In case $\eta=\emptyset$, then $\eta_-=\emptyset$.
Define $\eta\geq \widetilde{\eta}$ if
$$\eta_- \supset \widetilde{\eta}_-,$$
or equivalently,
$$ e_2\geq \widetilde{e}_2, \ e_3\geq \widetilde{e}_3, 
\ e_4\geq \widetilde{e}_4, \ldots\ .$$

Unwinding the definitions, we immediately find the following
property:
$$s_{\nu^+/\eta}(q^\rho)=0, \ \ \ \text{unless}\ \ \nu\geq \eta\ $$
since the skew Schur function vanishes unless
$\eta \subset \nu^+$.
We conclude the matrix \eqref{pp349} is block triangular
with respect to the ordering $\geq$.

For any matrix with a block triangular structure with 
respect to a partial ordering, invertibility is equivalent
to the invertibility of the blocks. 
In the case at hand, a block is specified by a partition
$$\theta = (t_1\geq t_2\geq \cdots \geq t_{\ell(\theta)})$$
satisfying $d-|\theta|\geq t_1$.
The block corresponding to $\theta$ is indexed by
the
partitions
$$P_\theta= \{ \eta\ | \ \eta_- = \theta \ \}\ .$$
The cardinality of $P_\theta$ is $M+1$ where
$$M= d-|\theta|- t_1\ .$$
In fact, the elements of $P_\theta$ are  simply the
partitions
$$( t_1,t_1,t_2,\ldots,t_{\ell(\theta)}),\ 
( t_1+1,t_1,t_2,\ldots,t_{\ell(\theta)}), \ldots,
( t_1+M,t_1,t_2,\ldots,t_{\ell(\theta)})\ .$$
The associated block is 
$$B_\theta = \Big[s_{\nu^+/\eta}(q^\rho)\Big]\,,
\quad \nu,\eta\in P_\theta \ .$$

To proceed further, we recall the definition of 
the skew Schur functions,
\begin{equation}\label{ksks}
s_{\lambda/\mu} = \text{det}(h_{\lambda_i-\mu_j+j-i})\ , 
\end{equation}
where $h_k$ is the complete symmetric function of degree $k$. 
If $\mu=\emptyset$,
$$s_{\lambda/\emptyset} = \text{det}(h_{\lambda_i+j-i})\  $$
is the standard Schur function associated to $\lambda$.
The elements of $P_\theta$ are 
$$\{t_1+i\}\cup \theta$$ for
$0\leq i \leq M$. Expanding definition \eqref{ksks}, we see
$$s_{ \{t_1+i+N\}\cup \theta\  /\ \{t_1+j\}\cup \theta} = h_{N+i-j}$$
Therefore, we can write 
the determinant of the block as 
\begin{equation}\label{lucky9}
\text{det}\ B_\theta =  \text{det}\Big(h_{N+i-j}(q^\rho)\Big)
\end{equation}
where $0\leq i,j \leq M$ on the right.
Fortunately, we recognize
the determinant \eqref{lucky9} as the Schur function
$s_{(N,\ldots,N)}(q^\rho)$
associated to the partition $(N,\ldots, N)$ with
$M+1$ parts equal to $N$.
The evaluation $q^\rho$ does not vanish on the Schur
functions.

\end{proof}

\subsection{Proof of Theorem \ref{ccc} in the 2-leg case}
%
We have already seen the integration relations
determine $\bC_{\mathsf{P}}(\tau_\alpha(\mathsf{p})\ |\ \emptyset,\mu,\lambda)$
by induction on the complexity of the legs.
We will now study the parallel integration relations in
Gromov-Witten theory. Our goal is to determine 
$
\bC_{\mathsf{GW}}(\widehat{\tau}_\alpha(\mathsf{p})\ |\ \emptyset,\mu,\lambda)
$
by the {\em same} induction and
 compatible with the 
correspondence of Theorem \ref{ccc} for capped 2-leg 
descendent vertices,

We will consider integration relations from Gromov-Witten theory
for
\begin{multline}\label{kqq2}
\bZ_{\mathsf{GW}}\Big(\alpha,\lambda,\mu'\Big)_\beta= \\
\sum_g u^{2g-2}
\sum_{\widehat{\alpha}\in {\mathcal{P}_{|\alpha|}}} \mathsf{K}_{\alpha,\widehat{\alpha}}(s_1-s_2,2s_2,s_3)
\int_{ [\overline{M}'_{g,\ell(\widehat{\alpha})}(\bF_2\times \Pp/\bD_\infty,{\mathsf{r}(\lambda,\mu')})_{\beta}]^{vir}}
\tau_{\widehat{\alpha}}(\mathsf{[\star_0]}).
\end{multline}
The first issue to confront is the broken symmetry in the
definition of $\mathsf{K}$ in Section \ref{conk}.
While $\mathsf{K}$ is symmetric in $s_1$ and $s_2$, the
variable $s_3$ is treated differently. We orient $\mathsf{K}$
in \eqref{kqq2} by setting the $s_3$ direction lie
along $\Pp$ (which, conveniently, is also $s_3$ in the conventions of 
Section \ref{angeo}).

The formal analysis of the Gromov-Witten relations is
identical to the above stable pairs analysis.
For fixed $\alpha$, there are only finitely many $\widehat{\alpha}$
which occur of the right side of 
$$\widehat{\tau}_\alpha= 
\sum_{\widehat{\alpha}\in {\mathcal{P}_{|\alpha|}}}
{\mathsf K}_{\alpha,\widehat{\alpha}}
\, \tau_{\widehat{\alpha}}\ .$$ 
In order to make the integration relations compatible
with Theorem \ref{ccc}, we must consider
the matrix of capped 2-leg vertices
\begin{equation}\label{rtt43}
\bC_{\mathsf{GW}}( 1\ |\emptyset, \widehat{\mu}',\mu')
\end{equation}
when $\widehat{\mu}'$ varies
among partitions of size $|\mu|$ and  $\mu'$ varies among
the infinite set of partitions satisfying
$$|\mu'|-\ell(\mu')>\text{Max} \Big\{ \ |\widehat{\alpha}|
+\ell(\widehat{\alpha})\ \Big| \ \mathsf{K}_{\alpha,\widehat{\alpha}}\neq 0
\ \Big\}\ .$$
As Proposition \ref{icecream} still applies{\footnote{The 
correspondence for capped stable pair and Gromov-Witten vertices
{\em without} descendents is used here.}}, the matrix \eqref{rtt43}
is of maximal rank.

The inductive determination of the 2-leg descendent
vertex via 
the integration relations therefore respects
the correspondence of Theorem \ref{ccc}. \qed

\subsection{Proof  of  Theorem \ref{ll33}} \label{4444}
Since the homogeneity assertion
was established in Section \ref{conk},
the inclusion \eqref{cl223} is the only part of Theorem \ref{ll33}
which remains to be proven .

The base of the induction in the proof of Theorem \ref{ccc}
 is the $\mu=\emptyset$ case of 1-leg established in Section \ref{ooolll}.
The orientation of $\mathsf{K}$ in \eqref{kqq2} is
perfect for the base case (as the 1-leg direction is then
along $P$).
Applying the induction, we prove Theorem 
\ref{ccc} for the case where $\mu$ is arbitrary and $\lambda=\emptyset$.
In other words, the 1-leg correspondence holds for the
{\em same} matrix 
$\mathsf{K}$ when $s_2$ and $s_3$ are interchanged!
By the uniqueness statement for the 1-leg descendent 
correspondence in Section \ref{conk},
we conclude
$\mathsf{K}$ is symmetric.

The $u^k$ coefficients
of $\mathsf{K}$ are polynomials in $s_3$ (having only poles
at along $s_1$ and $s_2$). 
The proof 
is obtained from two obsevations. First, we have
basic $\mathbf{T}$-equivariant
proper maps from the stable pairs and stable maps spaces
to the symmetric product of $\com^2$,
$$
 P_n(\mathsf{Cap}|\, \lambda)_d \rightarrow \text{Sym}^d(\com^2), $$
$$
 \overline{M}_{g,\ell}(\mathsf{Cap}|\, \lambda)_d \rightarrow 
\text{Sym}^d(\com^2).$$
see Lemma 1 of \cite{part1}.
The $\mathbf{T}$-action of the third torus factor
corresponding to $s_3$ on $\text{Sym}^d(\com^2)$ is trivial.
Pushing-forward the integrals in both cases shows the
$u^k$ coefficients of the 
capped descendent invariants are polynomial in $s_3$.
The matrix $\mathsf{K}$ is obtained in Section \ref{conk}
from the matrices of 1-leg capped descendents after inversion and
product. The second observation is that the determinants of the 1-leg capped 
descendent matrices (see the proofs of
 Lemmas \ref{lsp} and \ref{lgw}) have no
$s_3$ dependence. Hence, the $u^k$ coefficients of $\mathsf{K}$ 
are polynomials in $s_3$.

By the symmetry, the $u^k$ coefficients
of $\mathsf{K}$ are polynomial in all the variables
$s_1$, $s_2$ and $s_3$. 
\qed

\section{Descendent correspondence: 3-leg} \label{fulltt}

\subsection{Overview}
We now prove Theorem \ref{ccc} in the 3-leg case.
Consider the capped 3-leg descendent vertices
$$\bC_{\mathsf{P}}(\tau_\alpha(\mathsf{p})\ |\ \nu,\mu,\lambda), \ \ \
\bC_{\mathsf{GW}}(\tau_\alpha(\mathsf{p})\ |\ \nu,\mu,\lambda)\ .$$
Our proof of Theorem \ref{ccc}
 will again be by induction on the complexity of the legs.
The descendent insertion $\tau_\alpha(\mathsf{p})$ will be
fixed for the argument.
If $\nu=\emptyset$ or $\mu=\emptyset$, we are in the 2-leg case{\footnote{Since
the symmetry of $\mathsf{K}$ has been
proven, all the 2-leg cases are equivalent.}} where  
Theorem \ref{ccc} has been established in Section \ref{fullt}. 
The 2-leg case will be 
the base of the induction.

Define a partial ordering on pairs of partitions $(\nu,\mu,\lambda)$
satisfying the condition 
$(\nu,\mu,\lambda)\neq (\emptyset,\emptyset)$
 by the following rules. We say
\begin{equation*}
(\nu,\mu,\lambda) \ \triangleright \ (\nu',\mu',\lambda')
\end{equation*}
if we have $|\nu|+|\mu| > |\nu'|+|\mu'|$.
The proof of 
Theorem \ref{ccc} in the 3-leg case 
is by induction with respect to the partial ordering $\triangleright$.

The argument for descendent correspondence for the capped
3-leg vertex closely follows the 2-leg case. The main
difference is to replace $\bA_1$
geometry by $\bA_2$ geometry. 

\subsection{$\bA_2$ geometry} \label{angeo2}
Let $\bA_2\subset \bF$ 
be any nonsingular projective 
toric compactification.
We will only be interested in the 
two $(-2)$-curves of $\bA_2$,
$$C,\CC \subset \bA_2\ .$$
No other curves of $\bF$ will play a role in the
construction.

Let $\bbullet,\star,\bullet\in \bA_2$ be the $(\com^*)^2$-fixed points.
The curve $\CC$ connects $\bbullet$ to $\star$ and $C$ connects
$\star$ to $\bullet$.
The other $(\com^*)^2$-fixed points in $\bF\setminus \bA_2$
will not play an important role.

Consider the nonsingular projective toric variety  $\bF \times \Pp$.
The 3-torus 
$$\T=(\com^*)^3$$
 acts on $\bF$ via the first two factors and
acts on $\Pp$ via the third factor with tangent weights $s_3$ and $-s_3$
at the points $0,\infty\in \Pp$ respectively.
Let
$$\bD_0 = \bF \times \{0\}, \ \ \bD_\infty = \bF \times \{\infty\}$$
be $\T$-invariant divisors of $\bF_2 \times \Pp$.
The 3-fold $\bF\times \Pp$ has six important
 $\T$-fixed points which we denote by
$$\bbullet_0,\star_0,\bullet_0,
\bbullet_\infty,\star_\infty,\bullet_\infty
\in
\bF\times\Pp $$
where the subscript indicates the coordinate in $\Pp$.

Let $L_\infty \subset \bF\times \Pp$ be the $\T$-invariant line
connecting $\star_\infty$ to $(\bF\setminus \bA_2)_\infty$. 
We have
$$H_2(\bF\times \Pp,\mathbb{Z}) \supset \mathbb{Z}[C] \oplus
\mathbb{Z} [\CC] \oplus
\mathbb{Z}  [P]$$
where $P$ is the fiber of the projection to $\bF$.

\subsection{Integration}
We will find relations which express
$\bC(\tau_{\alpha}(\mathsf{p})\ |\ \nu, \mu,\lambda)$ in terms
of inductively treated vertices for stable pairs and Gromov-Witten
theory. The inductive equations will respect the
correspondence claimed in Theorem \ref{ccc}.

Let $\nu'$ and $\mu'$ be partitions. The relations will
be obtained from
vanishing  invariants of the relative
geometry $\bF \times \Pp/\bD_\infty$
in curve class
$$\beta = |\nu|\cdot [C] + |\mu|\cdot \CC + (|\lambda|+|\mu'|+|\nu'|)\cdot [P] \ . $$
The virtual dimensions of the associated moduli spaces are
\begin{eqnarray*} 
\text{dim}^{vir}\ P_n(\bF\times \Pp,\beta) & = &
2|\lambda|+2|\mu'|+2|\nu'| \ ,\\
\text{dim}^{vir}\ \overline{M}_g'(\bF\times \Pp,\beta) & = &
2|\lambda|+2|\mu'|+2|\nu'| \ .
\end{eqnarray*}

Relative conditions in $\text{Hilb}(\bD_\infty, |\lambda|+|\mu'|+|\nu'|)$
in the Nakajima basis are given by a
$\T$-equivariant cohomology weighted partition of $|\lambda|+|\mu'|+|\nu'|$.
We impose
the relative condition determined by the partition
$$\lambda \cup \mu'\cup \nu'=\lambda_1+\ldots+\lambda_{\ell(\lambda)}
+\mu_1'+\ldots+\mu'_{\ell(\mu')}
+\nu_1'+\ldots+\nu'_{\ell(\nu')}
$$ 
weighted by
$[\star_\infty]\in H^*_\T(\bD_\infty,\mathbb{Q})$ for the parts of $\lambda$, 
$[\bbullet_\infty]\in H^*_\T(\bD_\infty,\mathbb{Q})$ for the parts
of $\mu'$, and
$[\bullet_\infty]\in H^*_\T(\bD_\infty,\mathbb{Q})$ for the parts
of $\nu'$.
We denote the relative condition by $\mathsf{r}(\lambda,\mu',\nu')$.
After imposing $\mathsf{r}(\lambda,\mu',\nu')$,
the virtual dimension drops to
\begin{eqnarray*}
\text{dim}^{vir}\ P_n(\bF_2\times \Pp/\bD_\infty,{\mathsf{r}}  )_{\beta}
& = & |\lambda|-\ell(\lambda)+ |\mu'|-\ell(\mu')+|\nu'|-\ell(\nu')\ , \\
\text{dim}^{vir}\ \overline{M}_g'(\bF_2\times \Pp/\bD_\infty,
{\mathsf{r}})_\beta
& =& |\lambda|-\ell(\lambda)+ |\mu'|-\ell(\mu')+|\nu'|-\ell(\nu')\ .
\end{eqnarray*}

To define an equivariant integral, we specify the descendent
insertion by 
$$\tau_{\alpha}(\mathsf{[\star_0]}) =
\tau_{\alpha_1-1}(\mathsf{[\star_0]})\cdots
\tau_{\alpha_{\ell(\alpha)}-1}(\mathsf{[\star_0]})\ .$$
The descendent insertion imposes $|\alpha|+\ell(\alpha)$ conditions.
Therefore, the integrals
\begin{equation}\label{hhhooo5}
\int_{ [P_n(\bF\times \Pp/\bD_\infty,{\mathsf{r}})_{\beta}]^{vir}}
\tau_{\alpha}(\mathsf{[\star_0]}), \ \ \
\int_{ [\overline{M}'_{g,\ell(\alpha)}(\bF\times \Pp/\bD_\infty,
{\mathsf{r}})_{\beta}]^{vir}}
\tau_{\alpha}(\mathsf{[\star_0]})
\end{equation}
viewed as $\T$-equivariant push-forwards to a point,
both have dimension 
$$|\lambda|-\ell(\lambda)+ |\mu'|-\ell(\mu')+|\nu'|-\ell(\nu')-|\alpha|-\ell(\alpha)\ .$$
We conclude the following result.

\begin{Proposition} 
If either $|\mu'|-\ell(\mu')$ or $|\nu'|-\ell(\nu')$
exceeds $|\alpha|+\ell(\alpha)$, then
the $\T$-equivariant integrals 
\eqref{hhhooo5} vanish for all Euler characteristics $n$ and genera $g$. \label{ggh33}
\end{Proposition}

\subsection{Proof of Theorem \ref{ccc}}
Define the $\T$-equivariant series 
\begin{equation*}
\bZ_{\mathsf{P}}\Big(\alpha,\lambda,\mu',\nu'\Big)_\beta
=
\sum_{n}  q^n
\int_{ [P_n(\bF\times \Pp/\bD_\infty,{\mathsf{r}(\lambda,\mu',\nu')})_{\beta}]^{vir}}
\tau_{\alpha}([\star_0]) 
\end{equation*}
obtained from the stable pairs integrals \eqref{hhhooo5}.
On the Gromov-Witten side, we consider the integrals
\begin{multline*}
\bZ_{\mathsf{GW}}\Big(\alpha,\lambda,\mu',\nu'\Big)_\beta= \\
\sum_g u^{2g-2}
\sum_{\widehat{\alpha}\in {\mathcal{P}_{|\alpha|}}} \mathsf{K}_{\alpha,\widehat{\alpha}}
\int_{ [\overline{M}'_{g,\ell(\widehat{\alpha})}
(\bF\times \Pp/\bD_\infty,{\mathsf{r}(\lambda,\mu',\nu')})_{\beta}]^{vir}}
\tau_{\widehat{\alpha}}(\mathsf{[\star_0]}).
\end{multline*}
Since we have already established the symmetry of
$\mathsf{K}$ in the variable $s_i$,
we no longer need to worry about the orientation of $\mathsf{K}$.
When both 
$|\mu'|-\ell(\mu')$ and $|\nu'|-\ell(\nu)$
exceed
$$\text{Max} \Big\{ \ |\widehat{\alpha}|
+\ell(\widehat{\alpha})\ \Big| \ \mathsf{K}_{\alpha,\widehat{\alpha}}\neq 0
\ \Big\}\ ,$$
Proposition \ref{ggh33} implies 
\begin{equation}\label{ttxxxx}
\bZ_{\mathsf{P}}\Big(\alpha,\lambda,\mu',\nu'\Big)_\beta=0, \ \ \ 
\bZ_{\mathsf{GW}}\Big(\alpha,\lambda,\mu',\nu'\Big)_\beta=0\ .
\end{equation}

The inductive analysis of the capped localization of stable pairs
and Gromov-Witten relations \eqref{ttxxxx} exactly follows the
treatment given in Section \ref{2legrel} for the 2-leg case.
The outcome is an inductive determination of the 
capped descendent 3-leg vertices in terms of the capped
descendent 2-leg vertices which respects the correspondence
of Theorem \ref{ccc}.
The maximal rank result of Proposition \ref{icecream} is used twice.

\section{First consequences}\label{ffirr}

\subsection{Descendents over $0\in \Pp$}
Since Theorems \ref{ll33} and \ref{ccc} 
 together imply Theorem \ref{aaa}, we have proven 
the matrix $\mathsf{K}$ determines 
a Gromov-Witten/Pairs
descendent correspondence for all nonsingular quasi-projective 
toric 3-folds $X$.

Recall the surfaces $\bA_n$ defined in Section \ref{angeo}.
As before, let $$\bD_\infty \subset \bA_n \times \Pp$$
be the fiber over $\infty \in \Pp$.
The relative geometry
$\bA_n\times \Pp/ \bD_\infty$ is equivariant with respect
to the action of 3-dimensional torus 
$$\T= T \times \com^*$$
where the 2-dimensional torus $T$ acts on $\bA_n$ and
$\com^*$ acts on $\Pp$ with fixed points $0,\infty \in \Pp$.
The weights associated to $T$ are $s_1,s_2$, and the
weight associated to $\com^*$ is $s_3$.
Let 
$${p}_0, \ldots, {p}_{n} \in \bD_0$$
be the $\T$-fixed points lying over $0\in \Pp$.

Using the correspondence of Theorem \ref{ccc} for
capped descendent vertices over $0\in \Pp$ and the
correspondence for capped $\bA_n$-rubber (Lemma 6 of \cite{moop}
together with Section 5 of \cite{mpt}), we 
immediately conclude the following result.

\begin{Proposition} \label{hvv2}
For the $\T$-equivariant relative geometry $\bA_n\times\Pp/\bD_\infty$, 
after the variable change $-q=e^{iu}$, 
we have
\begin{multline*}
(-q)^{-d_\beta/2}\,
\bZ_{\mathsf{P}}
\Big( \bA_n\times\Pp/\bD_\infty   ;q\ \Big|\
\prod_{j=0}^n \tau_{\alpha^{(j)}}(\mathsf{p}_j)\, \Big| \mu \Big)^{\T}_\beta
=\\
(-iu)^{d_\beta+\ell(\mu)-|\mu|}\,  \bZ'_{\mathsf{GW}}
\Big( \bA_n\times\Pp/\bD_\infty;u\ \Big|\
\prod_{j=0}^n \widehat{\tau}_{\alpha^{(j)}}(\mathsf{p}_j) \, \Big| \mu  \Big )^\T_\beta
\,. 
\end{multline*}
where $\alpha^{(0)}, \ldots, \alpha^{(n)}$
are  partitions, $\mu$ is a $\T$-equivariant
relative condition along $\bD_\infty$, and $\beta \in H_2(\bA_n \times \Pp,
\mathbb{Z})$ is any curve class.
\end{Proposition}

Since the coefficients of the 
matrix $\mathsf{K}$ have no poles in the $s_i$ by Theorem \ref{ll33}, we
can restrict the correspondence of Proposition \ref{hvv2} to $s_3=0$
(so long as the relative conditions $\mu$ have no denominators in $s_3$). 
We will typically take $\mu$ to have no $s_3$ dependence at
all.
As an application of Proposition \ref{hvv2},
we will prove Theorem \ref{yaya38}.

\subsection{Proof of Theorem  \ref{yaya38}}\label{yayaya}
Let $\mathsf{L}\in H_2(\bA_n \times \Pp, \mathbb{Z})$ be the
class of the factor $\Pp$.
For the 3-fold $\bA_n \times \Pp$, 
we can uniquely write a
curve class $\beta$ as
$$\beta = d \mathsf{L} + \mathsf{F},$$
where $\mathsf{F}\in H_2(\bA_n,\mathbb{Z})$ is a fiber class and
$$ d_\beta= \int_\beta c_1(\bA_n\times \Pp) = 2d\ .$$
For fixed $d$, we define a matrix square matrix 
$$M_{\mathsf{P},d}(  {\ \alpha^{(0)}, \ldots, \alpha^{(n)} \ | \ \mu} )$$
with columns indexed by $(n+1)$-tuples of
partitions
$$\alpha^{(0)}, \ldots, \alpha^{(n)}, \ \ \ \sum_{j=0}^n |\alpha|^{(j)} =d,$$
rows indexed by 
relative conditions $\mu$ in the Nakajima basis{\footnote{The
Nakajima basis here is given by assigning
a partition to each $\mathsf{p_i}$.}} of the 
$T$-equivariant cohomology  
of $\text{Hilb}(\bA_n,d)$
with respect to the classes $\mathsf{p_0}, \ldots, \mathsf{p_n}$,
and 
with matrix coefficients
\begin{equation*}
 \sum_{F\in H_2(\bA_n,\mathbb{Z})} Q^{F}
(-q)^{-d}\,
\bZ_{\mathsf{P}}
\Big( \bA_n\times\Pp/\bD_\infty   ;q\ \Big|\
\prod_{j=0}^n \tau_{\alpha^{(j)}}(\mathsf{p}_j)\, \Big| \mu \Big)^{\T}
_{d\mathsf{L}+\mathsf{F}}\ .
\end{equation*}

\begin{Lemma} For all $d>0$, the matrix $M_{\mathsf{P},d}$ is invertible and
remains invertible after restriction $M_{\mathsf{P},d}|_{s_3=0}$.
\end{Lemma}

\begin{proof}
We may prove invertibility after restriction to
$Q=0$. The issue then separates to the invertibility 
of matrices of caps 
determined at each $p_i$.
The capped geometry $\Pp/\infty$ associated to $p_i$ is the line
$$p_i \times \Pp \subset \bA_n \times \Pp\ .$$
The required invertibility is then obtained from Lemma \ref{7777}.
\end{proof}

By the correspondence of Proposition \ref{hvv2}, the Gromov-Witten
matrix
$M_{\mathsf{GW},d}$ with the same indices and coefficients
$$\sum_{F\in H_2(\bA_n,\mathbb{Z})} Q^{F}
(-iu)^{2d+\ell(\mu)-|\mu|}\,
\bZ_{\mathsf{GW}}
\Big( \bA_n\times\Pp/\bD_\infty   ;u\ \Big|\
\prod_{j=0}^n \widehat{\tau}_{\alpha^{(j)}}(\mathsf{p}_j)\, \Big| 
\mu \Big)^{\T}_{d\mathsf{L}+\mathsf{F}}
$$
is also invertible.

To prove Theorem \ref{yaya38}, we restrict to the fiberwise $T$-action.
Let $p^0_j,\ p^1_j,\ p^\infty_j$ be the 
$T$-fixed points of $\bA_n$ corresponding to $p_j$
in the fibers over $0,1,\infty\in \Pp$ 
respectively.
We consider the stable pairs series for $\bA_n\times \Pp$
\begin{equation*}
 \sum_{F\in H_2(\bA_n,\mathbb{Z})} Q^{F}
(-q)^{-d}\,
\bZ_{\mathsf{P}}
\Big(
\prod_{j=0}^n \tau_{\alpha^{(j)}}(\mathsf{p}^0_j)
\prod_{j=0}^n \tau_{\beta^{(j)}}(\mathsf{p}^1_j)
\prod_{j=0}^n \tau_{\gamma^{(j)}}(\mathsf{p}^\infty_j)
\, \Big)^{T}_{d\mathsf{L}+\mathsf{F}}\ , 
\end{equation*}
where we have
$$\sum_{j=0}^n |\alpha^{(j)}| = \sum_{j=0}^n |\beta^{(j)}|=
\sum_{j=0}^n |\gamma^{(j)}| =d\ .$$
By the correspondence of Theorem \ref{aaa}, the above series
equals
\begin{equation*}
 \sum_{F\in H_2(\bA_n,\mathbb{Z})} Q^{F}
(-iu)^{2d}\,
\bZ_{\mathsf{GW}}
\Big(
\prod_{j=0}^n \widehat{\tau}_{\alpha^{(j)}}(\mathsf{p}^0_j)
\prod_{j=0}^n \widehat{\tau}_{\beta^{(j)}}(\mathsf{p}^1_j)
\prod_{j=0}^n \widehat{\tau}_{\gamma^{(j)}}(\mathsf{p}^\infty_j)
\, \Big)^{T}_{d\mathsf{L}+\mathsf{F}}\ , 
\end{equation*}
We degenerate the descendents over $0,1,\infty$ to
three $\bA_n$-caps. Using the compatibility of the
above correspondences with the degeneration formula and the 
invertibility of $M_{\mathsf{P},d}|_{s_3=0}$, we obtain
the correspondence of Theorem \ref{yaya38}. \qed

\subsection{Descendents on $S\times \Pp/S_{\infty}$} \label{descsp}
Let $S$ be a nonsingular quasi-projective toric surface.
Let $S_\infty \subset S \times \Pp$
be the fiber over $\infty \in \Pp$.
The basic relative geometry
$S\times \Pp/ S_\infty$ is equivariant with respect
to the full 3-dimensional torus
$$\T = T \times \com^*$$
where the 2-dimensional torus $T$ acts on $S$ and
$\com^*$ acts on $\Pp$ with fixed points $0,\infty \in \Pp$.
The weights associated to $T$ are $s_1,s_2$, and the
weight associated to $\com^*$ is $s_3$.

Let $p_1\ldots,p_m$ be the $\T$-fixed points of $S\times \Pp$
lying over $0\in \Pp$. As before,
let $\mathsf{L}\in H_2(S\times \Pp,\mathbb{Z})$ be the curve
class of the factor $\Pp$. For the class $d\mathsf{L}$, we have
$$d_{d\mathsf{L}} = \int_{d\mathsf{L}} c_1(S \times \Pp) = 2d\ . $$

\begin{Proposition} \label{mmpp66}
For the $\T$-equivariant relative geometry $S\times\Pp/S_\infty$, 
after the variable change $-q=e^{iu}$, 
we have
\begin{multline*}
(-q)^{-d}\,
\bZ_{\mathsf{P}}
\Big( S\times\Pp/S_\infty   ;q\ \Big|\
\prod_{j=1}^m \tau_{\alpha^{(j)}}(\mathsf{p}_j)
\, \Big| \mu \Big)^{\T}_{d\mathsf{L}}
=\\
(-iu)^{d+\ell(\mu)}\,  \bZ'_{\mathsf{GW}}
\Big( S\times\Pp/S_\infty;u\ \Big|\
\prod_{j=1}^m \widehat{\tau}_{\alpha^{(j)}}(\mathsf{p}_j)
\, \Big| \mu  \Big )^\T_{d\mathsf{L}}
\,,
\end{multline*}
where $\alpha^{(0)}, \ldots, \alpha^{(n)}$
are  partitions, $\mu$ is a $\T$-equivariant
relative condition along $S_\infty$, and $d>0$.
\end{Proposition}

\begin{proof}
The result is immediate from the 1-leg correspondence of Theorem \ref{ccc}.
Each fixed point $p_i$ of $S$ is contained in a torus
invariant open 
$U \cong \com^2$.
The open set $$U \times \Pp/U_\infty \subset S  \times\Pp/S_\infty$$
is simply the cap.   
Since the curve class is $d\mathsf{L}$, we can directly reduce Proposition 
to the correspondence for 1-leg capped descendent vertices.
\end{proof}

We will apply 
Proposition \ref{mmpp66} in case
$S=\Pp \times \Pp$ 
to
study the non-equivariant limit of the
descendent correspondence in Section \ref{neql}.
We will require there also the following
technical divisibility result valid for all
projective $S$.

Let $S$ be a nonsingular projective toric surface, and let
$p\in S\times \Pp$ be a toric fixed point lying over $0\in \Pp$.
Let 
$$F(\tau(\mathsf{p})) = \sum_{\alpha} \mathsf{C}_\alpha \tau_\alpha(\mathsf{p})$$
be a finite sum over $\alpha$ of positive size and
$\mathsf{C}_\alpha \in \mathbb{Q}[i,s_1,s_2,s_3]((u))$.

\begin{Proposition}
\label{peew45}
If the divisibility 
$$ s_1^k \ | \ \bZ'_{\mathsf{GW}}
\Big( S\times\Pp/S_\infty;u\ \Big|\
F(\tau(\mathsf{p}))
\, \Big| \mu  \Big )^\T_{d\mathsf{L}} $$
holds for all $d>0$ and all relative conditions $\mu$ with
cohomology weights in $H^*_\T(S,\mathbb{Q})$, 
then $s_1^k$ divides $F$.
\end{Proposition}

Since $S$ is projective, the series $\bZ'_{\mathsf{GW}}
\Big( S\times\Pp/S_\infty;u\ \Big|\
F(\tau(\mathsf{p}))
\, \Big| \mu  \Big )^\T_{d\mathsf{L}}$ has no poles in the  $s_i$,
so the divisibility hypothesis in Proposition 
\ref{peew45} 
is sensible. 

\begin{proof}
We can write $F$ in the following form{\footnote{Note the
minimum value of $|\alpha|+\ell(\alpha)$ for $\alpha$
of positive size is 2.}}:
$$F= \sum_{m>0} F_m \ , \ \ \ F_m = \sum_{|\alpha|+\ell(\alpha)=m} \mathsf{C}_\alpha \tau_\alpha(\mathsf{p})\ .$$
We argue by contradiction.
Let $F_M$ be the largest $M$ for which $F_M$ is not divisible by $s^k_1$.
Since the higher $F_{m>M}$ are divisible by $s_1^k$, the hypothesis implies
$$ s_1^k \ | \ \bZ'_{\mathsf{GW}}
\Big( S\times\Pp/S_\infty;u\ \Big|\
\sum_{m=2}^M F_m
\, \Big| \mu  \Big )^\T_{d\mathsf{L}} \ $$
for all $d>0$ and all relative conditions $\mu$.

The next step is a simple dimension analysis.
The codimension of the class $\tau_\alpha(\mathsf{p})$
is $|\alpha|+\ell(\alpha)$.
To each partition $\gamma$, we associate the
relative condition $\gamma$ along $S_\infty$
with all cohomology weights equal to
$1\in H^*_\T(S,\mathbb{Q})$.
The virtual dimension of the relative moduli space of maps
to  $S\times\Pp/S_\infty$ of class $|\gamma|\mathsf{L}$ satisfying
the relative
condition $\gamma$ is $|\gamma|+\ell(\gamma)$.
By compactness and dimension constraints,
we obtain the vanishing
$$
\bZ'_{\mathsf{GW}}
\Big( S\times\Pp/S_\infty;u\ \Big|\
\tau_\alpha(\mathsf{p})
\, \Big| \gamma  \Big )^\T_{|\gamma|\mathsf{L}}=0 \ $$
when $|\alpha|+\ell(\alpha) < |\gamma|+\ell(\gamma)$.
Hence, we see
\begin{equation}\label{yjjk4}
s_1^k \ | \ \bZ'_{\mathsf{GW}}
\Big( S\times\Pp/S_\infty;u\ \Big|\
F_M
\, \Big| \gamma  \Big )^\T_{|\gamma|\mathsf{L}} \ 
\end{equation}
for all partitions $\gamma$ satisfying 
$|\gamma|+\ell(\gamma)=M$.

Next, we consider the matrix indexed
by partitions $\alpha$ and $\gamma$ satisfying
\begin{equation}\label{gtt99}
|\alpha|+\ell(\alpha)= M, \ \ \ |\gamma|+\ell(\gamma)= M \
\end{equation}
with coefficients
\begin{equation}\label{tegg6}
\bZ'_{\mathsf{GW}}
\Big( S\times\Pp/S_\infty;u\ \Big|\
\tau_\alpha(\mathsf{p})
\, \Big| \gamma  \Big )^\T_{|\gamma|\mathsf{L}}\ . 
\end{equation}
By the dimension constraints,
the coefficient \eqref{tegg6} is independent of the
equivariant parameters --- we can treat the
coefficient as a non-equivariant integral. Hence, we can 
calculate \eqref{tegg6}
by separating the points in the descendent insertion. We
replace $\tau_{\alpha}(\mathsf{p})$ by
\begin{equation} \label{66yy}
\tau_{a_1-1}(\mathsf{p}'_1) \ldots \tau_{a_{\ell(\alpha)}-1}(\mathsf{p}'_{\ell(\alpha)})
\end{equation}
for distinct points $p_1',\ldots,p_{\ell({\alpha})}'\in S$.
If $\ell(\alpha)>\ell(\gamma)$, then the corresponding
coefficient
\eqref{tegg6} certainly vanishes as the relative condition
does not have enough parts{\footnote{Remember
the class of the curve is degree 0 when projected to $S$.}} to satisfy the incidences
with $p_1', \ldots, p_{\ell(\alpha)}'$.

The matrix \eqref{tegg6} is block triangular with blocks given by the
equal length condition
$$\ell(\alpha) = \ell(\gamma)\ .$$
The equal length condition implies $|\alpha|=|\gamma|$ by \eqref{gtt99}.
Using the separated insertion \eqref{66yy} and 
further dimension counting, we conclude the coefficient \eqref{tegg6}
vanishes in the block unless $\alpha=\gamma$.

We have proven the matrix \eqref{tegg6} is triangular.
The diagonal elements $\bigtriangleup_{\alpha}$ are determined by \eqref{ffggg} and
are non-zero
(with no $s_1$ dependence). The divisibility of the coefficients
of $F_M$ by $s^k_1$ then immediately follows from \eqref{yjjk4}.
\end{proof}

The proof of Proposition \ref{peew45} provides the
first step of the proof of Proposition \ref{t789789}.
Recall the partial ordering $\uncrazeright$ on partitions:
$$\alpha \uncrazeright \widetilde{\alpha} \ \ \ \ 
\Longleftrightarrow \ \ \ \ |\alpha|+\ell(\alpha) \geq |\widetilde{\alpha}|
+\ell(\widetilde{\alpha}) \ \ .$$

\vspace{10pt} 
\noindent{\bf{Proposition 12.}}
{\em We have
$$\widehat{\tau}_{\alpha}(\mathsf{p}) = (iu)^{\ell(\alpha)-|\alpha|} 
\tau_{\alpha}(\mathsf{p})
+ \ldots$$
where the dots stand for terms $\tau_{\widehat{\alpha}}$
with
$\alpha \crazeright \widehat{\alpha}$.}
\vspace{10pt}

\begin{proof} Let $S$ be any nonsingular projective toric
surface, and let $p\in S$ be a toric fixed point.
By Proposition \ref{mmpp66},
\begin{multline*}
(-q)^{-d}\,
\bZ_{\mathsf{P}}
\Big( S\times\Pp/S_\infty   ;q\ \Big|\
\tau_{\alpha}(\mathsf{p})
\, \Big| \mu \Big)^{\T}_{d\mathsf{L}}
=\\
(-iu)^{d+\ell(\mu)}\,  \bZ'_{\mathsf{GW}}
\Big( S\times\Pp/S_\infty;u\ \Big|\
\sum_{\widehat{\alpha}} \mathsf{K}_{\alpha,\widehat{\alpha}}
\tau_{\widehat{\alpha}}(\mathsf{p})
\, \Big| \mu  \Big )^\T_{d\mathsf{L}}
\,
\end{multline*}
for all $d$.

Consider the set of $\mathcal{P}$ of partitions  
$\widehat{\alpha}$ with $|\alpha|\geq |\widehat{\alpha}|$
which maximize $|\widehat{\alpha}|+\ell(\widehat{\alpha})$
subject to the condition $\mathsf{K}_{\alpha,\widehat{\alpha}}\neq 0$.
Let $\gamma \in \mathcal{P}$ minimize $\ell(\gamma)$.
We view $\gamma$ as a relative condition along $S_\infty$
with all cohomology weights equal to $1\in H^*_\T(S,\mathbb{Q})$.
If $\gamma \crazeright \alpha$, then
$$(-q)^{-d}\,
\bZ_{\mathsf{P}}
\Big( S\times\Pp/S_\infty   ;q\ \Big|\
\tau_{\alpha}(\mathsf{p})
\, \Big| \gamma \Big)^{\T}_{|\gamma|\mathsf{L}} =0, \ \ $$
$$(-iu)^{d+\ell(\mu)}\,  \bZ'_{\mathsf{GW}}
\Big( S\times\Pp/S_\infty;u\ \Big|\
\sum_{\widehat{\alpha}} \mathsf{K}_{\alpha,\widehat{\alpha}}
\tau_{\widehat{\alpha}}(\mathsf{p})
\, \Big| \gamma  \Big )^\T_{|\gamma|\mathsf{L}} = \mathsf{K}_{\alpha,\gamma}\cdot 
\bigtriangleup_\gamma
 $$
where $\bigtriangleup_\gamma$ is non-zero. Here, we have used the
geometric vanishing arguments of the proof of Proposition \ref{peew45} 
for both stable pairs and stable maps. 
Hence, $\mathsf{K}_{\alpha,\gamma}=0$
which is a contradiction.
We conclude
$$\widehat{\alpha} \crazeright \alpha \ \ \ \Longrightarrow
\ \ \ \mathsf{K}_{\alpha,\widehat{\alpha}}=0\ .$$

To prove the Proposition, we need now only consider $\widehat{\alpha}$
for which $$|\alpha|+\ell(\alpha) = |\widehat{\alpha}|+ 
\ell(\widehat{\alpha}).$$
Specifically, we need to exactly match
\begin{equation}\label{mat111}
(-q)^{-|\widehat{\alpha}|}\bZ_{\mathsf{P}}\Big(S\times \Pp/S_\infty;q\ \Big|
\ \tau_{\alpha}(\mathsf{p})\ \Big| \ 
\widehat{\alpha} \Big)^\T_{|\widehat{\alpha}|\mathsf{L}}  
\end{equation}
with the series
\begin{multline} \label{mat222}
(-iu)^{|\widehat{\alpha}|+\ell(\widehat{\alpha})} 
\bZ'_{\mathsf{GW}}
\Big(S \times \Pp/S_\infty;u\ \Big|
\ 
(iu)^{\ell(\alpha)-|{\alpha}|}
\tau_{\alpha}(\mathsf{p})\ \Big| \ 
\widehat{\alpha} \Big)^\T_{|\widehat{\alpha}|\mathsf{L}}\\
=
(-1)
^{|{\alpha}|}
u^{2\ell({\alpha})} 
\bZ'_{\mathsf{GW}}
\Big(S \times \Pp/S_\infty;u\ \Big|
\ \tau_{\alpha}(\mathsf{p})\ \Big| \ 
\widehat{\alpha} \Big)^\T_{|\widehat{\alpha}|\mathsf{L}}
.  
\end{multline}
The exact matching is proven in Section \ref{gg99gg}.
\end{proof}

\subsection{Matching}\label{gg99gg}
Let $Y$ be a nonsingular surface, and let $E\subset Y$ be
a nonsingular curve.
Let $\mathsf{L}\in H_2(Y\times \Pp,\mathbb{Z})$ be the curve class 
of the factor $\Pp$.
The divisor $$E\times \Pp \subset Y \times \Pp$$
intersects the divisor $Y_\infty \subset Y \times \Pp$ 
lying over $\infty\in\Pp$. 
We will consider here
the relative geometry
\begin{equation}\label{pss39}
Y\times \Pp \  / \ E\times \Pp \ \cup\  Y_\infty\ .
\end{equation}
for the curve classes $d\mathsf{L}$ in both
stable pairs and Gromov-Witten theory.

For normal crossings boundary, the most promising approach
to the relative theories is via log geometry \cite{grosss}.
However, our case is very simple since  we are only
considering the curve classes $d\mathsf{L}$.
Since $$\mathsf{L}\cdot [E\times \Pp]=0,$$
 our curves never meet
$E\times \Pp$ and the delicate choices required
for curves
passing through the singularities of 
$E\times \Pp \ \cup\  Y_\infty$ can be completely avoided.
The moduli spaces
$$P_{n}(Y\times \Pp \  / \ E\times \Pp \ \cup\  Y_\infty\ ,\  
{d\mathsf{L}})_\mu\ ,$$
$$\overline{M}'_{g,r}(Y\times \Pp \  / \ E\times \Pp \ \cup\  Y_\infty\ ,\  
{d\mathsf{L}})_\mu$$
are easily defined. In both cases, the projections of
the curves to $Y$ are never allowed to meet $E$ --- bubbling
occurs along $E$ to keep the projections away.
The relative boundary conditions over $\infty\in \Pp$ are then 
imposed as usual. The points of the resulting
moduli spaces corresponds
to stable pairs or stable maps which meet (the degeneration of)
$Y_\infty$ away from the intersection with $E\times\Pp$. Hence,
the deformation theories, virtual classes, and degeneration
formulas are all standard.

Let $\alpha$ and $\widehat{\alpha}$ be two partitions of positive
size satisfying
$$|\alpha|\geq |\widehat{\alpha}|, \ \ \ \ \ \
|\alpha|+\ell(\alpha) = |\widehat{\alpha}|+ 
\ell(\widehat{\alpha})$$
as in the proof of Proposition \ref{t789789}.
The required matching of \eqref{mat111} and \eqref{mat222}
concerns the relative geometry $S\times \Pp/S_\infty$.
By a dimension analysis, the series \eqref{mat111} and \eqref{mat222}
have no dependence on the equivariant parameters $s_i$.
Therefore, we can replace $\tau_\alpha(\mathsf{p})$
by 
\begin{equation*}
\tau_{a_1-1}(\mathsf{p}'_1) \ldots 
\tau_{a_{\ell(\alpha)}-1}(\mathsf{p}'_{\ell(\alpha)})\ 
\end{equation*}
where the points $p_i'\in S$ are
distinct.
Furthermore, we can degenerate to the normal cone of $p_i'\subset S$
for each $p_i'$. The limit of $p_i'$ then lies on 
surface $\mathbf{P}^2/E$ where $E\subset \mathbf{P}^2$
is a line.
We immediately conclude
\begin{multline*}
\bZ_{\mathsf{P}}\Big(S\times \Pp/S_\infty;q\ \Big|
\ \tau_{\alpha}(\mathsf{p})\ \Big| \ 
\widehat{\alpha} \Big)^\T_{|\widehat{\alpha}|\mathsf{L}}=\\
\sum_{\widehat{\alpha} = \cup_i \gamma^{(i)}}
\prod_{i=1}^{\ell(\alpha)} 
\bZ_{\mathsf{P}}\Big(\mathbf{P}^2\times \Pp/\ 
E\times \Pp \cup
\mathbf{P}^2_\infty ;q\ \Big|
\ \tau_{\alpha_i-1}(\mathsf{p}_i')\ \Big| \ 
\gamma^{(i)} \Big)^\T_{|\gamma^{(i)}|\mathsf{L}}\ 
\end{multline*}
where the sum on the right is over all ways of 
writing $\widehat{\alpha}$ as a union of $\ell(\alpha)$
disjoint subpartitions $\gamma^{(i)}$ satifying
$$\alpha_i +1 = |\gamma^{(i)}| + \ell(\gamma^{(i)})\ .$$  
Another degeneration argument implies
\begin{multline*}
\bZ_{\mathsf{P}}\Big(\mathbf{P}^2\times \Pp/\ 
E\times \Pp \cup
\mathbf{P}^2_\infty ;q\ \Big|
\ \tau_{\alpha_i-1}(\mathsf{p}_i')\ \Big| \ 
\gamma^{(i)} \Big)^\T_{|\gamma^{(i)}|\mathsf{L}}\\ =
\bZ_{\mathsf{P}}\Big(\mathbf{P}^2\times \Pp/\ 
\mathbf{P}^2_\infty ;q\ \Big|
\ \tau_{\alpha_i-1}(\mathsf{p}_i')\ \Big| \ 
\gamma^{(i)} \Big)^\T_{|\gamma^{(i)}|\mathsf{L}}
\end{multline*}
If $\gamma^{(i)}$ equals $\mu^{(i)}\cup (1^{m_i})$ where $\mu^{(i)}$ has no parts equal to 1, then
localization with respect to the torus action 
on $\mathbf{P}^2$ yields
\begin{multline*}
\bZ_{\mathsf{P}}\Big(\mathbf{P}^2\times \Pp/\ 
\mathbf{P}^2_\infty ;q\ \Big|
\ \tau_{a_i-1}(\mathsf{p}_i')\ \Big| \ 
\gamma^{(i)} \Big)^\T_{|\gamma^{(i)}|\mathsf{L}} = \\
\sum_{\substack{e_0+e_1+\cdots+e_j = m_i \\ e_0\ge 0, e_k > 0\text{ for }k > 0}}
\Big((-1)^j\bZ_{\mathsf{P}}\Big(\mathsf{Cap};q\ \Big|
\ \tau_{\alpha_i-1}(\mathsf{p})\ \Big| \ 
\mu^{(i)}\cup (1^{e_0}) \Big)^\T_{(|\mu^{(i)}| + e_0)\mathsf{L}}\\
\cdot\prod_{k = 1}^j\bZ_{\mathsf{P}}\Big(\mathsf{Cap};q\ \Big|\ 1\ \Big|\ (1^{e_k}) \Big)^\T_{e_k\mathsf{L}}\Big).
\end{multline*}
The parallel formulas hold in Gromov-Witten theory.

The matching of \eqref{mat111} and \eqref{mat222}
is then a consequence of the following calculation.

\begin{Lemma} \label{lee5}
Let $\gamma = \mu\cup (1^m)$ be a partition of positive size where
 $\mu$ has no parts equal to 1. Let
$a+1= |\gamma|+\ell(\gamma)$. Then,
\begin{multline*}
\sum_{\substack{e_0+e_1+\cdots+e_j = m \\ e_0\ge 0, e_k > 0\text{ for }k > 0}}
\Big((-1)^j\bZ_{\mathsf{P}}\Big(\mathsf{Cap};q\ \Big|
\ \tau_{a-1}(\mathsf{p})\ \Big| \ 
\mu\cup (1^{e_0}) \Big)^\T_{(|\mu| + e_0)\mathsf{L}}\\
\cdot\prod_{k = 1}^j\bZ_{\mathsf{P}}\Big(\mathsf{Cap};q\ \Big|\ 1\ \Big|\ (1^{e_k}) \Big)^\T_{e_k\mathsf{L}}\Big)\quad
=\quad q^{|\gamma|}  \ \frac{(-1)^{\ell(\gamma)-1}}{|\gamma|! \ 
|\text{\em Aut}(\gamma)|} \ 
,
\end{multline*}
\begin{multline*}
\sum_{\substack{e_0+e_1+\cdots+e_j = m \\ e_0\ge 0, e_k > 0\text{ for }k > 0}}
\Big((-1)^j\bZ'_{\mathsf{GW}}\Big(\mathsf{Cap};q\ \Big|
\ \tau_{a-1}(\mathsf{p})\ \Big| \ 
\mu\cup (1^{e_0}) \Big)^\T_{(|\mu| + e_0)\mathsf{L}}\\
\cdot\prod_{k = 1}^j\bZ'_{\mathsf{GW}}\Big(\mathsf{Cap};q\ \Big|\ 1\ \Big|\ (1^{e_k}) \Big)^\T_{e_k\mathsf{L}}\Big)
\quad=\quad u^{-2} \ \frac{1}{|\gamma|! \ 
|\text{\em Aut}(\gamma)|}\ .
\end{multline*}
\end{Lemma}
\begin{proof}
The Gromov-Witten calculation is well-known. In fact,
the result is just a genus 0 connected Gromov-Witten invariant
determined as a special case of Theorem 2 of \cite{opp1}.

We require the stable pairs evaluation. 
We know that the expression has no dependence on $s_1$ and $s_2$, 
so we can work mod $s_1+s_2$ as in \cite{PPstat} and
Lemma~\ref{match1}. 
Localization yields the formula (true mod $s_1+s_2$)
\begin{multline*}
\bZ_{\mathsf{P}}\Big(\mathsf{Cap};q\ \Big|
\ \tau_{a-1}(\mathsf{p})\ \Big| \ 
\mu\cup (1^{e_0}) \Big)^\T_{(|\mu| + e_0)\mathsf{L}} \equiv \\ \\
\frac{(-1)^{\ell(\gamma)-1}(s_1s_2)^{e_1+\cdots+e_j}q^{|\mu|+e_0}}{(a+1)!(|\mu|+e_0)!e_0!\mathfrak{z}(\mu)}\ \cdot\\ \\
\sum_{\sigma \vdash |\mu| + e_o} \chi_\sigma((1)^{|\mu|+e_0})\chi_\sigma(\mu\cup (1^{e_0}))\sum_{\substack{\square\in\sigma \\ c = c(\square)}}((c-1)^{a+1}-2c^{a+1}+(c+1)^{a+1}),
\end{multline*}
where $\chi_\sigma$ is the character of the irreducible representation of the symmetric group corresponding to $\sigma$ and $c(\square)$ is the content of a square in a partition.

The crucial step in the evaluation is the application of the
following  character sum identity:
\begin{multline}\label{charsum}
\sum_{\sigma \vdash |\mu| + e} \chi_\sigma((1)^{|\mu|+e})\chi_\sigma(\mu\cup (1^{e}))\sum_{\substack{\square\in\sigma \\ c = c(\square)}}X^c = \\
\left(\sum_{i = 0}^e i!\binom{e}{i}\binom{|\mu|+e}{i}(X^{\frac{1}{2}}-X^{-\frac{1}{2}})^{|\mu|+2e-2i-2}\right)\prod_{k=1}^{\ell(\mu)}(X^{\frac{\mu_k}{2}}-X^{-\frac{\mu_k}{2}}),
\end{multline}
valid for $\mu$ with no part of size $1$. Before proving this identity, we will complete the proof of Lemma \ref{lee5}. Using 
\eqref{charsum} along with the basic evaluation
\begin{equation*}
\bZ_{\mathsf{P}}\Big(\mathsf{Cap};q\ \Big|\ 1\ \Big|\ (1^{e}) \Big)^\T_{e\mathsf{L}} = \frac{1}{e!(s_1s_2)^e},
\end{equation*}
we see the stable pair expression we need to compute is given by replacing $X^c$ with $(c-1)^{a+1}-2c^{a+1}+(c+1)^{a+1}$ in the Laurent polynomial
\begin{multline*}
(-1)^{\ell(\gamma)-1}q^{|\gamma|}\sum_{\substack{e_0+e_1+\cdots+e_j = m \\ e_0\ge 0, e_k > 0\text{ for }k > 0}}
\frac{(-1)^j}{\prod_{k=1}^j e_k!}\frac{1}{(a+1)!(|\mu|+e_0)!e_0!\mathfrak{z}(\mu)}\cdot \\
\left(\sum_{i = 0}^{e_0} i!\binom{e_0}{i}\binom{|\mu|+e_0}{i}(X^{\frac{1}{2}}-X^{-\frac{1}{2}})^{|\mu|+2e_0-2i-2}\right)\prod_{k=1}^{\ell(\mu)}(X^{\frac{\mu_k}{2}}-X^{-\frac{\mu_k}{2}})
\end{multline*}
Fortunately, most of the terms in this double sum cancel. If $0 \le r \le m$, then the coefficient of 
$$q^{|\gamma|}(X^{\frac{1}{2}}-X^{-\frac{1}{2}})^{|\mu|+2m-2r-2}\prod_{k=1}^{\ell(\mu)}(X^{\frac{\mu_k}{2}}-X^{-\frac{\mu_k}{2}})$$ appearing above is
\begin{equation*}
\frac{(-1)^{\ell(\gamma)-1}}{(a+1)!(m-r)!(|\mu|+m-r)!\mathfrak{z}(\mu)}\sum_{\substack{e_0+e_1+\cdots+e_j = r \\ e_0\ge 0, e_k > 0\text{ for }k > 0}}\frac{(-1)^{j}}{\prod_{k=0}^j e_k!}.
\end{equation*}
The latter is equal to $\frac{(-1)^{\ell(\gamma)-1}}{(a+1)!m!(|\mu|+m)!\mathfrak{z}(\mu)}$ if $r = 0$ and otherwise vanishes because terms in the sum with $e_0 = 0$ can be paired off with those with $e_0 > 0$.

Thus, we just need to compute the result of replacing $X^c$ with 
$$(c-1)^{a+1}-2c^{a+1}+(c+1)^{a+1}$$
 in the Laurent polynomial
\begin{equation*}
\frac{(-1)^{\ell(\gamma)-1}}{(a+1)!m!(|\mu|+m)!\mathfrak{z}(\mu)}q^{|\gamma|}(X^{\frac{1}{2}}-X^{-\frac{1}{2}})^{|\mu|+2m-2}\prod_{k=1}^{\ell(\mu)}(X^{\frac{\mu_k}{2}}-X^{-\frac{\mu_k}{2}}).
\end{equation*}
Since the above polynomial is divisible by $(X-1)^{a-1}$, 
this is simply a matter of differentiating $a-1$ times with respect to $X$, setting $X = 1$, and then multiplying by $(a+1)!$ . We find
\begin{equation*}
\frac{(-1)^{\ell(\gamma)-1}}{m!(|\mu|+m)!\mathfrak{z}(\mu)}q^{|\gamma|}\prod_{k=1}^{\ell(\mu)}\mu_k = q^{|\gamma|}  \ \frac{(-1)^{\ell(\gamma)-1}}{|\gamma|! \ 
|\text{\em Aut}(\gamma)|},
\end{equation*}
as desired.

We now return to \eqref{charsum}. The proof of this identity requires 
only a slight modification of the arguments used in the proof of Theorem A.1 in \cite{FKMO}, so we will only give a sketch here.

Let $n = |\mu|+e$ and define the Jucys-Murphy elements $L_i$ $(1\le i \le n)$ as sums of transpositions
\begin{equation*}
L_i = (1,i)+(2,i)+\cdots+(i-1,i)
\end{equation*}
in the group algebra $\mathbb{C}[S_n]$ of the symmetric group $S_n$.
The elements $L_i$ have the following property: 
for any polynomial $f$, the element $f(L_1)+\cdots+f(L_n)$ acts as the scalar $\sum_{\square\in\lambda}f(c(\square))$ on the irreducible representation $V^\lambda$ of $S_n$ corresponding to a partition $\lambda$ of $n$.

If we let $\tau \in S_n$ be any permutation with cycle type $\mu\cup (1^{e})$, then the trace of the element $\tau^{-1}(f(L_1)+\cdots+f(L_n))$ acting on $\mathbb{C}[S_n]$ will be equal to
\begin{equation*}
\sum_{\sigma \vdash |\mu| + e} \chi_\sigma((1)^{|\mu|+e})\chi_\sigma(\mu\cup (1^{e}))\sum_{\substack{\square\in\sigma \\ c = c(\square)}}f(c).
\end{equation*}
But the trace is also equal to $n!$ times the coefficient of $\tau$ in 
the element $f(L_1)+\cdots+f(L_n)$, which can be computed easily using the Lascoux-Thibon formula to expand the power sum $L_1^r+\cdots+L_n^r$; see Theorem A.4 in \cite{FKMO}. The substitution $q=e^t$ appearing there corresponds precisely to replacing the $X^c$ in \eqref{charsum} with $c^r$.
\end{proof}

\subsection{Proof of Theorem \ref{mmpp22}}\label{mmppr}

Let $X$ be a nonsingular projective toric 3-fold with
$\T$-fixed points 
$$p_1, \ldots, p_m\in X\ .$$
Let $\gamma_1,\ldots, \gamma_r\in H^*(X,\mathbb{Q})$ 
be classes of positive degree.
Since $H^*(X,\mathbb{Q})$ is generated by divisors, we may regard
$\gamma_i$ as a polynomial of positive degree
in the divisor classes. 

\begin{Lemma}
Every  divisor class in $H^2(X,\mathbb{Q})$ can be lifted to 
$H^2_\T(X,\mathbb{Q})$ with trivial restriction to $p_m$.
\end{Lemma}

\begin{proof}
Since the divisor classes of $X$ are
spanned by toric divisors, all divisor classes can be lifted
to $\T$-equivariant cohomology. After further tensoring with
a 1-dimension representation of $\T$, the restriction to
$H^2_\T(p_m,\mathbb{Q})$ can be set to 0.
\end{proof}

After lifting each $\gamma_i$ to a polynomial $\widetilde{\gamma}_i$ in 
$\T$-equivariant divisor classes vanishing at $p_m$,
we can lift the $0$-descendent insertions as a finite sum
$$\prod_{i=1}^r \tau_0(\gamma_i) =
\sum_{k=1}^{m-1} \sum_{l\geq 0} f_{k,l}(s_1,s_2,s_3)\  \tau_0(\mathsf{p}_k)^l\ $$
where $f_{k,l}\in \mathbb{Q}(s_1,s_2,s_3)$.
Hence,
\begin{multline*}
(-q)^{-d_\beta/2}\
\bZ_{\mathsf{P}}\left(X;q \
\Bigg| \ \prod_{i=1}^r {\tau}_0(\gamma_{i})
 \prod_{j=1}^s {\tau}_{k_j}(\mathsf{p}) \right)_{\beta} =\\
(-q)^{-d_\beta/2}\
\bZ_{\mathsf{P}}\left(X;q \
\Bigg| \
\left(\sum_{k=1}^{m-1} \sum_{l\geq 0} f_{k,l}(s_1,s_2,s_3)\  \tau_0(\mathsf{p}_k)^l
\right)
 \prod_{j=1}^s {\tau}_{k_j}(\mathsf{p}_m) \right)^\T_{\beta}
\end{multline*}
We now apply Theorem \ref{aaa} to the right hand side
using the relation
$$\widehat{\tau}_{1^\ell}(\mathsf{p}_i) = \tau_{1^\ell}(\mathsf{p}_i)$$
proven in Section \ref{bacpro}. 
After the variable change $-q=e^{iu}$, we obtain
\begin{multline*} 
(-iu)^{d_\beta} 
 \bZ'_{\mathsf{GW}}\left(X;u \ \Bigg| \ \left(\sum_{k=1}^{m-1} \sum_{l\geq 0} f_{k,l}(s_1,s_2,s_3)\  \tau_0(\mathsf{p}_k)^l
\right) \widehat{\tau}_{\kappa}(\mathsf{p}_m)\right)^\T_{\beta}  =\\
(-iu)^{d_\beta} 
 \bZ'_{\mathsf{GW}}\left(X;u \ \Bigg| \ \prod_{i=1}^r {\tau}_0
(\widetilde{\gamma}_{i})
 \cdot \widehat{\tau}_{\kappa}(\mathsf{p}_m)\right)^\T_{\beta}
\end{multline*}
where $\kappa=(k_1+1, \ldots k_s+1)$.
By Proposition \ref{t789789},
we have
$$\widehat{\tau}_{\kappa}(\mathsf{p}) = (iu)^{\ell(\kappa)-|\kappa|} 
\tau_{\kappa}(\mathsf{p})
+ \ldots$$
where the dots stand for terms $\tau_{\widehat{\alpha}}$
with
$\alpha \crazeright \widehat{\alpha}$.
Finally, using the dimension constraint for 
non-equivariant integrals, we obtain
\begin{equation*}
(-iu)^{d_\beta}(iu)^{-\sum_j k_j} 
 \bZ'_{\mathsf{GW}}\left(X;u \ \Bigg| \ \prod_{i=1}^r {\tau}_0
({\gamma}_{i})
\prod_{j=1}^s
 {\tau}_{k_j}(\mathsf{p})\right)_{\beta}\ ,
\end{equation*}
which is the claimed correspondence. \qed

\section{Non-equivariant limit}
\label{neql}

\subsection{Overview}
Our goal here is to prove the correspondence
of Theorem \ref{aaa} can be written 
completely in non-equivariant terms. The outcome
is a descendent correspondence which makes sense
for any (not necessarily toric) nonsingular projective
3-fold. In the toric case, the non-equivariant
correspondence is a consequence of Theorem \ref{aaa}.
For general 3-folds, the correspondence is
conjectural.

\subsection{Notation}
 We denote the set of descendent
symbols by
$$\tau= \{ \tau_0, \tau_1, \tau_2, \ldots\ \}\ .$$
Let $\sigma$ be a partition with positive parts
$\sigma_1,\ldots,\sigma_\ell$.
We associate a polynomial $\tau_\sigma$ in the symbols
$\tau$
 to $\sigma$ by
\begin{equation*}
\tau_\sigma = \tau_{\sigma_1-1}\tau_{\sigma_2-1}\cdots\tau_{\sigma_l-1}\ 
\end{equation*}
following the conventions
of Section \ref{corr45}.  Using the
correspondence matrix, we define
$$\widehat{\tau}_\sigma = \sum_{\widehat{\sigma}} \mathsf{K}_{\sigma,\widehat{\sigma}}\ \tau_{\widehat{\sigma}}$$
for non-empty partitions $\sigma$.

For subsets
$S\subset \{1,\ldots,\ell\}$, we let $\sigma_S$ be the subpartition consisting of the parts $\sigma_i$ for $i \in S$. 
The following definition is 
crucial to our study:
\begin{equation} \label{kll4}
\widetilde{\tau}_\sigma = \sum_{P\text{ set partition of }\{1,\ldots,\ell\}}(-1)^{|P|-1}(|P|-1)!\prod_{S\in P}\widehat{\tau}_{\sigma_S}\ .
\end{equation}
Here,
 $\widetilde{\tau}_\sigma$ lies in the polynomial ring in the
symbols $\tau$ with coefficients in 
${Q}[i,s_1,s_2,s_3]((u))$.
The polynomials $\widehat{\tau}_{\sigma_S}$ on the right
carry the 
complexity of the correspondence matrix $\mathsf{K}$.

\subsection{Divisibility}\label{div555}

In order to obtain a non-equivariant formulation of
Theorem \ref{aaa}, our first step is to prove a
divisibility result constraining 
 the coefficients 
$${\mathsf{K}}_{\sigma,\widehat{\sigma}} \in \Q[i,s_1,s_2,s_3]((u))$$ 
of the correspondence matrix.

\begin{Proposition}\label{K divisibility}
Let $\sigma$ be a partition of positive length $\ell$. Then,
\begin{equation}\label{fmmk}
\widetilde{\tau}_{\sigma} \equiv 0 \mod{(s_1s_2s_3)^{\ell-1}},
\end{equation}
as a polynomial in the descendent symbols $\tau$ with 
coefficients in the ring $\Q[i,s_1,s_2,s_3]((u))$.
\end{Proposition}

\begin{proof}
We will use the relative correspondence 
established in Proposition \ref{mmpp66} for the geometry 
 $$S\times \Pp /S_\infty = \Pp\times\Pp\times\Pp \ / \ (\Pp\times \Pp)_\infty
\ .$$
The surface
$S=\Pp\times \Pp$  is viewed as the first two factors of the 3-fold 
$\Pp\times \Pp\times \Pp$.
We will consider $\T$-equivariant stationary
descendents at
$$p_\bullet = (0,0,0), \ \ \ p_\star = (\infty,0,0)\ .$$
 Let the tangent weights at $p_\bullet$ be $s_1,s_2,s_3$ respectively 
along the three $\Pp$ factors. Then,
 $$\p_\bullet-\p_\star = s_1\mathsf{P}_{\bullet\star}\ $$
where $\mathsf{P}_{\bullet\star}$ is the
class of the line $P_{\bullet\star}$ connecting the two points.

We will prove the Proposition together with the following divisibility
claim:
\begin{equation}\label{divisibility}
\bZ'_{\mathsf GW}\left(S\times \Pp/S_\infty; u \;\middle\vert\; 
\widetilde{\tau}_{\sigma}(\p_\bullet) \;\middle\vert\; \mu\right)^\T_{d\mathsf{L}} \equiv 0 \ \mod{s_1^{\ell-1}}
\end{equation}
for every curve class $d\mathsf{L}$ and relative condition $\mu$.
We prove the Proposition and the divisibility
\eqref{divisibility} simultaneously by induction on 
the length $\ell$ of $\sigma$. 
For $l = 1$, both statements are trivial. We assume $l > 1$.

If divisibility
\eqref{divisibility} holds for 
partitions $\sigma$ of length $l$, then $\widetilde{\tau}_\sigma$
must be divisible by $s_1^{\ell-1}$ by Proposition \ref{peew45}.
By the symmetry of the coefficients of $\mathsf{K}$ in the
variables $s_i$, we conclude $\widetilde{\tau}_\sigma$
is divisible by $(s_1s_2s_3)^{\ell-1}$. We have proven
claim \eqref{divisibility} for length $\ell$
implies claim \eqref{fmmk}
for length $\ell$.

Finally, we show if \eqref{fmmk} holds for
partitions of length $1,2,\ldots,l-1$,
 then divisibility \eqref{divisibility} holds for length $l$.
For any set partition 
$$Q = \{Q_1,\ldots,Q_k\}\ \  \text{of} \ \ \{1,\ldots,l\}$$
 with $1\in Q_1$ and $k > 1$, consider the Gromov-Witten
series for $S\times \Pp/ S_\infty$
\begin{equation}\label{divisible}
\bZ'_{\mathsf{GW}}\left( 
\widetilde{\tau}_{\sigma_{Q_1}}(\p_\bullet)\prod_{j=2}^k
\left(\widetilde{\tau}_{\sigma_{Q_j}}(\p_\bullet) 
+ (-1)^{|Q_j|}\widetilde{\tau}_{\sigma_{Q_j}}(\p_\star)\right) \;\middle\vert\; 
\mu\right)^\T_{d\mathsf{L}} 
\end{equation}
lying in $\Q[i,s_1,s_2,s_3]((u))$.
By the inductive hypothesis, we pick up a factor of $s_1^{|Q_j|-1}$ 
for each part in the set partition. 
Also, each part in the set partition 
after the first  contributes an additional factor of $s_1$ because
\begin{enumerate}
\item [(i)] the correspondence
matrices $\mathsf{K}$ at the points $p_\bullet$ and $p_\star$
are equal after changing the sign of $s_1$,
\item[(ii)]
 $s_1$ divides $\p_\bullet - \p_\star$ \ .
\end{enumerate}
Thus, we see \eqref{divisible} is divisible by $s_1^{l-1}$.

Using again the divisibility of $\p_\bullet - \p_\star$
by $s_1$ in $\T$-equivariant cohomology, we see 
$$\bZ_{\mathsf{P}}\left( {\tau_{\sigma_1}(\p_\bullet)\prod_{j=2}^l
\tau_{\sigma_j}(\p_\bullet - \p_\star)} \;\middle\vert\; \mu\right)^\T_{d\mathsf{L}}
\equiv 0 \ \ \mod s_1^{\ell-1}\ .$$
From our the descendent correspondence of Proposition
\ref{mmpp66}, we immediately conclude
\begin{equation}\label{jtt999}
\bZ'_{\mathsf{GW}}\left( \widehat{\tau_{\sigma_1}(\p_\bullet)
\prod_{j=2}^l\tau_{\sigma_j}(\p_\bullet - \p_\star)} \;\middle\vert\; \mu\right)^\T_{d\mathsf{L}} \equiv 0 \mod{s_1^{l-1}},
\end{equation}
where $\widehat{\tau_\delta(\p_\bullet)\tau_\eta(\p_{\star})} =
 \widehat{\tau}_\delta(\p_\bullet)\widehat{\tau}_\eta(\p_{\star})$.

The desired divisibility \eqref{divisibility} now follows from the basic
identity
\begin{multline}\label{basid}
\sum_{\substack{Q \text{ set partition of }\{1,\ldots,\ell\}\\ 1\in Q_1}}
\widetilde{\tau}_{\sigma_{Q_1}}(\p_\bullet)\prod_{j=2}^{|Q|}
\left(\widetilde{\tau}_{\sigma_{Q_j}}(\p_\bullet) 
+ (-1)^{|Q_j|}\widetilde{\tau}_{\sigma_{Q_j}}(\p_\star)\right) \\
 =\widehat{\tau_{\sigma_1}(\p_\bullet)\prod_{j=2}^\ell\tau_{\sigma_j}(\p_\bullet
 - \p_\star)}.
\end{multline}
Since $s_1^{\ell-1}$ divides both \eqref{divisible} and \eqref{jtt999},
we conclude the claim \eqref{divisibility} holds for length $\ell$ 
from the
identity \eqref{basid}. The term $\tau_\sigma(\p_\bullet)$ occurs
on the left side of \eqref{basid} when $Q_1=\{1,\ldots, \ell\}$.

We now check identity \eqref{basid} by computing the coefficient of
\begin{equation}\label{jrpp9}
\prod_{S\in A}\widehat{\tau_{\sigma_S}(\p_\bullet)}\prod_{T\in B}\widehat{\tau_{\sigma_T}(\p_\star)}
\end{equation}
appearing when the $\widetilde{\tau}$ on the left side are expanded in terms of $\widehat{\tau}$. Here,
$A\cup B = C$
 is a partition of a set partition $C$ of $\{1,\ldots,\ell\}$ into two nonempty subsets, with $1$ belonging to one of the parts in $A$. Let 
$$a = |\bigcup_{S\in A}S| \text{\ \ and\ \ } b = |\bigcup_{T\in B}T|,$$
 so the coefficient of this term on the right side 
of \eqref{basid}
is equal to $(-1)^b$ if $|A| = 1$ and $|B|\le 1$, and $0$ otherwise.

The term on the left side of \eqref{basid} given by a set partition $Q$ of $\{1,\ldots,\ell\}$ will contribute to the coefficient 
of \eqref{jrpp9}
if and only if $Q = Q_A\cup Q_B$ with $A$ a refinement of $Q_A$ and $B$ a refinement of $Q_B$. Such $Q$ are parametrized by choosing one set partition of $A$ and one of $B$, so we compute the coefficient to be
\[
\sum_{ \substack{P_A \text{ set partition of } A \\ P_B \text{ set partition of }B}}
\left(\prod_{S\in P_A}(-1)^{|S|-1}(|S|-1)!\right)\cdot(-1)^b\cdot\left(\prod_{T\in P_B}(-1)^{|T|-1}(|T|-1)!\right).
\]

Finally, we need only prove the fundamental identity
\begin{equation*}
\sum_{P \text{ set partition of }\{1,\ldots,k\}}\prod_{S\in P}(-1)^{|S|-1}(|S|-1)! = \begin{cases}
1 &\text{if }k = 0,1 \\
0 &\text{if }k > 1, \end{cases}
\end{equation*}
which follows immediately from the observation that each term is counting the permutations of $\{1,\ldots,k\}$ that yield a given orbit partition $P$, with sign equal to the sign of the permutations of this type.
\end{proof}

We define the correspondence matrix $\widetilde{\mathsf{K}}$ 
which we will use for the non-equivariant limit by
\begin{equation}\label{frg3}
\widetilde{\mathsf{K}}_{\sigma,\widehat{\sigma}}= 
\frac{1}{(s_1s_2s_3)^{\ell(\sigma)-1}}
\text{Coeff}_{\tau_{\widehat{\sigma}}}\Big(
\widetilde{\tau}_\sigma\Big)\ .
\end{equation}
By the vanishing $\mathsf{K}_{\sigma,\widehat{\sigma}}=0$ unless
$|\sigma|\geq |\widehat{\sigma}|$, we deduce the vanishing
 $$\widetilde{\mathsf{K}}_{\sigma,\widehat{\sigma}}=0\ \  \text{unless} \ 
|\sigma|\geq |\widehat{\sigma}| \ .$$
By the divisibility of Proposition \ref{K divisibility},
$$\widetilde{\mathsf{K}}_{\sigma,\widehat{\sigma}}\in\Q[i,s_1,s_2,s_3]((u))\ .$$ 
In fact, since $\mathsf{K}$ is symmetric in the $s_i$, we may
view
$$\widetilde{\mathsf{K}}_{\sigma,\widehat{\sigma}}\in\Q[i,c_1,c_2,c_3]((u))\ .$$
where the $c_i$ are elementary symmetric functions in the $s_i$.

\begin{Proposition}\label{ll33t}
The $u$ coefficients of  $\widetilde{\mathsf{K}}_{\sigma,\widehat{\sigma}}\in
\mathbb{Q}[i,s_1,s_2,s_3]((u))$
are symmetric and homogeneous in the variables $s_i$
of degree $$|\sigma|+\ell(\sigma) - |\widehat{\sigma}| 
- \ell(\widehat{\sigma})-3(\ell(\sigma)-1).$$ 
\end{Proposition}

\begin{proof}
The result follows from Theorem \ref{ll33} and definitions
\eqref{kll4} and \eqref{frg3}.
\end{proof}

\subsection{Proof of Theorem \ref{zzz}}
Let $X$ be a nonsingular quasi-projective toric 3-fold.
Let $\alpha$ be a partition of length $\ell$ and
positive parts.
Let 
$$\gamma_1,\ldots,\gamma_\ell\in H^*_\T(X,\mathbb{Q})$$ be
$\T$-equivariant classes. We can express
\begin{equation}\label{gnnl}
\bZ_{\mathsf{P}}\left(X; q \;\middle\vert\; \tau_{a_1-1}(\gamma_1)
\cdots\tau_{a_\ell-1}(\gamma_\ell)\right)_\beta 
\end{equation}
in terms of Gromov-Witten theory by writing each class $\gamma_l$ 
as a combination of 
the $\T$-fixed points via \eqref{ff34}  and 
then applying the descendent correspondence of Theorem \ref{aaa}.


Let $\mathcal{P}^{\mathrm{set}}_\ell$ be the set of set
partitions of $\{1,\ldots, \ell \}$. For a partition
$P\in \mathcal{P}^{\mathrm{set}}_\ell$, each $S\in P$
is a subset of $\{1,\ldots, \ell\}$.
Let
$$\gamma_S = \prod_{i\in S}\gamma_i \text{\ \ \ and \ \ \ } \tau_{\alpha_S} = 
\prod_{i\in S}\tau_{a_i-1}\ .$$

A first formula for the
Gromov-Witten descendent corresponding to 
the stable pairs integral \eqref{gnnl} is
given by
\begin{equation}\label{kk334}
\sum_{P\in \mathcal{P}^{\mathrm{set}}_\ell}\ \
\sum_{\text{Injective }\phi: P\to\{1,\ldots,m\}}
\ \
\prod_{S\in P}\frac{\iota_{\phi(S)}^*(\gamma_S)}{\iota_{\phi(S)}^*(c_3(X)^{|S|})}\widehat{\tau}_{\alpha_S}(\p_{\phi(S)})\ .
\end{equation}
Here, the $\T$-fixed points of $X$ are $p_1,\ldots, p_m$,
and we follow the notation of the localization identity
\eqref{ff34}.

We may 
extend the second sum in \eqref{kk334} to run
over {all}  functions 
$$\phi:P\to\{1,\ldots,m\}$$
 (rather than just the 
injective ones) by rewriting the formula as
\begin{equation}\label{pqq2}
\sum_{P\in \mathcal{P}^{\mathrm{set}}_\ell}
\ \sum_{\phi: P\to\{1,\ldots,m\}}\ 
\prod_{S\in P}\frac{\iota_{\phi(S)}^*(\gamma_S)}{\iota_{\phi(S)}^*(c_3(X)^{|S|})}\widetilde{\tau}_{\alpha_S}(\p_{\phi(S)})\ 
\end{equation}
using definition \eqref{kll4}.
Finally, 
the cohomological identity
\[
\sum_{j=1}^m\frac{\iota_{j}^*(\gamma)}{\iota_{j}^*(c_3(X))}
\ \p_j
\otimes\cdots\otimes \p_j = \gamma\cdot\Delta
\ \in H^*_\T(X \times \cdots\times X, \mathbb{Q})
\]
allows us to rewrite \eqref{pqq2} efficiently as
\begin{equation} \label{qqq777}
\sum_
{P\in \mathcal{P}^{\mathrm{set}}_\ell}\ 
\prod_{S\in P}\frac{1}{c_3(X)^{|S|-1}}\
\widetilde{\tau}_{\alpha_S}(\gamma_S)\
\end{equation}
following convention \eqref{j77833}.
Theorem \ref{zzz} then follows from Theorem \ref{aaa},
formula \eqref{mqq23}, and
our definition of $\widetilde{\mathsf{K}}$. \qed
\vspace{10pt}

For a nonsingular quasi-projective toric 3-fold $X$
with $\T$-fixed points $p_1,\ldots, p_m$, we
have two descendent correspondences for the 
stable pairs series
\begin{equation}
\label{p3p3}
\ZZ_{\mathsf{P}}\Big(X;q\ \Big|   \prod_{j=1}^m \tau_{\alpha^{(j)}}(\mathsf{p}_j)
\Big)^\T_\beta\ 
\end{equation}
of Section \ref{corr45}. We can apply Theorem \ref{aaa} or
Theorem \ref{zzz}. In fact, the result is the same.

\begin{Lemma} Theorem \ref{zzz} applied to \eqref{p3p3}
specializes exactly to Theorem \ref{aaa}.
\end{Lemma}
\begin{proof}
The claim reduces to the inversion formula
\begin{equation*}
\sum_{Q\text{ set partition of }\{1,\ldots,\ell(\alpha^{(j)})\}}\prod_{S\in Q}\widetilde{\tau}_{\alpha^{(j)}_S}(\mathsf{p}_j) = \widehat{\tau}_{\alpha^{(j)}}(\mathsf{p}_j)
\end{equation*}
obtained from the specialization $\mathsf{p}_\star = 0$ of \eqref{basid}.
\end{proof}

Proposition \ref{ll33t} implies the descendent correspondence
of Theorem \ref{zzz} respects the dimensions of the insertions.

\subsection{Relative descendent correspondence}
Let $X$ be a nonsingular projective 3-fold, and let
$D\subset X$ be a nonsingular divisor.
Let $\Omega_X[D]$ denote the locally free sheaf of 
differentials with logarithmic poles along $D$.
Let $$T_{X}[-D] = \Omega_{X}[D]^{\ \vee}$$
denote the dual sheaf of tangent fields with logarithmic
zeros.

For the relative geometry $X/D$, we let the coefficients of
$\widetilde{\mathsf{K}}$ act on the cohomology of $X$ via the
substitution
$$c_i= c_i(T_{X}[-D])$$
instead of the substitution $c_i=T_X$ used in the absolute case.
Then, we would like to define 
\begin{equation*}
\overline{\tau_{\alpha_1-1}(\gamma_1)\cdots
\tau_{\alpha_{\ell}-1}(\gamma_{\ell})}
=
\sum_{P \text{ set partition of }\{1,\ldots,l\}}\ \prod_{S\in P}\ \sum_{\widehat{\alpha}}\tau_{\widehat{\alpha}}(\widetilde{\mathsf{K}}_{\alpha_S,\widehat{\alpha}}\cdot\gamma_S) \ .
\end{equation*}
as before.
The correct definition is subtle for arbitrary classes $\gamma_i$.
A full discussion of the descendent correspondence for
relative geometries will be given in \cite{PPcy3}.
However, a restricted case in which the above definition is appropriate
will be relevant for Section \ref{8888}.

\begin{Conjecture}
\label{ttt444} 
Let $\gamma_1, \ldots, \gamma_l\in H^*(X,\mathbb{Q})$ be classes
which restrict to 0 on $D$, then 
we have 
\begin{multline*}
(-q)^{-d_\beta/2}\ZZ_{\mathsf{P}}\Big(X/D;q\ \Big|  
{\tau_{\alpha_1-1}(\gamma_1)\cdots
\tau_{\alpha_{\ell}-1}(\gamma_{\ell})} \ \Big| \ \mu
\Big)_\beta \\ =
(-iu)^{d_\beta+\ell(\mu)-|\mu|}\ZZ'_{\mathsf{GW}}\Big(X/D;u\ \Big|   
\ \overline{\tau_{a_1-1}(\gamma_1)\cdots
\tau_{\alpha_{\ell}-1}(\gamma_{\ell})}
\ \Big| \ \mu\Big)_\beta 
\end{multline*}
under the variable change $-q=e^{iu}$.
\end{Conjecture}

In addition, the stable pairs descendent series on the left is
conjectured to be a rational function in $q$, so the change
of variables is well-defined.
Conjecture \ref{ttt444} is open.

\section{Log Calabi Yau 3-folds}
\label{8888}
\subsection{Overview}
Let $X$ be a nonsingular projective toric Fano 3-fold with a
nonsingular irreducible anti-canonical divisor $S$ (necessarily
isomorphic to a $K3$ surface). The relative geometry $X/S$ is
{\em log Calabi-Yau} since the sheaf of differentials of $X$
with logarithmic poles along $S$ has trivial determinant.  
In order to prove the Gromov-Witten/Pair correspondence of Theorem \ref{mmpp44}
for $X/S$, we will require
 results about projective bundles over $S$.

Let $L_0$ and $L_\infty$ be two line bundles on $S$. The projective bundle
$$\mathbf{P}_S= \mathbf{P}(L_0 \oplus L_\infty) \rightarrow S$$
admits sections 
$$ S_i= \mathbf{P}(L_i) \subset \mathbf{P}_S\ .$$
Before proving Theorem \ref{mmpp44}, we will establish 
the 
relative descendent correspondence 
of Conjecture \ref{ttt444}
for $\mathbf{P}_S/S_\infty$
for descendent insertions supported on $S_0$.
While the result goes beyond toric varieties,
the vanishings which hold for $K3$ geometries 
make $\mathbf{P}_S/S_\infty$
accessible. The descendent correspondences
for projective bundles over surfaces will be studied in
more detail in \cite{PPcy3}.

Theorem \ref{mmpp44} and Corollary \ref{yaya34} will follow easily 
from Theorem \ref{zzz}, degeneration, and the descendent
correspondence for $\mathbf{P}_S/S_\infty$.

\subsection{Descendent correspondence for $\mathbf{P}_S/S_\infty$}
Let $\phi_1, \ldots, \phi_\ell$ be cohomology classes on
$\mathbf{P}_S$ supported{\footnote{Each $\phi_i$ is
push-forward of a class on $S$.
Since $K3$ surfaces have only even cohomology, the $\phi_i$
have even degrees.}} on the section $S_0$.
Let $\mu$ be a relative condition along $S_\infty$. 
Our first step is to prove the non-equivariant descendent correspondence of
Conjecture \ref{ttt444} for the classes $\phi_i$.

\begin{Proposition}\label{qq99}
We have
\begin{multline}\label{plgt}
(-q)^{-d_\beta/2}\ZZ_{\mathsf{P}}\Big(\mathbf{P}_S/S_\infty;q\ \Big|  
{\tau_{\alpha_1-1}(\phi_1)\cdots
\tau_{\alpha_{\ell}-1}(\phi_{\ell})} \Big| \ \mu
\Big)_\beta \\ =
(-iu)^{d_\beta+\ell(\mu)-|\mu|}\ZZ'_{\mathsf{GW}}\Big(\mathbf{P}_S/S_\infty;u\ \Big|   
\ \overline{\tau_{\alpha_1-1}(\phi_1)\cdots
\tau_{\alpha_{\ell}-1}(\phi_{\ell})}\ 
\ \Big| \ \mu \Big)_\beta 
\end{multline}
after  the variable change $-q=e^{iu}$.
\end{Proposition}

The log tangent bundle of $\mathbf{P}_S/S_\infty$ restricts
to the standard tangent bundle of $\mathbf{P}_S$
on the section $S_0$.
Since the classes $\phi_i$ are supported on $S_0$,
the descendent correspondence matrix $\widetilde{\mathsf{K}}$
for $\mathbf{P}_S/S_\infty$ is the same as the matrix
 for $\mathbf{P}_S$.

\begin{proof}
By the vanishings in stable pair and Gromov-Witten theory
obtained from the holomorphic symplectic structure of $K3$
surfaces, only invariants of $\mathbf{P}_S$ in multiples
of the 
fiber class $$\mathsf{L} \in H_2(\mathbf{P}_S,\mathbb{Z})$$
contracted over $S$ are non-zero.
Moreover, for the stable pairs series, only the
initial $q$-coefficient is non-zero.

Let $X$ be any nonsingular projective surface equipped with
line bundles $L_0$ and $L_\infty$.
Let $\phi_1,\ldots, \phi_\ell$
be cohomology classes on
$$\mathbf{P}_S= \mathbf{P}(L_0 \oplus L_\infty) \rightarrow S$$
supported on the section $X_0$.
Consider the Gromov-Witten series
\begin{equation} \label{hwhwhw}
(-iu)^{|\mu|+\ell(\mu)}\ZZ'_{\mathsf{GW}}\Big(\mathbf{P}_X/X_\infty;u\ \Big|   
\ \overline{\tau_{\alpha_1-1}(\phi_1)\cdots
\tau_{\alpha_{\ell}-1}(\phi_{\ell})}\ 
\ \Big| \ \mu \Big)_{|\mu|\mathsf{L}}\ .
\end{equation}
Each $u$-coefficient of \eqref{hwhwhw} can be expressed by
an explicit study of the moduli space of stable maps to
the fiber classes of $\mathbf{P}_X \rightarrow X$.
By a standard analysis (see Section 1.2 of \cite{mptop}),
each $u$-coefficient is a universal 
polynomial over $\mathbb{Q}$ in the 
all classical pairings 
\begin{equation}\label{cll333}
\int_X \Theta\Big(c(T_X),c(L_0),
c(L_\infty)\Big) \cup \prod_{i \in I} \phi_i
\end{equation}
where $\Theta$ is a monomial in the Chern classes of the
bundles 
$$T_X,\ L_0,\ L_\infty \rightarrow X\,$$
of bounded degree (determined by
the descendent partition $\alpha=(\alpha_1,\ldots, \alpha_\ell)$ and the degrees of 
$\phi_i$).
In the product on the right side of \eqref{cll333}, 
$$I\subset \{1, \ldots, \ell\}$$ is a subset.
 
Let $X$ be a nonsingular projective toric surface
with toric line bundles $L_0$ and $L_\infty$.
For fixed $\alpha$ and
$\phi_i$, there are only finitely many
classical pairings \eqref{cll333}. Moreover,
as we vary the toric surface $X$ and the toric line
bundle $L_i$, we easily see
a Zariski dense set of possible
classical pairings is achieved.{\footnote{Since
the classes $\phi_i$ we consider are of even degrees,
the degrees can be matched in toric geometry.}}
Hence, the $u$-coefficient polynomials of the Gromov-Witten
series are fully determined by the toric examples.

If $X$ is a nonsingular projective toric surface
with toric line bundles $L_i$, 
then the relation
\begin{multline*}
(-q)^{-|\mu|}\ZZ_{\mathsf{P}}\Big(\mathbf{P}_X/X_\infty;q\ \Big|  
{\tau_{\alpha_1-1}(\phi_1)\cdots
\tau_{\alpha_{\ell}-1}(\phi_{\ell})} \Big| \ \mu
\Big)_{|\mu|\mathsf{L}} \\ =
(-iu)^{|\mu|+\ell(\mu)}\ZZ'_{\mathsf{GW}}\Big(\mathbf{P}_X/X_\infty;u\ \Big|   
\ \overline{\tau_{\alpha_1-1}(\phi_1)\cdots
\tau_{\alpha_{\ell}-1}(\phi_{\ell})}\ 
\ \Big| \ \mu \Big)_{|\mu| \mathsf{L}} 
\end{multline*}
is a direct consequence, by localization, of the descendent 
correspondence for the cap. When we localize with respect
to the 2-dimension torus $T$ acting on $X$ and the $L_i$, the result
is cap for each $T$-fixed point of $X$.
The stable pairs series
\begin{equation}\label{gttl9}
(-q)^{-|\mu|}\ZZ_{\mathsf{P}}\Big(\mathbf{P}_X/X_\infty;q\ \Big|  
{\tau_{\alpha_1-1}(\phi_1)\cdots
\tau_{\alpha_{\ell}-1}(\phi_{\ell})} \Big| \ \mu
\Big)_{|\mu|\mathsf{L}}\ 
\end{equation}
is thus also determined by the
classical pairings \eqref{cll333} in the toric case.
In fact, using the denominator results proven in Theorem 5 of \cite{part1},
the $q$-coefficients of \eqref{gttl9} are
 polynomials in the pairings \eqref{cll333}.

Next, let the nonsingular projective
surface $X$ with line bundles $L_0$ and $L_\infty$
be arbitrary.
The $q^0$-coefficient of the stable pairs series
\eqref{gttl9} is special. The associated moduli
space of stable pairs is simply the Hilbert scheme
of $|\mu|$ point of $X$.
The stable pairs invariant then can be calculated
by Hilbert scheme techniques \cite{lehn}. The result
is also a polynomial in classical pairings \eqref{cll333}.
Hence, we have two polynomials in the classical pairing
\eqref{cll333}:
\begin{enumerate}
\item [(i)] the $q^0$-coefficient of the
Gromov-Witten series \eqref{hwhwhw} for $X$ after the variable
change $-q=e^{iu}$,
\item [(ii)] the polynomial obtained
from the Hilbert scheme of points calculation
of the $q^0$-coefficient of the stable pairs series \eqref{gttl9} for $X$.
\end{enumerate}
The two polynomials are equal when
evaluated in the toric geometry and thus must be
identical (by Zariski denseness).

The polynomials (i) and (ii) are therefore equal for the
$K3$ geometry $\mathbf{P}_S/S_\infty$.
To complete the proof of the correspondence \eqref{plgt},
we must only prove  the higher $q$-coefficients, 
obtained 
 after the variable
change $-q=e^{iu}$
for 
Gromov-Witten series \eqref{hwhwhw} for 
$\mathbf{P}_S/S_\infty$,
all vanish.

Let $X$ be a nonsingular quasi-projective toric surface with
toric line bundles $L_0$ and $L_\infty$.
Consider the $T$-equivariant Gromov-Witten series
\begin{equation} \label{hwhwhwhw}
(-iu)^{|\mu|+\ell(\mu)}\ZZ'_{\mathsf{GW}}\Big(\mathbf{P}_X/X_\infty;u\ \Big|   
\ \overline{\tau_{\alpha_1-1}(\phi_1)\cdots
\tau_{\alpha_{\ell}-1}(\phi_{\ell})}\ 
\ \Big| \ \mu \Big)^T_{|\mu|\mathsf{L}}\ 
\end{equation}
where $T$ is the 2-dimensional torus acting on $X$ and the
$L_i$.
As before, each $u$-coefficient 
is a universal 
polynomial over $\mathbb{Q}$ in the 
all classical $T$-equivariant pairings 
\begin{equation}\label{cll3333}
\int_X \Theta\Big(c(T_X),c(L_0),
c(L_\infty)\Big) \cup \prod_{i \in I} \phi_i
\end{equation}
where $\Theta$ is a monomial in the Chern classes of the
bundles 
$$T_X,\ L_0,\ L_\infty \rightarrow X,$$
of bounded degree (determined by
the descendent partition $\alpha$ and the degrees of $\phi_i$).
In $T$-equivariant geometry,  more
pairings may be non-zero.
Otherwise, the situation is exactly the same as
in  
the non-equivariant case. 
The universal polynomials in the $T$-equivariant geometry restrict to the
universal polynomials in the non-equivariant geomery.

We finally specialize $X$ to the quasi-projective surfaces $\bA_n$.
If we  restrict to the sub-torus $\com^* \subset T$
which preserves the holomorphic form, then 
$$c_1(T_X)=0\ .$$
The $\bA_n$ geometries, as $n$ varies, provide 
a rich supply of $\com^*$-equivariant
pairings \eqref{cll3333} 
subject to the vanishing of $c_1(T_X)$, 
The $T$-equivariant correspondence 
\begin{multline*}
(-q)^{-|\mu|}\ZZ_{\mathsf{P}}\Big(\mathbf{P}_X/X_\infty;q\ \Big|  
{\tau_{a_1-1}(\phi_1)\cdots
\tau_{a_{\ell}-1}(\phi_{\ell})} \Big| \ \mu
\Big)^T_{|\mu|\mathsf{L}} \\ =
(-iu)^{|\mu|+\ell(\mu)}\ZZ'_{\mathsf{GW}}\Big(\mathbf{P}_X/X_\infty;u\ \Big|   
\ \overline{\tau_{a_1-1}(\phi_1)\cdots
\tau_{a_{\ell}-1}(\phi_{\ell})}\ 
\ \Big| \ \mu \Big)^T_{|\mu| \mathsf{L}} 
\end{multline*}
has already been proven for $X=\bA_n$. 
The higher $q$-coefficients of the stable pairs side above
vanish (since $\bA_n$ has a holomorphic symplectic form
invariant under $\com^*$).
Hence,
the higher $q$-coefficients 
obtained 
 after the change of variables
 $-q=e^{iu}$
for 
the Gromov-Witten series \eqref{hwhwhwhw} for 
$\bA_n$
all vanish.
By the universality of the polynomials and the sufficient
Zariski density of the $\bA_n$ geometries (subject
to the vanishing of the first Chern class of the
tangent bundle), we conclude the necessary 
vanishing of 
the higher $q$-coefficients 
for 
Gromov-Witten series \eqref{hwhwhw} for 
$\mathbf{P}_S/S_\infty$.
\end{proof}

\subsection{Proof of Theorem \ref{mmpp44}} \label{mmpp44pr}
Let $X$ be a nonsingular projective Fano toric 3-fold, and let
$S\subset X$ be a nonsingular anti-canonical $K3$ surface.
Let 
$N$
be the normal bundle of $S$ in $X$.
Let
$$S_0,S_\infty \subset \mathbf{P}(\cO_S \oplus N)$$
be the sections determined by the summand $\cO_S$ and
$N$ 
respectively. Let
$$\iota_0: S \hookrightarrow \mathbf{P}(\cO_S \oplus N)$$
be the inclusion of $S_0$.

Let $\mathcal{B}$ be a fixed self-dual basis of the cohomology of $S$.
Recall a Nakajima basis element in the Hilbert scheme 
$\text{Hilb}(S,n)$ is a cohomology weighted partition $\mu$ of $n$,
$$( (\mu_1,\phi_1), \ldots, (\mu_\ell,\phi_\ell))\ ,  \ \ \
n=\sum_{i=1}^\ell \mu_i, \ \ \ \phi_i \in \mathcal{B}\ . $$
Such a weighted partition determines a descendent
insertion
$$\tau[\phi]=\prod_{i=1}^\ell \tau_{\mu_i-1}(\iota_{\infty*}(\phi_i)) \ .$$

By standard $K3$ vanishing arguments \cite{mpt},
the stable pairs invariants of the relative 3-fold geometry
$\mathbf{P}(\cO_S \oplus N)/S_\infty$ are nontrivial only for
curves classes in the fibers of
\begin{equation*}
\mathbf{P}(\cO_S \oplus N)/S_\infty\rightarrow S \ .
\end{equation*}
Define the partition function for the relative geometry by
\begin{equation} \label{fq22}
\bZ_{\mathsf{P}}
\Big( \mathbf{P}(\cO_S \oplus N)/S_\infty
;q\ \Big|
\tau[\phi] \Big|\ \mu \Big)_{d\mathsf{L}}
\end{equation}
where $\phi$ and $\mu$ are both partitions of $d$ weighted
by $\mathcal{B}$. 
By further vanishing, only the leading $q^d$ terms of \eqref{fq22}
are possibly nonzero. The following result is proven in
Section 4.1 of \cite{PP2}.

\begin{Proposition}\label{dgbb5}
Let $d>0$ be an integer.
The square matrix indexed by $\mathcal{B}$-weighted partitions
of $d$ with coefficients
\begin{equation}\label{kkww3}
\bZ_{\mathsf{P}}
\Big( \mathbf{P}(\cO_S \oplus N)/S_\infty
;q\ \Big|
\tau[\phi] \Big|\ \mu \Big)_{d\mathsf{L}} 
\end{equation}
has maximal rank.
\end{Proposition}

We can also consider the Gromov-Witten analogue of 
Proposition \ref{dgbb5}.
By Proposition \ref{qq99}, we have a descendent correspondence
\begin{multline}\label{plgtt}
(-q)^{-d}\ZZ_{\mathsf{P}}\Big(\mathbf{P}(\cO_S \oplus N ) /S_\infty;q\ \Big|  
{\tau[\phi]} \Big| \ \mu
\Big)_{d\mathsf{L}} \\ =
(-iu)^{|\nu|+\ell(\nu)}\ZZ'_{\mathsf{GW}}\Big(\mathbf{P}(\cO_S \oplus N)/S_\infty;u
\ \Big|   
 \overline{\tau[\phi]}\ 
 \Big| \ \mu \Big)_{d\mathsf{L}} 
\end{multline}
after  the variable change $-q=e^{iu}$.
In particular, the Gromov-Witten matrix corresponding to \eqref{kkww3}
is also invertible.

Let $\beta \in H_2(X,\mathbb{Z})$ be a curve class, and let
$$d_\beta = \int_\beta c_1(X)=\int_\beta [S] \ .$$
Consider the descendent correspondence of Theorem \ref{zzz},
\begin{equation} \label{bll3}
(-q)^{-d_\beta/2}\,
\bZ_{\mathsf{P}}
\Big( X;q\ \Big|\ \tau[\phi] \ \Big)_\beta
=(-iu)^{d_\beta}\,  \bZ'_{\mathsf{GW}}
\Big( X;u\ \Big|\
\overline{\tau[\phi]} \  \Big )_\beta\, ,
\end{equation}
where $\phi$ is a partition of $d_\beta$ weighted
by $\mathcal{B}$. 
Since all the cohomology classes of the descendent $\tau[\phi]$
lie on $S$, we can degenerate to the normal cone.
The resulting degeneration formula{\footnote{We follow
here the notation of Section \ref{sdeg}.}} in stable pairs theory for
$
\bZ_{\mathsf{P}}
\left(  X   \Big|
\tau[\phi]\, \right)_\beta$ 
is
$$
\sum \bZ_{\mathsf{P}}
\left( \mathbf{P}(\cO_S \oplus N)/S_\infty    \Big|
\tau[\phi] \, \Big| \mu  \right)_{d_\beta\mathsf{L}}  \, 
(-1)^{|\mu|-\ell(\mu)} \,
\zz(\mu) \, q^{-|\mu|} \,
\bZ_{\mathsf{P}}
\left( X/S \Big|
\ \  \Big| \mu^\vee  \right)_{\beta} \,,
$$
where the sum is over all elements $\mu$ of
the Nakajima basis of cohomology of $\text{Hilb}(S,d_\beta)$.
The parallel degeneration 
formula for Gromov-Witten theory together
with Propositions \ref{qq99} and \ref{dgbb5}
imply Theorem \ref{mmpp44} in case there are {\em no
descendent insertions}.

Consider now the correspondence of Theorem
\ref{mmpp44} with the full descendent insertion
\begin{equation}\label{kwwk3}
\tau_0(\gamma_1) \ldots \tau_0(\gamma_r)\ .
\end{equation}
Since $X$ is a toric variety, the cohomological degree
of each $\gamma_i$ must be even. Degrees 0 and 2 can be
removed from both stable pairs and Gromov-Witten theory
by the fundamental class and divisor equations.
We need only consider insertions $\gamma_i$ of degree 4 or 6.
The divisor 
$$\iota_S: S\subset X$$ 
is ample since $X$ is Fano.
Hence, classes $\gamma_i\in H^*(X,\mathbb{Q})$ 
of degrees 4 and 6 can be written
as 
\begin{equation}\label{pdd4}
\iota_{S*}(\phi_i)= \gamma_i
\end{equation}
for $\phi_i \in H^*(S,\mathbb{Q})$ by Hard Lefschetz.
We can write the insertion \eqref{kwwk3} as
$$\tau_0(\iota_{S*}(\phi_1)) \ldots \tau_0(\iota_{S*}(\phi_r))\ .$$

We now reduce correspondence of Theorem \ref{mmpp44} with
the full insertion \eqref{kwwk3} to Theorem \ref{mmpp44} with no 
insertions. 
Via degeneration to the normal cone of $S$, we can write
$$
\bZ_{\mathsf{P}}
\left(  X   \Big| \tau_0(\iota_{S*}(\phi_1)) \ldots \tau_0(\iota_{S*}(\phi_r))
\, \right)_\beta$$
in terms of  the relative geometries as
\begin{multline*}
\sum \bZ_{\mathsf{P}}
\left( \mathbf{P}(\cO_S \oplus N)/S_\infty    \Big|
\tau_0(\iota_{S_0*}(\phi_1)) \ldots \tau_0(\iota_{S_0*}(\phi_r))
\, \Big| \mu  \right)_{d_\beta\mathsf{L}}  \, \\
(-1)^{|\mu|-\ell(\mu)} \,
\zz(\mu) \, q^{-|\mu|} \,
\bZ_{\mathsf{P}}
\left( X/S \Big|
\ \  \Big| \mu^\vee  \right)_{\beta} \,,
\end{multline*}
where the sum is as before. 
The parallel degeneration 
formula for Gromov-Witten theory together
with Proposition \ref{qq99}  achieves the desired
reduction.
\qed

\subsection{Proof of Corollary \ref{yaya34}}
Let $S\subset \mathbf{P}^3$ be a nonsingular 
quartic surface (anti-canonical and $K3$).
Let $\beta \in H_2(\mathbf{P}^3,\mathbb{Z})$. Since $c_1(T_{\mathbf{P}^3})$
is even, $d_\beta$ is even.
Then, by Theorem \ref{mmpp44}, we have
\begin{equation}\label{pcc4}
\bZ'_{\mathsf{GW}}
\left( \mathbf{P}^3/S\ \Big|
\ \  \Big| \mu^\vee  \right)_{\beta} \ \in 
\mathbb{Q}(-q=e^{iu},i)[u,\frac{1}{u}]\ 
\end{equation}
by the rationality in $q$ of the corresponding
stable pairs series \cite{PP2}.

Since the classes $\gamma_j \in H^*(\mathbf{P}^3,\mathbb{Q})$
are assumed to be of positive degree, we can write
$$\iota_{S*}(\phi_j) = \gamma_j$$
for classes $\phi_j \in H^*(S,\mathbb{Z})$.
After replacing the descendent insertion with
$$\tau_{k_1}(\iota_{S*}(\phi_1)) \ldots
 \tau_{k_s}(\iota_{S*}(\phi_s)),$$
we can degenerate to the normal cone of $S$. We find
$$\bZ'_{\mathsf{GW}}
\left( \mathbf{P}^3\ \Big|
\tau_{k_1}(\iota_{S*}(\phi_1)) \ldots
 \tau_{k_s}(\iota_{S*}(\phi_s))
\   \right)_{\beta}$$
is equal to 
\begin{multline}\label{k399}
\sum \bZ'_{\mathsf{GW}}
\left( \mathbf{P}(\cO_S \oplus N)/S_\infty    \Big|
\tau_{k_1}(\iota_{S_0*}(\phi_1)) \ldots \tau_{k_s}(\iota_{S_0*}(\phi_s))
\, \Big| \mu  \right)_{d_\beta\mathsf{L}}  \, \\
 \,
\zz(\mu) \, u^{2\ell(\mu)} \,
\bZ'_{\mathsf{GW}}
\left( \mathbf{P}^3/S \Big|
\ \  \Big| \mu^\vee  \right)_{\beta} \, .
\end{multline}
The terms of \eqref{k399} which are invariants of
$\mathbf{P}(\cO_S \oplus N)/S_\infty$ are Laurent polynomials
in $u$ and $\frac{1}{u}$ by $K3$ vanishings (the only connected
contributions are of genus 0 and
1). The terms with are invariants of $\mathbf{P}^3/S$
are constrained by \eqref{pcc4}. The claim of the
Corollary then follows immediately. \qed

\vspace{+16 pt}
\noindent Departement Mathematik \hfill Department of Mathematics \\
\noindent ETH Z\"urich \hfill  Princeton University \\
\noindent rahul@math.ethz.ch  \hfill rahulp@math.princeton.edu \\

\vspace{+8 pt}
\noindent
Department of Mathematics\\
Princeton University\\
apixton@math.princeton.edu


\begin{thebibliography}{MNOP2}




\bibitem{Beh} K.~Behrend, {\em Gromov-Witten invariants in algebraic geometry}, Invent. Math.
{\bf 127} (1997), 601--617.


\bibitem{BehFan}
K.~Behrend and B.~Fantechi,
\newblock {\em The intrinsic normal cone,} 
{Invent. Math.} {\bf 128} (1997), 45--88.










\bibitem{BP}
J.~Bryan and R.~Pandharipande.
\newblock {\em The local {G}romov-{W}itten theory of curves}, 
 JAMS {\bf 21} (2008), 101-136.





\bibitem{FP} C. Faber and R. Pandharipande, {\em Hodge integrals and
Gromov-Witten theory}, Invent. Math. {\bf 139} (2000), 173--199.



\bibitem{FKMO}
S.~Fujii, H.~Kanno, S.~Moriyama, S.~Okada,
{\em Instanton calculus and chiral one-point functions in
              supersymmetric gauge theories}, Adv. Theor. Math. Phys. {\bf 12} (2008), 1401--1428.


\bibitem{GraberP}
T.~Graber and R.~Pandharipande,
\newblock {\em Localization of virtual classes}, Invent. Math., {\bf 135},
  487--518, 1999.




\bibitem{grosss}
M. Gross and B. Siebert, {\em Logarithmic Gromov-Witten invariants},
arXiv:1102.4322. 



\bibitem{IP} 
E.~Ionel and T.~Parker, 
\emph{Relative Gromov-Witten invariants}, Ann. of Math {\bf 157} (2003), 45--96.






\bibitem{LR} 
A.-M.~Li and Y.~Ruan, 
\emph{Symplectic surgery and
Gromov-Witten invariants of Calabi-Yau 3-folds I},  
Invent.\ Math.\ \textbf{145} (2001), 151--218.





\bibitem{L}
J.~Li, 
\emph{
A degeneration formula of GW-invariants}, JDG {\bf 60} (2002), 199--293.





\bibitem{LiTian}
J.~Li and G.~Tian,
\newblock {\em 
Virtual moduli cycles and {G}romov-{W}itten invariants of algebraic
  varieties,} {JAMS}  {\bf 11}, 119--174, 1998.

\bibitem{lehn} M. Lehn, 
{\em Chern classes of tautological sheaves on Hilbert schemes 
of points on surfaces}, Invent. Math. {\bf 136} (1999), 157--207.


\bibitem{mo1}
D.~Maulik and A.~Oblomkov, {\em The quantum cohomology of
the Hilbert scheme of points of
$A_n$-resolutions}, JAMS {\bf 22} (2009),\\ 1055--1091.




\bibitem{mo2}
D.~Maulik and A.~Oblomkov, {\em Donaldson-Thomas theory of\\
$A_n\times {\mathbb{P}}^1$}, 
Comp. Math. {\bf 145} (2009), 1249--1276.



\bibitem{moop}
D.~Maulik, A. ~Oblomkov, A.~Okounkov, and R.~Pandharipande,
\newblock {\em The Gromov-{W}itten/{D}onaldson-{T}homas correspondence
for toric 3-folds}, Invent. Math. (to appear).


\bibitem{MNOP1}
D.~Maulik, N.~Nekrasov, A.~Okounkov, and R.~Pandharipande,
\newblock {\em Gromov-{W}itten theory and {D}onaldson-{T}homas theory. {I}},
  Compos. Math. {\bf 142} (2006), 1263--1285.


\bibitem{MNOP2}
D.~Maulik, N.~Nekrasov, A.~Okounkov, and R.~Pandharipande,
\newblock {\em Gromov-{W}itten theory and {D}onaldson-{T}homas theory. {II}},
  Compos. Math. {\bf 142} (2006), 1286--1304.

\bibitem{mptop}
D.~Maulik and R.~Pandharipande,
\newblock {\em A topological view of Gromov-Witten theory},
  Topology {\bf 45} (2006), 887--918.




\bibitem{mpt}
D. Maulik, R. Pandharipande, R. Thomas, {\em Curves on $K3$ surfaces
and modular forms}, 
With an appendix by A. Pixton. J. Topol. {\bf 3} (2010), 937--996.




\bibitem{oop} A. Oblomkov, A. Okounkov, and R. Pandharipande,
{\em in preparation}.


\bibitem{unknot}
A.~Okounkov and R.~Pandharipande,
\newblock {\em Hodge integrals and invariants of the unknot},
Geom. Topol. {\bf 8} (2004), 675--699.


\bibitem{opp1}
A.~Okounkov and R.~Pandharipande,
\newblock {\em Gromov-Witten theory, Hurwitz numbers, and completed
cycles},
Ann. of Math. {\bf 163} (2006), 517--560.


\bibitem{hilb1}
A.~Okounkov and R.~Pandharipande,
\newblock {\em The quantum cohomology of the Hilbert
scheme of points of the plane},
Invent. Math. {\bf 179} (2010), 523--557.



\bibitem{lcdt}
A.~Okounkov and R.~Pandharipande,
\newblock {\em The local Donaldson-Thomas theory of curves},
Geom. Topol.
{\bf 14} (2010), 1503--1567.


\bibitem{ORV}
A.~Okounkov, N.~ Reshetikhin, and C.~Vafa, {\em
Quantum Calabi-Yau and classical crystals}, in {\em The unity 
of mathematics}, 597–618, Progr. Math., 244, Birkhäuser Boston, Boston, MA, 
2006. 


\bibitem{part1}
R.~Pandharipande and A.~Pixton,
\newblock {\em Descendents on local curves: rationality}, 
arXiv:1011.4050.


\bibitem{PP2}
R.~Pandharipande and A.~Pixton,
\newblock {\em Descendent theory for stable pairs on
toric 3-folds}, 
arXiv:1011.4054.

\bibitem{PPstat}
R.~Pandharipande and A.~Pixton,
\newblock {\em Descendents on local curves: stationary theory}, 
arXiv:1109.1258.

\bibitem{PPcy3}
R.~Pandharipande and A.~Pixton,
\newblock {\em Gromov-Witten pairs correspondence for the quintic 3-fold},
arXiv:1203.0468.

\bibitem{pt}
R.~Pandharipande and R.~P. Thomas,
\newblock {\em Curve counting via stable pairs in the derived
category}, Invent Math. {\bf 178} (2009), 407 -- 447.


\bibitem{pt2}
R.~Pandharipande and R.~P. Thomas,
\newblock {\em The 3-fold vertex via stable pairs}, Geom. Topol.
{\bf 13} (2009), 1835--1876.

\bibitem{pt3}
R.~Pandharipande and R.~P. Thomas,
\newblock {\em Stable pairs and BPS invariants}, JAMS {\bf 23}
(2010), 267--297.









\end{thebibliography}
\end{document}